\documentclass[leqno,11pt]{amsart}
\usepackage{amsmath,amscd,amsthm,amsxtra}
\usepackage{epsfig,graphics,color,colortbl}
\usepackage{amssymb,latexsym}
\usepackage{mathabx}
\usepackage{mathrsfs,eucal,upgreek}
\usepackage[poly,all]{xy}
\usepackage{hyperref}
\usepackage{tikz-cd}
\usepackage{arydshln}
\usepackage{ytableau}
\usepackage{adjustbox}
\usepackage[draft]{todonotes}

\newtheorem{thm}{\bf Theorem}[section]
\newtheorem{df}[thm]{\bf Definition}
\newtheorem{prop}[thm]{\bf Proposition}
\newtheorem{cor}[thm]{\bf Corollary}
\newtheorem{lem}[thm]{\bf Lemma}
\newtheorem{rem}[thm]{\bf Remark}
\newtheorem{ex}[thm]{\bf Example}

\numberwithin{equation}{section}

\newcommand{\bs}{\boldsymbol}
\newcommand{\B}{\mathbf{B}}

\newcommand{\W}{\mathcal{W}}
\newcommand{\cP}{\mathscr{P}}
\newcommand{\pf}{\noindent{\bfseries Proof. }}
\newcommand{\ov}{\overline}

\newcommand{\ba}{\bs{\rm a}}
\newcommand{\bb}{\bs{\rm b}}
\newcommand{\bc}{\bs{\rm c}}

\newcommand{\bi}{\bs{\rm i}}
\newcommand{\bj}{\bs{\rm j}}

\newcommand{\gl}{\mathfrak{gl}}
\newcommand{\Z}{\mathbb{Z}}

\newcommand{\Q}{\mathbb{Q}}

\newcommand{\te}{\widetilde{e}}
\newcommand{\tf}{\widetilde{f}}
\newcommand{\g}{\mathfrak{g}}
\newcommand{\td}{\widetilde}

\newcommand{\mc}{\mathcal}
\newcommand{\mf}{\mathfrak}
\newcommand{\La}{\Lambda}

\newcommand{\la}{\lambda}
\newcommand{\ep}{\epsilon}

\newcommand{\nw}{^{\nwarrow}}
\newcommand{\se}{^{\searrow}}

\newcommand{\blue}[1]{{\color{blue}#1}}
\newcommand{\red}[1]{{\color{red}#1}}

\newcommand{\tl}[1]{\substack{\scalebox{0.75}{#1}}}
\newcommand{\Tl}[1]{\substack{\scalebox{0.95}{#1}}}

\begin{document}
\title[Crystals of quantum nilpotent subalgebras]
{Quantum nilpotent subalgebras of classical quantum groups and affine crystals}

\author{IL-SEUNG JANG}

\address[I.-S. Jang]{Department of Mathematical Sciences, Seoul National University, Seoul 08826, Korea}
\email{is\_jang@snu.ac.kr}
\curraddr{Department of Mathematics, Incheon National University, Incheon 22012, Republic of Korea (current email:~\texttt{ilseungjang@inu.ac.kr})}

\author{JAE-HOON KWON}

\address[J.-H. Kwon]{Department of Mathematical Sciences and RIM, Seoul National University, Seoul 08826, Korea}
\email{jaehoonkw@snu.ac.kr}


\thanks{This work was supported by Samsung Science and Technology Foundation under Project Number SSTF-BA1501-01.}

\begin{abstract}
We study the crystal of quantum nilpotent subalgebra of $U_q(D_n)$ associated to a maximal Levi subalgebra of type $A_{n-1}$. We show that it has an affine crystal structure of type $D_n^{(1)}$ isomorphic to a limit of perfect  Kirillov-Reshetikhin crystal $B^{n,s}$ for $s\geq 1$, and give a new polytope realization of $B^{n,s}$. We show that an analogue of RSK correspondence for type $D$ due to Burge is an isomorphism of affine crystals and give a generalization of Greene's formula for type $D$.

\end{abstract}

\maketitle
\setcounter{tocdepth}{1}

\section{Introduction}

Let $\mf g$ be a classical Lie algebra and let $\mf l$ be its proper maximal Levi subalgebra of type $A$ (or a sum of type $A$). Let $\mf u^-$ be the negative  nilradical of the parabolic subalgebra $\mf p=\mf l + \mf b$, where $\mf b$ is a Borel subalgebra of $\mf g$.
The enveloping algebra $U(\mf u^-)$ is an integrable $\mf l$-module, which has a multiplicity-free decomposition \cite{H}, and the expansion of its character
\begin{equation}
{\rm ch} \ U(\mf u^-) = {\prod_{\alpha\in \Phi({\mf u}^-)}(1-e^\alpha)^{-1}}
\end{equation}
into irreducible $\mf l$-characters (that is, Schur polynomials or a product of Schur polynomials) gives the celebrated Cauchy identity when $\mf g$ is of type $A$, and Littlewood identities when $\mf g$ is of type $B, C, D$, where $\Phi(\mf u^-)$ is the set of roots of $\mf u^-$.

The decomposition of $U(\mf u^-)$ into $\mf l$-modules has a purely combinatorial interpretation by RSK 
correspondence and its variations, say $\kappa$ (cf. \cite{B,Ful}). Indeed, the correspondence for type $B$ and $C$ is given by symmetrizing the map for type $A$.
A more representation theoretic meaning is available by crystal base theory \cite{Kas91}. In \cite{La}, Lascoux showed that $\kappa$ is an isomorphism of $\mf l$-crystals, which immediately implies the same result for type $B$ and $C$ \cite{K09} by using similarity of crystals \cite{Kas96}.

%
Furthermore, it is shown in \cite{K13} that the RSK correspondence $\kappa$ can be extended to an isomorphism of affine crystals of type $A_{n}^{(1)}$ when $\mf g$ is of type $A_{n}$, and of type $D_{n+1}^{(2)}$ and $C_n^{(1)}$ when $\mf g$ is of type $B_n$ and $C_n$, respectively.
It is done by regarding the set of biwords (or the set of matrices with non-negative integral entries) as the crystal $B(U_q(\mf u^-))$ of the quantum nilpotent subalgebra $U_q(\mf u^-)$ (see also \cite{K18,K18-2}). This approach enables us to define naturally the Kashiwara operators on both sides of the correspondence $\kappa$ for the simple roots other than the ones in $\mf l$.
Moreover it is proved that $B(U_q(\mf u^-))$ is isomorphic to a limit of perfect Kirillov-Reshetikhin crystals $B^{r,s}$, which are classically irreducible (cf. \cite{FOS}). An explicit description of $B^{r,s}$ as a subcrystal of $B(U_q(\mf u^-))$ is also given in \cite{K13}. We remark that a similar viewpoint appears in a recent study of affine geometric crystals of type $A$ \cite{MN}.

In this paper, we establish an analogue of the above result when $\g$ is of type $D$. The main difficulty for type $D$ is that the description of crystal $B(U_q(\mf u^-))$ and related combinatorics are more involved than in case of type $A$, $B$, $C$.
 To handle with it, we use a recent work by Salisbury-Schultze-Tingley \cite{SST} on crystal structures of Lusztig data of PBW basis \cite{Lu90}.

We consider the crystal $\B_{\bi_0}$ of $\bi_0$-Lusztig data, where $\bi_0$ is  a reduced expression associated to a specific convex order on the set of positive roots of $\mf g$. 
The subcrystal $B(U_q(\mf u^-))$ of $\B_{\bi_0}$ consisting of Lusztig data on $\Phi(\mf u^-)$ has a nice combinatorial realization, and naturally admits an affine crystal structure of type $D_n^{(1)}$ isomorphic to a limit of KR crystals $B^{n,s}$ for $s\geq 1$, where $B^{n,1}$ is the crystal of the spin representation as a classical crystal of type $D_n$.
We give an explicit description of $B^{n,s}\subset B(U_q(\mf u^-))$ in terms of double paths on $\Phi(\mf u^-)$, which yields a polytope realization of $B^{n,s}$ (Theorem \ref{thm:main-1}). 
%
%

We then consider an analogue of RSK correspondence for type $D$ due to Burge \cite{B}. We apply this map to $B(U_q(\mf u^-))$, which sends a Lusztig datum on $\Phi(\mf u^-)$ to a semistandard tableau with columns of even length. 
As a main result of this paper, we prove that it is an isomorphism of affine crystals of type $D_n^{(1)}$, where a suitable affine crystal structure is defined on the side of tableaux  (Theorem \ref{thm:isomorphism theorem}). Furthermore, we present an interesting formula for the shape of a semistandard tableau corresponding to a Lusztig datum on $\Phi(\mf u^-)$ in terms of non-intersecting double paths on $\Phi(\mf u^-)$ (Theorem \ref{thm:shape}). This formula can be viewed as an analogue of Greene's formula for the shape of a tableau corresponding to a biword under RSK given in terms of disjoint weakly decreasing subwords \cite{G} (see also Remark \ref{rem:Greene}).

We should remark that the formula of Berenstein-Zelevinsky \cite{BZ-1} for the transition matrix between string parametrization and Lusztig data  plays a crucial role when we characterize $B^{n,s}$ in $B(U_q(\mf u^-))$ and derive the type $D$ analogue of Greene's result in Burge's correspondence. A key observation is that the notion of trail and its combinatorics introduced in \cite{BZ-1} recovers the Greene's formula in  type $A$, and can be reformulated in terms of double paths in case of type $D$.

The paper is organized as follows. In Section \ref{sec:crystal}, we review necessary background including crystals of Lusztig data and a work by Salisbury-Schultze-Tingley. In Section \ref{sec:crystal of quantum nil}, we describe in detail the crystal $\B_{\bi_0}$ and $B(U_q(\mf u^-))$ when $\mf g$ is of type $D$. In Section \ref{sec:main}, we state the main results in this paper, whose complete proofs are given in Section \ref{sec:proof}.

\noindent
{\bf Acknowledgment}.\!
The authors would like to thank Akito Uruno for his very careful reading and helpful comments.

\section{Quantum groups and PBW crystals}\label{sec:crystal}

\subsection{Crystals}\label{subsec:crystal}
Let us give a brief review on crystals (see \cite{HK,Kas91,Kas95} for more details). Let $\Z_+$ denote the set of non-negative integers. Let $\g$ be the Kac-Moody algebra associated to a symmetrizable generalized Cartan matrix $A =(a_{ij})_{i,j\in I}$ indexed by a set $I$. Let $P^\vee$ be the dual weight lattice, $P = {\rm Hom}_\Z( P^\vee,\Z)$ the weight lattice, $\Pi^\vee=\{\,h_i\,|\,i\in I\,\}\subset P^\vee$ the set of simple coroots, and $\Pi=\{\,\alpha_i\,|\,i\in I\,\}\subset P$ the set of simple roots of $\g$ such that $\langle \alpha_j,h_i\rangle=a_{ij}$ for $i,j\in I$. Let $P^+$ be the set of integral dominant weights. 

For an indeterminate $q$, let $U_q(\g)$ be the quantized enveloping algebra of $\g$ generated by $e_i$, $f_i$, and $q^h$ for $i\in I$ and $h\in P^\vee$ over $\Q(q)$.  A {\it $\g$-crystal} (or simply a {\it crystal} if there is no confusion on $\g$) is a set $B$ together with the maps ${\rm wt} : B \rightarrow P$, $\varepsilon_i, \varphi_i: B \rightarrow \mathbb{Z}\cup\{-\infty\}$ and $\te_i, \tf_i: B \rightarrow B\cup\{{\bf 0}\}$ for $i\in I$ satisfying certain axioms. We denote by $B(\infty)$ the crystal associated to the negative part $U^-_q(\g)$ of $U_q(\g)$.
Let $\ast$ be the $\mathbb{Q}(q)$-linear anti-automorphism of $U_q(\g)$ such that $e_i^\ast =e_i$, $f_i^\ast=f_i$, and $(q^h)^\ast =q^{-h}$ for $i\in I$ and $h\in P$. Then $\ast$ induces a bijection on $B(\infty)$. For $i\in I$, we define $\te_i^\ast =\ast \circ \te_i \circ \ast$ and $\tf_i^\ast =\ast \circ \tf_i \circ \ast$ on $B(\infty)$. For $\Lambda\in P^+$, we denote by $B(\Lambda)$ the crystal associated to an irreducible highest weight $U_q(\g)$-module $V(\Lambda)$ with highest weight $\Lambda$. For $\mu\in P$, let $T_\mu=\{t_\mu\}$ be a crystal, where ${\rm wt}(t_\mu)=\mu$, and $\varphi_i(t_\mu)=-\infty$ for all $i\in I$.

A morphism
$\psi : B_1 \rightarrow B_2$ is called an embedding if it is injective, and in this case $B_1$ called a subcrystal of $B_2$. For crystals $B_1$ and $B_2$, the {\em tensor product $B_1 \otimes B_2$} is defined to be $B_1 \times B_2$ as a set with elements denoted by $b_1 \otimes b_2$, where
\begin{equation*} \label{eq:tensor_product_rule}
\begin{split}
	& {\rm wt}(b_1 \otimes b_2) = {\rm wt}(b_1) + {\rm wt}(b_2), \\
	& \varepsilon_{i}(b_{1} \otimes b_{2}) = \max\{ \varepsilon_{i}(b_{1}), \varepsilon_{i}(b_{2})-\langle {\rm wt}(b_{1}), h_i \rangle \}, \\
	& \varphi_{i}(b_{1} \otimes b_{2}) = \max\{ \varphi_{i}(b_{1})+\langle {\rm wt}(b_2), h_i \rangle, \varphi_{i}(b_{2}) \}, \\
	& \tilde{e}_{i}(b_{1} \otimes b_{2}) = \left\{ \begin{array}{cc} \tilde{e}_{i}b_{1} \otimes b_{2} & \textrm{if} \ \varphi_{i}(b_{1}) \ge \varepsilon_{i}(b_{2}), \\ b_{1} \otimes \tilde{e}_{i}b_{2} & \textrm{if} \ \varphi_{i}(b_{1}) < \varepsilon_{i}(b_{2}), \end{array} \right. \\
	& \tilde{f}_{i}(b_{1} \otimes b_{2}) = \left\{ \begin{array}{cc} \tilde{f}_{i}b_{1} \otimes b_{2} & \textrm{if} \ \varphi_{i}(b_{1}) > \epsilon_{i}(b_{2}), \\ b_{1} \otimes \tilde{f}_{i}b_{2} & \textrm{if} \ \varphi_{i}(b_{1}) \le \epsilon_{i}(b_{2}),\end{array} \right.
\end{split}	
\end{equation*} for $i \in I$. Here, we assume that $\textbf{0} \otimes b_{2} = b_{1} \otimes \textbf{0} = \textbf{0}$. Then $B_1 \otimes B_2$ is a crystal.

\subsection{PBW crystals}\label{sec:PBW crystal}
Suppose that $\g$ is of finite type. Let us briefly recall a PBW basis and the crystal of Lusztig data which is isomorphic to $B(\infty)$ (see \cite{Lu90, Lu90-2,S94}).
Let $W$ be the Weyl group of $\g$ generated by the simple reflection $s_i$ for $i\in I$.
Let $w_0$ be the longest element in $W$ of length $N$, and let $R(w_0)=\{\,\bi=(i_1,\ldots,i_N)\,|\,w_0=s_{i_1}\ldots s_{i_N}\,\}$ be the set of reduced expressions of $w_0$. 

For $\bi\in R(w_0)$,
\begin{equation}\label{eq:beta_k}
\Phi^+=\{\beta_1:=\alpha_{i_1}, \beta_2:=s_{i_1}(\alpha_{i_2}),  \ldots ,\beta_N:=s_{i_1}\cdots s_{i_{N-1}}(\alpha_{i_N})\}
\end{equation}
is the set of positive roots of $\g$. For $i\in I$, let $T_i$ be the $\mathbb{Q}(q)$-algebra automorphism of $U_q(\mathfrak{g})$, which is given as {$T''_{i,1}$} in \cite{Lu93}. For $1\leq k\leq N$, put 
$f_{\beta_k}:=T_{i_{1}}T_{i_2}\cdots T_{i_{k-1}}(f_{i_{k}})$,
and for ${\bf c}=(c_{\beta_1},\ldots,c_{\beta_N})\in\Z_+^N$, let
\begin{equation}\label{eq:PBW vector}
\begin{split}
b _{\bi}(\bf c)=&
f_{\beta_1}^{(c_{\beta_1})}f_{\beta_2}^{(c_{\beta_2})}\cdots f_{\beta_N}^{(c_{\beta_N})},
\end{split}
\end{equation}
where $f_{\beta_k}^{(c_{\beta_k})}$ is a divided power of $f_{\beta_k}$. Then the set $B_{\bf i}:=\{\,b _{\bi}({\bf c})\,|\,{\bf c}\in\Z_+^{N}\,\}$ is a $\mathbb{Q}(q)$-basis of $U^-_q(\g)$ called a {\it PBW basis}.


Let $A_0$ be the subring of $\mathbb{Q}(q)$ consisting of rational functions regular at $q=0$. The $A_0$-lattice $L(\infty)$ of $U^-_q(\g)$ generated by $B_{\bf i}$ is independent of the choice of $\bi$ and invariant under $\te_i$, $\tf_i$, and the induced crystal $\pi(B_{\bf i})$ under a canonical projection $\pi : L(\infty) \rightarrow L(\infty)/q L(\infty)$ is isomorphic to $B(\infty)$. We identify ${\bf B}_{\bi}:=\Z_+^N$ with a crystal $\pi(B_{\bi})$ under the map ${\bf c}\mapsto b_{\bi}({\bf c})$, and call ${\bf c}\in \B_{\bi}$ an {\it $\bi$-Lusztig datum}.

Let $w\in W$ be given with length $r$. One may assume that there exists $\bi =(i_1,\ldots,i_N) \in R(w_0)$ such that $w=s_{i_1}\cdots s_{i_r}$. The $\Q(q)$-subspace of $U^-_q(\mf g)$ spanned by $b_{\bi}({\bf c})$ for ${\bf c}\in \B_{\bi}$ with $c_k=0$ for $r+1\leq k\leq N$ is the $\Q(q)$-subalgebra of $U^-_q(\mf g)$ generated by $f_{\beta_k}$ for $1\leq k\leq r$. It is independent of the choice of the reduced expression of $w$. We call this subalgebra the {\em quantum nilpotent subalgebra associated to $w\in W$} and denote it by $U_q^-(w)$ (see for example, \cite{Ki} and references therein).

\subsection{Description of $\tf_i$}\label{subssec:tensor product rule on B}
Suppose that $\mf g$ is of finite type. 
Let $\bi\in R(w_0)$ be given. For $\beta\in \Phi^+$, we denote by ${\bf 1}_{\beta}$ the element in $\B_{\bi}$ where $c_\beta=1$ and $c_\gamma=0$ for $\gamma\in \Phi^+\setminus \{\beta\}$. The Kashiwara operators $\tf_i$ or $\tf_i^\ast$ on $\B_{\bi}$ for $i\in I$ is not easy to describe in general except 
\begin{equation} \label{eq:f}
\begin{split}
\tf_{i}{\bf c} &= (c_1+1,c_2,\ldots,c_N)= {\bf c}+{\bf 1}_{\alpha_i},\quad 
\text{when $\beta_1=\alpha_i$},\\
\tf^\ast_{i}{\bf c} &= (c_1,\ldots,c_{N-1},c_N+1)= {\bf c}+{\bf 1}_{\alpha_i},\quad \text{when $\beta_N=\alpha_i$},\\
\end{split}
\end{equation}
for ${\bf c}\in \B_{\bi}$ \cite{Lu93}.

Let us review the results in \cite{SST}, where it is shown that $\tf_i$ can be described more explicitly in terms of so-called signature rule under certain conditions on $\bi$ with respect to $i$. For simplicity, let us assume that $\mf g$ is simply laced.

Let $\sigma=(\sigma_{1},\sigma_2,\ldots,\sigma_s)$ be a sequence with $\sigma_{u}\in \{\,+\,,\,-\, , \ \cdot\ \}$. We replace a pair $(\sigma_{u},\sigma_{u'})=(+,-)$, where $u<u'$ and $\sigma_{u''}=\,\cdot\,$
for $u<u''<u'$, with $(\,\cdot\,,\,\cdot\,)$, and repeat this process as far as possible until we get a sequence with no $-$ placed to the right of $+$. We denote the resulting sequence by ${\sigma}^{\rm red}$. For another sequence $\tau=(\tau_1,\ldots,\tau_{t})$, we denote by $\sigma\cdot\tau$ the concatenation of $\sigma$ and $\tau$.

Recall that a total order $\prec$ on $\Phi^+$ is called {\em convex} 
if either $\gamma \prec \gamma' \prec \gamma''$ or $\gamma'' \prec \gamma' \prec \gamma$ whenever $\gamma'=\gamma+\gamma''$ for $\gamma, \gamma', \gamma''\in \Phi^+$. It is well-known that there exists a one-to-one correspondence between $R(w_0)$ and the set of convex orders on $\Phi^+$, where the convex order $\prec$ 
associated to $\bi=(i_1,\ldots,i_N)\in R(w_0)$ is given by 
\begin{equation} \label{eq:convex}
\beta_1\prec \beta_2\prec \ \ldots \prec \beta_N,
\end{equation} 
where $\beta_k$ is as in \eqref{eq:beta_k} \cite{Pap94}.

There exists a reduced expression $\bi'$ obtained from $\bi$ by a 3-term braid move 
$(i_k,i_{k+1},i_{k+2}) \rightarrow (i_{k+1},i_{k},i_{k+1})$ with $i_k=i_{k+2}$ if and only if 
$\{\,\beta_{k},\beta_{k+1},\beta_{k+2}\,\}$ forms the positive roots of type $A_2$,
where the corresponding convex order $\prec'$ is given by replacing $\beta_k\prec \beta_{k+1}\prec \beta_{k+2}$ with $\beta_{k+2}\prec' \beta_{k+1}\prec' \beta_{k}$. 
Also there exists a reduced expression $\bi'$ obtained from $\bi$ by a 2-term braid move $(i_k,i_{k+1})\rightarrow (i_{k+1},i_{k})$ if and only if 
$\beta_{k}$ and $\beta_{k+1}$ are orthogonal, where the associated convex ordering $\prec'$ is given by replacing $\beta_k\prec \beta_{k+1}$ with 
$\beta_{k+1}\prec'\beta_{k}$.

Given $i\in I$, suppose that $\bi$ is {\em simply braided for $i\in I$}, that is, if one can obtain $\bi'=(i'_1,\ldots ,i'_N)\in R(w_0)$ with $i'_1=i$ by applying a sequence of braid moves consisting of either a 2-term move or 3-term braid move  
$(\gamma,\gamma',\gamma'')\rightarrow (\gamma'',\gamma',\gamma)$ with $\gamma''=\alpha_i$.
Suppose that 
\begin{equation}\label{eq:triple in 3-term}
\Pi_s=\{\gamma_s,\gamma'_s,\gamma''_s\}
\end{equation}
is the triple of positive roots of type $A_2$ with 
$\gamma'_s=\gamma_s+\gamma''_s$ and $\gamma''_s=\alpha_i$ corresponding to the $s$-th 3-term braid move for $1\leq s\leq t$.

For ${\bf c}\in \B_{\bi}$, let
\begin{equation}\label{eq:sigma_i}
\sigma_i({\bf c})= 
(\underbrace{-\cdots -}_{c_{\gamma'_1}}\,\underbrace{+\cdots +}_{c_{\gamma_1}}
\ \cdots \
\underbrace{-\cdots -}_{c_{\gamma'_t}}\,\underbrace{+\cdots +}_{c_{\gamma_t}}
).
\end{equation}
Then we have the following description of $\tf_i$ on $\B_{\bi}$ \cite[Theorem 4.6]{SST}.
\begin{thm}\label{thm:signature rule} 
Let $\bi\in R(w_0)$ and $i\in I$. 
Suppose that $\bi$ is simply braided for $i$. 
Let ${\bf c}\in \B_{\bi}$ be given.
\begin{itemize}
\item[(1)] If there exists $+$ in $\sigma_i({\bf c})^{\rm red}$ and the leftmost $+$ appears in $c_{\gamma_s}$, then 
\begin{equation*}
\tf_i{\bf c} = {\bf c} - {\bf 1}_{\gamma_s} + {\bf 1}_{\gamma'_s}.
\end{equation*}
\item[(2)] If there exists no $+$ in $\sigma_i({\bf c})^{\rm red}$, then 
\begin{equation*}
\tf_i{\bf c} = {\bf c} + {\bf 1}_{\alpha_i}.
\end{equation*}

\end{itemize}
\end{thm}\qed

\section{Crystal of quantum nilpotent subalgebra}\label{sec:crystal of quantum nil}

\subsection{Crystal $\B_{\bi_0}$} \label{subsection:nilpotent_subalgebra}
From now on, we assume that $\g$ is of type $D_n$ ($n\geq 4$). 
We assume that the weight lattice is $P=\bigoplus_{i=1}^n\Z\epsilon_i$, where $\{\,\epsilon_i\,|\,1\leq i\leq n\,\}$ is an orthonormal basis with respect to a symmetric bilinear form $(\, ,\,)$, and 
the Dynkin diagram is
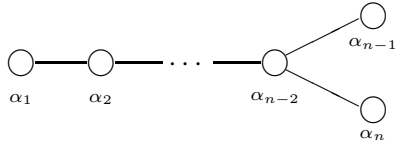
\begin{center} 
\setlength{\unitlength}{0.19in}
\begin{picture}(15,4.5)
\put(3.4,2){\makebox(0,0)[c]{$\bigcirc$}}
\put(5.6,2){\makebox(0,0)[c]{$\bigcirc$}}
\put(10.4,2){\makebox(0,0)[c]{$\bigcirc$}}
\put(13.1,3.3){\makebox(0,0)[c]{$\bigcirc$}}
\put(13.1,0.7){\makebox(0,0)[c]{$\bigcirc$}}
\put(3.8,2){\line(1,0){1.4}}
\put(6,2){\line(1,0){1.3}} 
\put(8.7,2){\line(1,0){1.3}} 
%
\put(10.7,2.2){\line(2,1){2}}
\put(10.7,1.8){\line(2,-1){2}}

\put(8,1.95){\makebox(0,0)[c]{$\cdots$}}
\put(3.4,1){\makebox(0,0)[c]{\tiny ${\alpha}_1$}}
\put(5.6,1){\makebox(0,0)[c]{\tiny ${\alpha}_2$}}
\put(10.4,1){\makebox(0,0)[c]{\tiny ${\alpha}_{n-2}$}}
\put(13.1,2.5){\makebox(0,0)[c]{\tiny ${\alpha}_{n-1}$}}
\put(13.1,0.0){\makebox(0,0)[c]{\tiny ${\alpha}_{n}$}}

\end{picture}
\end{center} 
where $\alpha_i=\ep_i-\ep_{i+1}$ for $1\leq i\leq n-1$, and $\alpha_n=\ep_{n-1}+\ep_n$. The set of positive roots is $\Phi^+=\{\,\ep_i\pm\ep_j\,|\,1\leq i<j\leq n\,\}$. 
Recall that $W$ acts faithfully on $P$ by $s_i(\ep_i)=\ep_{i+1}$, $s_i(\ep_k)=\ep_k$ for $1\leq i\leq n-1$ and $k\neq i, i+1$, and $s_n(\ep_{n-1})=-\ep_n$ and $s_n(\ep_k)=\ep_k$ for $k\neq n-1, n$.
The fundamental weights are $\varpi_i=\sum_{k=1}^i\epsilon_k$ for $i=1,\ldots, n-2$, $\varpi_{n-1}=(\epsilon_1+\cdots+\epsilon_{n-1}-\epsilon_n)/2$ and $\varpi_{n}=(\epsilon_1+\cdots+\epsilon_{n-1}+\epsilon_n)/2$. 

Put $J=I\setminus \{n\}$. 
Let $\mf l$ be the Levi subalgebra of $\mf g$ associated to $\{\,\alpha_i\,|\, i\in J\,\}$ of type $A_{n-1}$. Then 
\begin{equation*}
\Phi^+=\Phi^+(J)\cup \Phi^+_J,
\end{equation*}
where $\Phi^+_J=\{\,\ep_i-\ep_j\,|\,1\leq i<j\leq n\,\}$ is the set of positive roots of $\mf l$ and $\Phi^+(J)=\{\,\ep_i+\ep_j\,|\,1\leq i<j\leq n\,\}$ is the set of roots of the nilradical $\mf u$ of the parabolic subalgebra of $\mf g$ associated to ${\mf l}$.

Throughout this paper, we consider a specific $\bi_0\in R(w_0)$,
whose associated convex order on $\Phi^+$ is given by
\begin{equation}\label{eq:convex order i_0}
\begin{split}
&\ep_i+\ep_j \prec \ep_k-\ep_l,\\
\ep_i+\ep_j \prec \ep_k+& \ep_l \quad \Longleftrightarrow \quad 
\text{$(j>l)$ or $(j=l$, $i>k)$}, \\
\ep_i-\ep_j \prec \ep_k-& \ep_l \quad \Longleftrightarrow \quad
\text{$(i<k)$ or $(i=k$, $j<l)$},
\end{split}
\end{equation}
for $1\leq i<j\leq n$ and $1\leq k<l\leq n$.
An explicit form of $\bi_0$ is as follows.
For $1\leq k\leq n-1$, put
{\allowdisplaybreaks
\begin{align*}
\bi_k & = 
\begin{cases}
(n,n-2,\ldots,k+1,k), & \text{if $k$ is odd},\\
(n-1,n-2,\ldots,k+1,k), & \text{if $k$ is even},\\
(n), & \text{if $n$ is even and $k=n-1$},
\end{cases}
\\
\bi'_k & =
\begin{cases}
(n-1,n-2,\ldots, k+1,k), & \text{if $n$ is even and $1\leq k\leq n-1$},\\
(n,n-2,\ldots, k+1,k), & \text{if $n$ is odd and $1\leq k\leq n-2$},\\
(n), & \text{if $n$ is odd and $k=n-1$}.
\end{cases}
\end{align*}}
Let 
$\bi^J = \bi_1 \cdot \bi_2 \cdot \cdots \cdot \bi_{n-1}$ and 
$\bi_J = \bi'_1 \cdot \bi'_2 \cdots \cdot \bi'_{n-1}$. Then
\begin{equation} \label{eq:rx}
\bi_0 = \bi^J\cdot \bi_J,
\end{equation}
where $\bi\cdot \bj$ denotes the concatenation of $\bi\in I^r$ and $\bj \in I^s$.
We write $\bi_0=(i_1,\ldots,i_N)$,
where $i_1=n$, and put $\bi^J=(i_1,\ldots,i_M)$, and $\bi_J=(i_{M+1},\ldots,i_N)$ with $N=n^2-n$ and  $M=N/2$.

\begin{ex}\label{ex:i_0}
{\rm 
We have 
\begin{equation*}
\begin{split}
&\bi^J=(4,2,1,3,2,4),\quad\quad\quad\quad\ \ \bi_J=(3,2,1,3,2,3),\quad\quad\quad\quad\quad  \text{when $n=4$},\\
&\bi^J=(5,3,2,1,4,3,2,5,3,4),\quad \bi_J=(5,3,2,1,5,3,2,5,3,5),\, \quad \text{when $n=5$}.
\end{split}
\end{equation*}
The associated convex order when $n=4$ is
\begin{equation*}
\begin{split}
&\ep_3+\ep_4\prec
\ep_2+\ep_4\prec
\ep_1+\ep_4\prec
\ep_2+\ep_3\prec
\ep_1+\ep_3\prec
\ep_1+\ep_2 \\ 
& \prec \ep_1-\ep_2
\prec \ep_1-\ep_3
\prec \ep_1-\ep_4
\prec \ep_2-\ep_3
\prec \ep_2-\ep_4
\prec \ep_3-\ep_4.
\end{split}
\end{equation*}
}
\end{ex}

Throughout the paper, we set
$$
\B:=\B_{\bi_0}.
$$
For ${\bf c}=(c_\beta)\in \B$, we also write 
\begin{equation*}
c_{\beta_k} = 
\begin{cases}
c_{\ov{j}\ov{i}}, & \text{if $\beta_k=\ep_i+\ep_j$ for $1\leq i<j\leq n$}, \\
c_{j\ov{i}} , & \text{if $\beta_k=\ep_i-\ep_j$ for $1\leq i<j\leq n$}.
\end{cases}
\end{equation*}

\begin{prop}\label{prop:signature for type D}
For $i \in I \setminus \{ n \}$, there exists a reduced expression $\bi \in R(w_0)$, which is equal to $\bi_0$ up to $2$-term braid moves, such that 
\begin{itemize}
	\item[(1)] $\bi$ is simply braided for $i$,
	\item[(2)] for ${\bf c} \in {\bf B}_{\bi}$,  $\sigma_i({\bf c})$ \eqref{eq:sigma_i} is given by
\end{itemize}
\begin{equation*}
\begin{split}
\sigma_i({\bf c})=& 
\sigma_{i,1}({\bf c})\cdot \sigma_{i,2}({\bf c})\cdot \sigma_{i,3}({\bf c}),
\end{split}
\end{equation*}
where
\begin{equation} \label{eq:sigma}
\begin{split}
\sigma_{i,1}({\bf c})=&
(
\underbrace{-\cdots -}_{c_{\ov{n}\,\ov{i}}}\,\underbrace{+\cdots +}_{c_{\ov{n}\,\ov{i+1}}}
\underbrace{-\cdots -}_{c_{\ov{n-1}\,\ov{i}}}\,\underbrace{+\cdots +}_{c_{\ov{n-1}\,\ov{i+1}}}
\ \cdots \
\underbrace{-\cdots -}_{c_{\ov{i+2}\,\ov{i}}}\,\underbrace{+\cdots +}_{c_{\ov{i+2}\,\ov{i}}}), \\
\sigma_{i,2}({\bf c})=& (\underbrace{-\cdots -}_{c_{\ov{i}\,\ov{i-1}}}\,\underbrace{+\cdots +}_{c_{\ov{i+1}\,\ov{i-1}}}
\underbrace{-\cdots -}_{c_{\ov{i}\,\ov{i-2}}}\,\underbrace{+\cdots +}_{c_{\ov{i+1}\,\ov{i-2}}}
\ \cdots \
\underbrace{-\cdots -}_{c_{\ov{i}\,\ov{1}}}\,\underbrace{+\cdots +}_{c_{\ov{i+1}\,\ov{1}}}
), \\
\sigma_{i,3}({\bf c})=& (\underbrace{-\cdots -}_{c_{{i+1}\,\ov{1}}}\,\underbrace{+\cdots +}_{c_{{i}\,\ov{1}}} \underbrace{-\cdots -}_{c_{{i+1}\,\ov{2}}}\,\underbrace{+\cdots +}_{c_{{i}\,\ov{2}}}\ \cdots \ \underbrace{-\cdots -}_{c_{{i+1}\,\ov{i-1}}}\,\underbrace{+\cdots +}_{c_{{i}\,\ov{i-1}}} \underbrace{-\cdots -}_{c_{{i+1}\,\ov{i}}}).
\end{split}
\end{equation}
Here we assume that $c_{ab}$ is zero when it is not defined.
\end{prop}
\pf
We assume that $n$ is even since the proof for $n$ odd is almost the same. We fix $i\in I\setminus\{n\}$.

{\em Step 1}. We first observe that if the first letter $n-1$ in $\bi'_i$ corresponds to $i_k$ in $\bi_0$ for some $k$, then $\beta_k=\alpha_i$.

{\em Step 2}. Let $i_k=n-1$ be as in {\em Step 1}. Suppose that $i\neq 1$.
Then we can apply $2$-term move or $3$-term braid move $(\gamma,\gamma',\gamma'')\rightarrow (\gamma'',\gamma',\gamma)$ with $\gamma''=\alpha_i$ to $\bi_0$ 
(indeed to the subword $\bi'_1\cdot\bi'_2\cdots\bi'_{i-1}\cdot i_k$)
to get $\bi^{(3)}_0=(i_1,\ldots,i_M,j,\ldots)$ with $j=n-i$. 
We can check that $3$-term braid move occurs once in each $\bi'_s$ for $s=1,\ldots,i-1$, and the positive roots of the corresponding root system of type $A_2$ is 
\begin{equation} \label{eq:A_2-3}
\Pi_s^{(3)} = \{ \ep_{s}-\ep_{i}, \ep_{s}-\ep_{i+1}, \alpha_i \}
\end{equation} for $s = 1, \cdots, i-1$. We assume that $\Pi_s^{(3)}$ is empty when $i=1$.

{\em Step 3}. We consider the reduced word $\bi^{(3)}_0$. 
Suppose that $i\neq 1$ or $j\neq n$. 
First, we apply 2-moves only to $\bi_j\cdot\bi_{j+1}$ so that the last $i-2$ letters in $\bi_j$ and the first $i-2$ letters are shuffled by a permutation of length $(i-2)(i-1)/2$ and hence appear in an alternative way. We denote this subword by $\ov{\bi}_j\cdot\ov{\bi}_{j+1}$.

Then we apply $2$-term move or $3$-term braid move $(\gamma,\gamma',\gamma'')\rightarrow (\gamma'',\gamma',\gamma)$ with $\gamma''=\alpha_i$
to the subword $\ov{\bi}_j\cdot\ov{\bi}_{j+1}\cdot\bi_{j+2} \cdots \bi_{n-1}\cdot j$ 
to obtain a word starting with $j''$ where $j''$ is $n$ (resp. $n-1$) when $j$ is odd (resp. even).
We denote the resulting whole word by $\bi_0^{(2)}$.

Here we have $i-1$ $3$-term braid moves only in $\ov{\bi}_j\cdot\ov{\bi}_{j+1}$
and the positive roots of the corresponding root system of type $A_2$ is 
\begin{equation} \label{eq:A_2-2}
\Pi_s^{(2)} = \{ \ep_{i+1}+\ep_{s}, \ep_{i}+\ep_{s}, \alpha_i \}
\end{equation} for $1\leq s\leq i-1$, and the order of occurrence of 3-term braid move is when $s$ ranges from 1 to $i-1$.
If $i=1$, then we assume that $\Pi_s^{(2)}$ is empty, and $\bi_0^{(2)}=\bi_0^{(3)}$.

{\em Step 4}. Finally, we apply $2$-term move or $3$-term braid move $(\gamma,\gamma',\gamma'')\rightarrow (\gamma'',\gamma',\gamma)$ with $\gamma''=\alpha_i$
to the subword ${\bi}_1\cdots {\bi}_{j}\cdot j''$ of $\bi_0^{(2)}$ to obtain a word starting with $i$, and denote the resulting whole word by $\bi_0^{(1)}$.
In this case, $3$-term braid move occurs once in each $\bi_s$ for $s=1,\ldots,j-1$, and the positive roots of the corresponding root system of type $A_2$ is 
\begin{equation} \label{eq:A_2-1}
\Pi_s^{(1)} = \{ \ep_{n-s+1}+\ep_{i+1}, \ep_{n-s+1}+\ep_{i}, \alpha_i \}
\end{equation} for $s = 1, \cdots, j-1$.

By the above steps, we conclude that $\bi^{(1)}_0\in R(w_0)$ is obtained from $\bi_0$ by applying $2$-term move or $3$-term braid move $(\gamma,\gamma',\gamma'')\rightarrow (\gamma'',\gamma',\gamma)$ with $\gamma''=\alpha_i$. 
We define 
\begin{equation*}
\bi := \ov{\bi}^J \cdot \bi_J,
\end{equation*} where $\ov{\bi}^J = \bi_1 \cdot \cdots \cdot \ov{\bi}_j \cdot \ov{\bi}_{j+1} \cdot \cdots \bi_{n-1}$ (cf. \eqref{eq:rx}) and $\ov{\bi}_j \cdot \ov{\bi}_{j+1}$ is obtained from $\bi_j \cdot \bi_{j+1}$ in {\em Step 3}. By {\em Step 1}--{\em Step 4}, the reduced expression $\bi$ is simply braided for $i$.
It follows from Theorem \ref{thm:signature rule} that the sequence $\sigma_i({\bf c})$ in \eqref{eq:sigma_i} for ${\bf c}\in \B_{\bi}$ is given by \eqref{eq:sigma}, where the positive roots of the root systems of type $A_2$ associated to $\sigma_{i,j}(\bc)$ are given by $\Pi_s^{(j)}$ for $j=1,2,3$ in \eqref{eq:A_2-1}, \eqref{eq:A_2-2}, and \eqref{eq:A_2-3}.
\qed

\begin{rem}
{\rm For $i \in I \setminus \{ n \}$, the crystal operator $\tf_i$ on $\B$ may be understood by 
\begin{equation*} \label{eq:ith crystal operator on B}
\xymatrixcolsep{3pc}\xymatrixrowsep{0pc}\xymatrix{
\B  \ \ar@{->}[r]^{R_{\bi_0}^{\bi}} & \ \B_{\bi} \ \ar@{->}[r]^{\tf_i} & \  \B_{\bi} \ \ar@{->}[r]^{R_{\bi}^{\bi_0}} & \ \B},
\end{equation*} where $R_{\bi_0}^{\bi}$ (resp. $R^{\bi_0}_{\bi}$) is the transition map from $\B$ to $\B_{\bi}$ (resp. from $\B_{\bi}$ to $\B$) (cf. \cite{Lu90}). The map $R_{\bi_0}^{\bi}$ corresponds to the $2$-term braid moves from $\bi_j \cdot \bi_{j+1}$ to $\ov{\bi}_j\cdot\ov{\bi}_{j+1}$ (see {\em Step 3} in the proof of Proposition \ref{prop:signature for type D}), which is simply given by exchanging the multiplicities related to them, and the map $R_{\bi}^{\bi_0}$ is the inverse of it. Therefore, the crystal operator $\tf_i$ on $\B$ can be described in the same way as in Theorem \ref{thm:signature rule} with $\sigma_i({\bf c})$ in Proposition \ref{prop:signature for type D}. }
\end{rem}

\begin{ex} \label{ex:cal_sig_rank5}
{\rm Let us illustrate $\sigma_i({\bf c})$ for ${\bf c} \in \B$ when $n=5$ and $i=3$.  
Consider $\bi_0 = \bi^J \cdot \bi_J=(i_1,\ldots,i_{20})$ (see Example \ref{ex:i_0}).
Note that $i_{18}=5$ is the first letter in $\bi'_3$, and $\beta_{18}=\alpha_3$.

For convenience, let ${\rightsquigarrow}$ (resp. ${\longrightarrow}$) mean the 3-term (resp. 2-term) braid move. Then
\begin{equation*}
\begin{split}
\bi_J=(5,3,2,1,5,3,2,{\bf 5},3,5) 
&\ {\longrightarrow}\ (5,3,2,1,5,3,{\bf 5},2,3,5) \\
&\ {\rightsquigarrow}\ (5,3,2,1,{\bf 3},5,3,2,3,5) \quad \cdots \quad \Pi_2^{(3)}\\
&\ {\longrightarrow}\ (5,3,2,{\bf 3},1,5,3,2,3,5) \\
&\ {\rightsquigarrow}\ (5,{\bf 2},3,2,1,5,3,2,3,5) \quad \cdots \quad  \Pi_1^{(3)}\\
&\ {\longrightarrow}\ ({\bf 2},5,3,2,1,5,3,2,3,5),
\end{split}
\end{equation*} 
where 
$\Pi_2^{(3)} = \{ \ep_{2}-\ep_{3}, \ep_{2}-\ep_{4}, \alpha_{3} \}$ 
and 
$\Pi_1^{(3)} = \{ \ep_{1}-\ep_{3}, \ep_{1}-\ep_{4}, \alpha_{3} \}$.
Here the bold letter denotes the one corresponding to $\alpha_3$ in the associated convex order on $\Phi^+$.

Next, we have 
$\bi_2\cdot \bi_3 = (4,3,{\bf 2}, {\bf 5},3)\ {\longrightarrow}\ \ov{\bi}_2\cdot\ov{\bi}_3=(4,3, {\bf 5}, {\bf 2},3)$, and hence
\begin{equation*}
\begin{split}
\ov{\bi}_2\cdot\ov{\bi}_3\cdot {\bf 2} = (4,3,5,2,3, \bf{2}) 
& \ {\rightsquigarrow} \ (4,3,5,{\bf 3},2,3)\quad \cdots \quad \Pi_2^{(2)}  \\
& \ {\rightsquigarrow} \ (4,{\bf 5},3,5,2,3)\quad \cdots \quad \Pi_1^{(2)} \\ 
& \ {\longrightarrow} \ ({\bf 5},4,3,5,2,3),
\end{split}
\end{equation*} 
where $\Pi_2^{(2)} = \{ \ep_{1}+\ep_{4}, \ep_{1}+\ep_{3}, \alpha_{3} \}$ and $\Pi_1^{(2)} = \{  \ep_{2}+\ep_{4}, \ep_{2}+\ep_{3}, \alpha_{3} \}$.
Finally,
\begin{equation*}
\begin{split}
\bi_1\cdot {\bf 5} = (5,3,2,1, {\bf 5}) & \ {\longrightarrow} \ (5,3,2,{\bf 5},1) \\
& \ {\longrightarrow}\ (5,3,{\bf 5},2,1) \\
& \ {\rightsquigarrow}\ ({\bf 3},5,3,2,1) \quad \cdots \quad \Pi_1^{(1)},
\end{split}
\end{equation*} where $\Pi_1^{(3)} = \{ \ep_{4}+\ep_{5}, \ep_{3}+\ep_{5}, \alpha_{3} \}$.
Thus $\bi_0$ is simply braided for $i=3$, and 
\begin{equation*}
\begin{split}
\sigma_{3,1}({\bf c})&=
(\underbrace{\Tl{$-\cdots -$}}_{c_{\ov{5}\ov{3}}}\,\underbrace{\Tl{$+\cdots +$}}_{c_{\ov{5}\ov{4}}}),\\
\sigma_{3,2}({\bf c})&=
(\underbrace{\Tl{$-\cdots -$}}_{c_{\ov{3}\ov{2}}}\,\underbrace{\Tl{$+\cdots +$}}_{c_{\ov{4}\ov{2}}}
\underbrace{\Tl{$-\cdots -$}}_{c_{\ov{3}\ov{1}}}\,\underbrace{\Tl{$+\cdots +$}}_{c_{\ov{4}\ov{1}}}),\\
\sigma_{3,3}({\bf c})&=
(\underbrace{\Tl{$-\cdots -$}}_{c_{4\ov{1}}}\,\underbrace{\Tl{$+\cdots +$}}_{c_{3\ov{1}}}
\underbrace{\Tl{$-\cdots -$}}_{c_{4\ov{2}}}\,\underbrace{\Tl{$+\cdots +$}}_{c_{3\ov{2}}}
\underbrace{\Tl{$-\cdots -$}}_{c_{4\ov{3}}}),
\end{split}
\end{equation*}
for ${\bf c}\in \B$.
}
\end{ex}

Set
\begin{equation}
\begin{split}
\B^J &=\left\{\,{\bf c}=(c_{\beta})\in \B\,\big\vert \,c_{\beta}=0 \text{ unless $\beta\in \Phi^{+}(J)$}\,\right\},\\
\B_J &=\left\{\,{\bf c}=(c_{\beta})\in \B\,\big\vert \,c_{\beta}=0 \text{ unless $\beta\in \Phi_J$}\,\right\}. 
\end{split}
\end{equation}
which we regard them as subcrystals of $\B$, where we assume that 
$\te_n{\bf c}=\tf_n{\bf c}={\bf 0}$ with  
$\varepsilon_n({\bf c})=\varphi_n({\bf c})=-\infty$ for ${\bf c}\in \B_J$.
The subscrystal $\B^J$ is the crystal of the quantum nilpotent subalgebra $U^-_q(w^J)$, where $w^J=s_{i_1}\cdots s_{i_M}$ with $\bi^J=(i_1,\ldots,i_M)$, which can be viewed as a $q$-deformation of $U(\mf u^-)$.

\begin{cor} \label{tensor_decomposition}
\mbox{}
\begin{itemize}
\item[(1)] The crystal $\B_J$ is isomorphic to the crystal of $U^-_q(\mf l)$ as an ${\mf l}$-crystal.

\item[(2)] The map 
\begin{equation*}
\xymatrixcolsep{3pc}\xymatrixrowsep{0pc}\xymatrix{
\B \ar@{->}[r]  & \ \B^J \,\otimes \B_J \\
{\bf c}  \ar@{|->}[r] &   {\bf c}^J\otimes{\bf c}_{J}}
\end{equation*}
is an  isomorphism of $\mf g$-crystals.
\end{itemize}
\end{cor}
\pf (1) It follows directly from comparing the crystal structure of $U_q^-(\mf l)$ given in  \cite[Section 4.1]{S12} (see also \cite[Section 4.2]{K18}).

(2) It follows from Theorem \ref{thm:signature rule}, Proposition \ref{prop:signature for type D}, and the tensor product rule of crystals.
\qed

\subsection{Crystal $\B^J$ of quantum nilpotent subalgebra} \label{section:Delta_n}
Let us consider the subcrystal $\B^J$ in more details.  
Let $\Delta_n$ be the arrangements of dots in the plane to represent the $(n-1)$-th triangular number. 
We often identify $\Delta_n$ with $\Phi^+(J)$ in such a way that $\ep_{k+1}+\ep_{l+1}$, $\ep_{k+1}+\ep_{l}$ and $\ep_{k}+\ep_{l}$ for $1\leq k,l\leq n-1$ are the vertices of a triangle of minimal shape in $\Delta_n$ as follows:
\begin{equation}\label{eq:small triangle}
\xymatrixcolsep{-1.5pc}\xymatrixrowsep{0.5pc}\xymatrix{
& &    & & \stackrel{{\ep_k+\ep_{l+1}}}{\bullet} & & &  & \\
&  & & \stackrel{{\ep_{k+1}+\ep_{l+1}}}{\bullet} & &  \ {}\stackrel{{\ep_k+\ep_l}}{\bullet}  }
\end{equation}
We also identify ${\bf c}\in \B^J$ with an array of $c_{\beta}$'s in ${\bf c}$ with $c_\beta$ at the corresponding dot in $\Delta_n$.

\begin{ex}\label{ex:Delta_n}
{\rm For $n=5$ and $\bc\in \B^J$, we have
\begin{equation*} 
\xymatrixcolsep{-0.9pc}\xymatrixrowsep{0.0pc}
\xymatrix{
& & & & \stackrel{{}_{\ep_1+\ep_5}}{\bullet} & & &  & \\
\Delta_5=\quad & & & \stackrel{{}_{\ep_2+\ep_5}}{\bullet} & &  {}\stackrel{{}_{\ep_1+\ep_4}}{\bullet} & & \\
& & \stackrel{{}_{\ep_3+\ep_5}}{\bullet}  & & \stackrel{{}_{\ep_2+\ep_4}}{\bullet} &  & \stackrel{{}_{\ep_1+\ep_3}}{\bullet}  & \\  
& \stackrel{{}_{\ep_4+\ep_5}}{\bullet} & & \stackrel{{}_{\ep_3+\ep_4}}{\bullet} & & \stackrel{{}_{\ep_2+\ep_3}}{\bullet} & & \stackrel{{}_{\ep_1+\ep_2}}{\bullet}}
\quad\quad
\xymatrixcolsep{-0.2pc}\xymatrixrowsep{0.3pc}\xymatrix{
& & & & c_{\ov{5}\ov{1}} & & & & \\
\bc= & & & c_{\ov{5}\ov{2}} & & c_{\ov{4}\ov{1}} & & \\
& & c_{\ov{5}\ov{3}} & & c_{\ov{4}\ov{2}} & & c_{\ov{3}\ov{1}} & \\  
& c_{\ov{5}\ov{4}} & & c_{\ov{4}\ov{3}} & & c_{\ov{3}\ov{2}} & & c_{\ov{2}\ov{1}}}
\end{equation*}\vskip 3mm
}
\end{ex}

\begin{lem}\label{lem:B^J}
We have
$\B^J=\{\,{\bf c}\,|\,\varepsilon_i^*({\bf c})=0\ (i\in J)\,\}$.
\end{lem} 
\pf It follows from  \cite[Section 2.1]{Lu90} and \eqref{eq:f}.
\qed

For $s\geq 1$, let 
\begin{equation}\label{eq:B^J_s-1}
\B^{J,s}:=\{\,{\bf c}\in \B^J\,|\,\varepsilon_n^*({\bf c})\leq s\,\},
\end{equation} 
which is a subcrystal of $\B^J$.
By Lemma \ref{lem:B^J} and \cite[Proposition 8.2]{Kas95} (cf. \cite{Kas93}), we have 
\begin{equation}\label{eq:B^J_s-2}
B(s\varpi_n)\cong \B^{J,s}\otimes T_{s\varpi_n},\quad\quad \bigcup_{s\geq 1}\B^{J,s}=\B^J,
\end{equation}
as $\mf g$-crystals.

By \eqref{eq:B^J_s-2}, $\B^J$ is a regular $\mf l$-crystal, that is, any connected component with respect to $\te_i$ and $\tf_i$ for $i\in J$ is isomorphic to the crystal of an integrable highest weight $U_q(\mf l)$-module, say $B_J(\la)$ for some $\la =\sum_{i=1}^n\la_i\ep_i\in P$ with $\langle\la, h_i\rangle\ge 0$ for $i\in J$.

\begin{ex}{\rm 
By Proposition \ref{prop:signature for type D}, one can easily describe $\te_i$ and $\tf_i$ for $i\in J$ based on tensor product rule.
When $n=5$ and $i=3$, we have for ${\bf c}\in\B^J$,
\begin{equation*} 
\sigma_{3}({\bf c})=
(\underbrace{-\cdots -}_{c_{\ov{5}\ov{3}}}\,\underbrace{+\cdots +}_{c_{\ov{5}\ov{4}}}
\underbrace{-\cdots -}_{c_{\ov{3}\ov{2}}}\,\underbrace{+\cdots +}_{c_{\ov{4}\ov{2}}}
\underbrace{-\cdots -}_{c_{\ov{3}\ov{1}}}\,\underbrace{+\cdots +}_{c_{\ov{4}\ov{1}}}),
\end{equation*} 
(see Example \ref{ex:cal_sig_rank5}) and the multiplicities $c_{ab}$ involved in $\sigma_{3}({\bf c})$ are indicated below in $\Delta_5$
\begin{equation*} 
\xymatrixcolsep{-0.2pc}\xymatrixrowsep{0.3pc}\xymatrix{
& & & & \bullet & & & & \\
& & & \bullet & & c_{\ov{4}\ov{1}} & & \\
& & c_{\ov{5}\ov{3}} & & c_{\ov{4}\ov{2}} & & c_{\ov{3}\ov{1}} & \\  
& c_{\ov{5}\ov{4}} & & \bullet & & c_{\ov{3}\ov{2}} & & \bullet}
\end{equation*}\vskip 3mm
}
\end{ex}

Proposition \ref{prop:signature for type D} enables us to decompose $\B^J$ into $\mf l$-crystals directly as follows, and hence the decomposition of $U_q(w^J)$ into irreducible $U_q(\mf l)$-modules.
\begin{prop}\label{prop:decomp of B^J}
As an $\mf l$-crystal, we have
\begin{equation*}
\B^J \cong \bigsqcup_{\la}B_J(\la),
\end{equation*}
where the union is over $\la =\sum_{i=1}^n\la_i\ep_i\in P$ such that $0\ge \la_1=\la_2\ge \la_3=\la_4\ge \cdots$. 
\end{prop}
\pf It is enough to characterize the highest weight elements in $\B^J$ as an $\mf l$-crystal. We claim that $\textbf{c} = (c_{\ov{j}\ov{i}}) \in \B^J$ is an $\mf l$-highest weight element if and only if 
\begin{equation}\label{eq:characterization of l h.w.}
c_{\ov{n}\ov{n-1}} \ge c_{\ov{n-2}\ov{n-3}} \ge \cdots, 
\quad 
c_{\ov{j}\ov{i}} = 0 \ \textrm{elsewhere}.
\end{equation}
It is immediate from Proposition \ref{prop:signature for type D} that if $\bc$ satisfies \eqref{eq:characterization of l h.w.}, then $\te_i\bf\bc={\bf 0}$ for $i\in I\setminus\{n\}$. 
Conversely, suppose that $\textbf{c} = (c_{\ov{j}\ov{i}}) \in \B^J$ is an $\mf l$-highest weight element. If $c_{\ov{j}\ov{i}} \neq 0$ for some 
$(i,j)\not\in \{\,(n-1,n), (n-3,n-2),\ldots\,\}$,
then choose $c_{\ov{j}\ov{i}} \neq 0$ whose corresponding root $\ep_i+\ep_j$ is minimal with respect to \eqref{eq:convex order i_0}. If $j-i>1$, then $\te_{i}\bc\neq {\bf 0}$, and if $j-i=1$, then $\te_{j}\bc\neq {\bf 0}$. This is a contradiction. Next, if $c_{\ov{i+2} \ \ov{i+1}}<c_{\ov{i} \ov{i-1}}$ for some $i\geq 2$, then we have $\te_i\bc\neq {\bf 0}$, which is also a contradiction. Hence $\bc$ satisfies \eqref{eq:characterization of l h.w.}.
\qed

\subsection{Combinatorial description of $\varepsilon_n^*$} \label{section:bipath}
Let us give an explicit combinatorial description of $\varepsilon_n^*$ on $\B^J$.

\begin{df}{\rm

 A {\em path in $\Delta_n$} is a sequence $p=(\gamma_1,\ldots,\gamma_s)$ in $\Phi^+(J)$ for some $s\geq 1$ such that
\begin{itemize} 
\item[(1)] $\gamma_1,\ldots,\gamma_s\in \Phi^+(J)$,
\item[(2)] if $\gamma_{i} = \ep_{k}+\ep_{l+1}$ for some $k<l$, then $\gamma_{i+1}=\ep_{k+1}+\ep_{l+1}$ or $\ep_k+\ep_{l}$ (see \eqref{eq:small triangle}), 
\item[(3)] $\gamma_s = \ep_k+\ep_{k+1}$ for some $k$.
\end{itemize}
For $\beta\in \Phi^+(J)$, a {\em double path at $\beta$ in $\Delta_n$} is a pair of paths ${\bf p}=(p_1,p_2)$ in $\Delta_n$ of the same length with $p_1=(\gamma_1,\ldots,\gamma_s)$ and $p_2=(\delta_1,\ldots,\delta_s)$ such that 
\begin{itemize} 
\item[(1)] $\gamma_1=\delta_1=\beta$,   
\item[(2)] $\gamma_i$ is located to the strictly left of $\delta_i$ for $2\leq i\leq s$, 
\item[(3)] $\gamma_s=\ep_{k+1}+\ep_{k+2}$, $\delta_s=\ep_k+\ep_{k+1}$ for some $k\geq 1$.
\end{itemize}}
\end{df}

\begin{ex}\label{ex:double path}
{\rm For a double path ${\bf p}=(p_1,p_2)$ at $\beta$, if we draw an arrow from $\gamma_i$ to $\gamma_{i+1}$ in $p_1$ and from $\delta_i$ to $\delta_{i+1}$ in $p_2$, then $p_1$ and $p_2$ form a pair of non-intersecting paths starting from $\beta$ going downward to the bottom row in $\Delta_n$ with $p_1$ on the left, and $p_2$ on the right. The following is the list of double paths ${\bf p}$ at $\ep_1+\ep_5$ in $\Delta_5$.  
{\allowdisplaybreaks\begin{align*} 
&\xymatrixcolsep{0.3pc}\xymatrixrowsep{0.4pc}
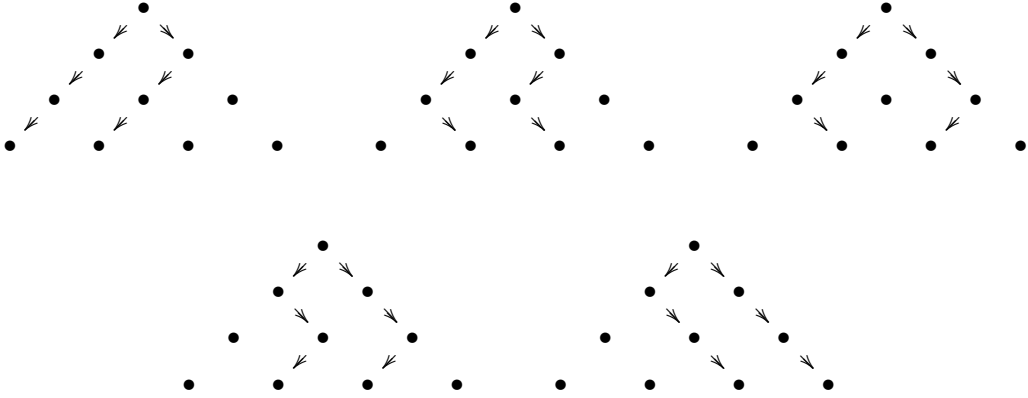
\xymatrix{
 &  & &  {\bullet}\ar@{->}[ld]\ar@{->}[rd] & & &  & \\
  & &  {\bullet}\ar@{->}[ld]  & &  {\bullet}\ar@{->}[ld] & & \\
   &  {\bullet}\ar@{->}[ld] & &  {\bullet}\ar@{->}[ld] &  &  {\bullet}  & \\  
     {\bullet} & &  {\bullet} & &  {\bullet} & &  {\bullet}
     } \ \ \ \
\xymatrix{
 &  & &  {\bullet}\ar@{->}[ld]\ar@{->}[rd] & & &  & \\
  & &  {\bullet}\ar@{->}[ld] & &  {\bullet}\ar@{->}[ld]  & & \\
   &  {\bullet}\ar@{->}[rd] & &  {\bullet}\ar@{->}[rd] &  &  {\bullet} & \\  
     {\bullet} & &  {\bullet} & &  {\bullet} & &  {\bullet}
     }\ \ \ \
\xymatrix{
 &  & &  {\bullet}\ar@{->}[ld]\ar@{->}[rd] & & &  & \\
  & &  {\bullet}\ar@{->}[ld] & &  {\bullet}\ar@{->}[rd] & & \\
   &  {\bullet}\ar@{->}[rd] & &  {\bullet}  &  &  {\bullet}\ar@{->}[ld] & \\  
     {\bullet} & &  {\bullet} & &  {\bullet} & &  {\bullet}
     } \\ \mbox{} \\
&\hskip 2cm\quad \xymatrixcolsep{0.3pc}\xymatrixrowsep{0.4pc}
\xymatrix{
 &  & &  {\bullet}\ar@{->}[ld]\ar@{->}[rd] & & &  & \\
  & &  {\bullet}\ar@{->}[rd] & &  {\bullet}\ar@{->}[rd] & & \\
   &  {\bullet} & &  {\bullet}\ar@{->}[ld] &  &  {\bullet}\ar@{->}[ld] & \\  
     {\bullet} & &  {\bullet} & &  {\bullet} & &  {\bullet}
     } \ \ \ \
\xymatrix{
 &  & &  {\bullet}\ar@{->}[ld]\ar@{->}[rd] & & &  & \\
  & &  {\bullet}\ar@{->}[rd] & &  {\bullet}\ar@{->}[rd] & & \\
   &  {\bullet} & &  {\bullet}\ar@{->}[rd] &  &  {\bullet}\ar@{->}[rd] & \\  
     {\bullet} & &  {\bullet} & &  {\bullet} & &  {\bullet}
     } 
\end{align*}}
}
\end{ex}

For ${\bf c}\in \B^J$ and a double path ${\bf p}$, let 
\begin{equation}
||{\bf c}||_{\bf p} = \sum_{\text{$\beta$ lying on ${\bf p}$}}c_\beta.
\end{equation}

\begin{thm}\label{thm:epsilon^*_n}
For ${\bf c}\in \B^J$,
\begin{equation*}
\varepsilon_n^*({\bf c}) = \max\{\,||{\bf c}||_{\bf p}\,|\,\text{$\bf p$ is a double path in $\Delta_n$}\,\}.
\end{equation*}
\end{thm}
\pf We give the proof of this formula in Section \ref{pf:epsilon}.
\qed

\begin{rem}{\rm
Let $\theta=\ep_1+\ep_n$ be the longest root in $\Phi^+$. Since $\theta$ is located at the top of $\Delta_n$, the formula in Theorem \ref{thm:epsilon^*_n} is equivalent to 
\begin{equation*}
\varepsilon_n^*({\bf c}) = \max\{\,||{\bf c}||_{\bf p}\,|\,\text{$\bf p$ is a double path at $\theta$ in $\Delta_n$}\,\}.
\end{equation*}

}
\end{rem}

\section{Affine crystal structure and Burge correspondence}\label{sec:main}

\subsection{KR crystals $B^{n,s}$} 
Let $\hat{\mf g}$ be an affine Kac-Moody algebra of type $D_n^{(1)}$ with
$\hat{I}=\{\,0,1,\ldots,n\,\}$ the index set for the simple roots.

\vskip 3mm

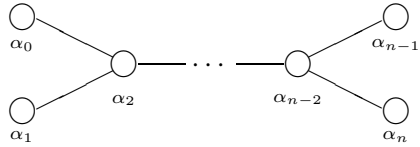
\begin{center} 
\setlength{\unitlength}{0.19in}
\begin{picture}(16,4.5)
\put(2.8,0.7){\makebox(0,0)[c]{$\bigcirc$}}
\put(2.8,3.3){\makebox(0,0)[c]{$\bigcirc$}}
\put(5.6,2){\makebox(0,0)[c]{$\bigcirc$}}
\put(10.4,2){\makebox(0,0)[c]{$\bigcirc$}}
\put(13.1,3.3){\makebox(0,0)[c]{$\bigcirc$}}
\put(13.1,0.7){\makebox(0,0)[c]{$\bigcirc$}}
\put(5.3,1.8){\line(-2,-1){2.1}}
\put(5.3,2.2){\line(-2,1){2.1}}
\put(6,2){\line(1,0){1.3}} 
\put(8.7,2){\line(1,0){1.3}} 
\put(10.7,2.2){\line(2,1){2}}
\put(10.7,1.8){\line(2,-1){2}}

\put(8,1.95){\makebox(0,0)[c]{$\cdots$}}
\put(2.8,2.5){\makebox(0,0)[c]{\tiny ${\alpha}_0$}}
\put(2.8,0.0){\makebox(0,0)[c]{\tiny ${\alpha}_1$}}
\put(5.6,1){\makebox(0,0)[c]{\tiny ${\alpha}_2$}}
\put(10.4,1){\makebox(0,0)[c]{\tiny ${\alpha}_{n-2}$}}
\put(13.1,2.5){\makebox(0,0)[c]{\tiny ${\alpha}_{n-1}$}}
\put(13.1,0.0){\makebox(0,0)[c]{\tiny ${\alpha}_{n}$}}

\end{picture}
\end{center} \vskip 3mm
For $r\in \{0,n\}$, let $\hat{\mf g}_r$ be the subalgebra of $\hat{\mf g}$ corresponding to $\{\,\alpha_i\,|\,i\in \hat{I} \setminus\{r\}\,\}$. Then $\hat{\mf g}_0=\mf g$, and $\hat{\mf g}_0\cap \hat{\mf g}_n=\mf l$. 
Let $\hat{P}=\bigoplus_{i\in \hat{I}}\Z\La_i\oplus\Z\delta$ be the weight lattice of $\hat{\g}$, where $\delta$ is the positive imaginary null root and $\La_i$ is the $i$-th fundamental weight (see \cite{K}). We regard $P=\bigoplus_{i=1}^n\Z\epsilon_i$ as a sublattice of $\hat{P}/\Z\delta$ by putting 
$\epsilon_1=\Lambda_1-\Lambda_0$, $\epsilon_2=\Lambda_2-\Lambda_1-\Lambda_0$, $\epsilon_k=\Lambda_k-\Lambda_{k-1}$ for $k=3,\ldots, n-2$, $\epsilon_{n-1}=\Lambda_{n-1}+\Lambda_n-\Lambda_{n-2}$ and $\epsilon_{n}=\Lambda_{n}-\Lambda_{n-1}$.
In particular, we have $\alpha_0 = -\ep_1-\ep_2$ in $P$. 
If $\varpi'_i$ are the fundamental weights for $\hat{\mf g}_{n}$ for $i\in \hat{I}\setminus\{n\}$, then $\varpi'_i=\varpi_i$ for $i\in \hat{I}\setminus\{0,n\}$ and $\varpi'_0=-\varpi_n$.

For ${\bf c}\in \B^J$, define
\begin{equation}\label{eq:e_0 and f_0 on B^J}
\begin{split}
&\te_0{\bf c} = {\bf c} + {\bf 1}_{\ep_1+\ep_2},\quad
\tf_0{\bf c} = 
\begin{cases}
{\bf c} - {\bf 1}_{\ep_1+\ep_2},& \text{if $c_{\ep_1+\ep_2}>0$},\\
{\bf 0}, & \text{otherwise},
\end{cases}\\
&\varphi_0({\bf c})=\max\{\,k \,|\, \tf_0^k {\bf c}\neq {\bf 0}\,\}, \ \ 
\varepsilon_0({\bf c})=\varphi_0({\bf c})- \langle {\rm wt}({\bf c}), h_0 \rangle.
\end{split}
\end{equation}

\begin{lem}\label{lem:affine crystal B^J}
The set $\B^J$ is a $\hat{\mf g}$-crystal with respect to ${\rm wt}$, $\varepsilon_i$, $\varphi_i$, $\te_i$, $\tf_i$ for $i\in \hat{I}$, where ${\rm wt}$ is the restriction of ${\rm wt} : \B \longrightarrow P$ to $\B^J$.
\end{lem}
\pf It follows directly from \eqref{eq:e_0 and f_0 on B^J}.
\qed

\begin{rem}{\rm
The inclusions  
	$\B^{J,s}\hookrightarrow\B^{J,t}\hookrightarrow\B^J$
	for $s \le t$ are embeddings of $\widehat{\g}$-crystals (cf. \eqref{eq:B^J_s-2}), and hence $\B^J$ is a direct limit of $\left\{ \B^{J,s} \,|\, s \in \mathbb{Z}_+ \right\}$.
	}
\end{rem}

\begin{thm}\label{thm:main-1}
For $s\geq 1$, $\B^{J,s}\otimes T_{s\varpi_n}$ is a regular $\hat{\mf g}$-crystal and $$\B^{J,s}\otimes T_{s\varpi_n} \cong B^{n,s},$$ 
where $B^{n,s}$ is the Kirillov-Reshetikhin crystal of type $D_n^{(1)}$ associated to $s\varpi_n$.
\end{thm}
\pf By \eqref{eq:B^J_s-2}, $\B^{J,s}\otimes T_{s\varpi_n}$ is a regular $\hat{\mf g}_0$-crystal. 
By Proposition \ref{prop:signature for type D}, we see that the $\hat{\mf g}_n$-crystal $\B^{J,s}\otimes T_{s\varpi_n}$ is isomorphic to the dual of the $\hat{\mf g}_0$-crystal 
$\B^{J,s}\otimes T_{s\varpi_n}$ assuming that 
$\hat{\mf g}_n \cong \hat{\mf g}_0$ under the correspondence 
$\alpha_i \leftrightarrow -\alpha_{n-i}$ for $0\leq i\leq n-1$. 
This implies that $\B^{J,s}\otimes T_{s\varpi_n}$ is a regular $\hat{\mf g}_n$-crystal, and hence a regular $\hat{\mf g}$-crystal.
It is known that $B^{n,s}$ is classically irreducible, that is, $B^{n,s} \cong B(s\varpi_n)$ as a $\hat{\mf g}_0$-crystal (see \cite{FOS}). Therefore, it follows from \cite[Lemma 2.6]{ST} that $\B^{J,s}\otimes T_{s\varpi_n} \cong B^{n,s}$.
\qed
  
\begin{rem}\label{rem:polytope and limit}{\rm
By Theorem \ref{thm:epsilon^*_n}, we have 
	\begin{equation*}
	\B^{J,s} = \underset{\text{${\bf p}$}}{\bigcap}\left\{ {\bf c} \in \B^J \,|\, ||{\bf c}||_{\bf p} \le s \right\},
	\end{equation*}
	where $\bf p$ runs over the double paths in $\Delta_n$. This gives a polytope realization of the KR crystal $B^{n,s}$.
	By \cite{FOS2}, $\{ B^{J,s} \}$ is a family of perfect KR crystals. It is conjectured that $\{ B^{n,s} \}$ has the limit in the sense of \cite{KNO}, that is, $\{ B^{n,s} \}$ is a coherent family.
}
\end{rem}

\subsection{Burge correspondence}\label{subsec:burge}
Let us recall some necessary notions following \cite{Ful}. 
Let $\mathscr{P}$ be the set of partitions $\lambda=(\lambda_i)_{i\geq 1}$, 
which are often identified with Young diagrams. 
Let $\lambda'=(\lambda'_i)_{i\geq 1}$ be the conjugate of $\lambda$, and 
let $\lambda^\pi$ be the skew Young diagram obtained by $180^{\circ}$-rotation of $\lambda$. Let $\ell(\lambda)$ denote the {length of $\lambda$}, and let $\cP_n=\{\,\la\,|\,\ell(\la)\leq n\,\}$.

Let $[\ov{n}]:=\{\,\ov{n}<\cdots< \ov{1}\,\}$ be a linearly ordered set. 
Let $\W$ be the set of finite words in $[\ov{n}]$.
For a skew Young diagram $\lambda/\mu$, let $SST_n(\lambda/\mu)$ or simply $SST(\la/\mu)$ denote the set of semistandard tableaux of shape $\lambda/\mu$ with entries in $[\ov{n}]$.
For $T\in SST(\lambda/\mu)$, let $w(T)$ be a word in $\W$ obtained by reading the entries of $T$ row by row from top to bottom, and from right to left in each row, and let ${\rm sh}(T)$ denote the shape of $T$.

Let $T^{\nw}$ be the unique semistandard tableau such that ${\rm sh}(T^{\nw})\in \cP$ and $w(T^{\nw})$ is Knuth equivalent to $w(T)$. We define $T^{\se}$ in a similar way such that ${\rm sh}(T^{\se})\in \cP^\pi$. Note that if ${\rm sh}(T^{\nw})=\nu$, then ${\rm sh}(T^{\se})=\nu^\pi$.

For $a\in [\ov{n}]$ and $U\in SST(\lambda)$ with $\la\in\cP_n$, let $a\rightarrow U$ be the tableau obtained by applying the Schensted's column insertion of $a$ into $U$. 
Similarly, for $V\in SST(\lambda^\pi)$ and $b\in [\ov{n}]$, let $V\leftarrow b$ be the tableau obtained by applying the Schensted's column insertion of $b$ into $V$ in a reverse way starting from the rightmost column. 
For $w=w_1\ldots w_r\in\W$, we define $P(w)^{\nw}=( w_r\rightarrow(\cdots(w_2\rightarrow w_1)))$. Note that $P(w)^{\se} = ( (w_r\leftarrow w_{r-1}) \leftarrow \cdots \leftarrow w_1)$.

The goal in the remaining of this section is to give an explicit isomorphism in Proposition \ref{prop:decomp of B^J}, and extend it as an isomorphism of $D_n^{(1)}$-crystals.

Let us first recall a variation of RSK correspondence for type $D$ \cite{B}. Set 
\begin{equation}
\begin{split}
{\mc T}^{\se}:= \bigsqcup_{\substack{\la\in \cP_n \\ \la' \text{:even}}} SST(\la^\pi),
\quad\quad
{\mc T}^{\nw}:= \bigsqcup_{\substack{\la\in \cP_n \\ \la' \text{:even}}} SST(\la),
\end{split}
\end{equation}
where we say that $\la'$ is even if each part of $\la'$ is even
Let $\Omega$ be the set of biwords $(\ba,\bb)\in
\W\times \W$ such that
\begin{itemize}
\item[(1)] $\ba=a_1\cdots a_r$ and $\bb=b_1\cdots b_r$ for some $r\geq 0$,

\item[(2)] $a_i < b_i$ for $1\leq i\leq r$,

\item[(3)] $(a_1,b_1)\leq \cdots \leq (a_r,b_r)$, 
\end{itemize}
where $(a,b)< (c,d)$ if and only if $(a<c)$ or ($a=c$ and $b>d$) for $(a,b)$, $(c,d)\in \W \times \W$. 
We denote by ${\bf c}(\ba,\bb)$ the unique element in $\B^J$ corresponding to $(\ba,\bb)$ such that $c_{ab}=\left|\{\,k\,|\,(a_k,b_k)=(a,b) \,\}\right|$.

For $(\ba,\bb)\in \Omega$ with $\ba=a_1\cdots a_r$ and $\bb=b_1\cdots b_r$, we define a sequence of tableaux $P_r, P_{r-1}, \ldots, P_1$ inductively as follows:
\begin{itemize}
\item[(1)] let $P_1$ be a vertical domino  
{\tiny 
${\def\lr#1{\multicolumn{1}{|@{\hspace{.6ex}}c@{\hspace{.6ex}}|}{\raisebox{.1ex}{$#1$}}}\raisebox{-2.5ex}
{$\begin{array}[b]{c}
\cline{1-1}
\lr{a_r}\\
\cline{1-1}
\lr{b_r}\\
\cline{1-1}
\end{array}$}
}$
},

\item[(2)] if $P_{k+1}$ is given for $1\leq k\leq r-1$, then define $P_{k}$ to be the tableau obtained by first applying the column insertion to get $P_{k+1}\leftarrow b_{k}$, and then adding $\boxed{a_{k}}$ at the conner of $P_{k+1}\leftarrow b_{k}$ located above the box ${\rm sh}(P_{k+1}\leftarrow b_{k})/{\rm sh}(P_{k+1})$.
\end{itemize}
We put $P^{\se}(\ba,\bb):=P_1$.
It is not difficult to see from the definition that $P^{\se}(\ba,\bb) \in SST(\la^\pi)$ for some $\la\in \cP$ such that $\la'$ is even.

For ${\bf c}\in \B^J$, let $P^{\se}({\bf c})=P^{\se}(\ba,\bb)$ where ${\bf c}={\bf c}(\ba,\bb)$.
Since the map $(\ba,\bb) \mapsto P^{\se}(\ba,\bb)$ is a bijection from $\Omega$ to $\mc T^{\se}$ \cite{B}, we have a bijection 
\begin{equation}\label{eq:kappa_se}
\xymatrixcolsep{3pc}\xymatrixrowsep{0pc}\xymatrix{
\kappa^{\se} : \B^J  \ \ar@{->}[r] & \ {\mc T}^{\se}  \\
\quad {\bf c} \ar@{|->}[r] & P^{\se}({\bf c}) }.
\end{equation}
Similarly, let $\Omega'$ be the set of biwords $(\ba,\bb)\in
\W \times \W$ satisfying the same conditions as in $\Omega$ except that $<$ is replaced by $<'$, where $(a,b)<' (c,d)$ if and only if $(b<d)$ or ($b=d$ and $a<c$) for $(a,b)$ and $(c,d)\in \W \times \W$. We define ${\bf c}'(\ba,\bb)$ in the same way as in ${\bf c}(\ba,\bb)$. Given $(\ba,\bb)\in \Omega'$ with $\ba=a_1\cdots a_r$ and $\bb=b_1\cdots b_r$, define a sequence of tableaux $P_1, P_2, \ldots, P_r$ inductively as follows:
\begin{itemize}
\item[(1)] let $P_1$ be a vertical domino  
{\tiny 
${\def\lr#1{\multicolumn{1}{|@{\hspace{.6ex}}c@{\hspace{.6ex}}|}{\raisebox{.1ex}{$#1$}}}\raisebox{-2.5ex}
{$\begin{array}[b]{c}
\cline{1-1}
\lr{a_1}\\
\cline{1-1}
\lr{b_1}\\
\cline{1-1}
\end{array}$}
}$
},

\item[(2)] if $P_{k-1}$ is given for $2\leq k\leq r$, then define $P_{k}$ to be the tableau obtained by first applying the column insertion to get $a_k\rightarrow P_{k-1}$, and then adding $\boxed{b_{k}}$ at the conner of $a_k\rightarrow P_{k-1}$ located below the box ${\rm sh}(a_k\rightarrow P_{k-1})/{\rm sh}(P_{k-1})$,
\end{itemize}
and put $P^{\nw}(\ba,\bb):=P_r$.
For ${\bf c}\in \B^J$, let $P^{\nw}({\bf c})=P^{\nw}(\ba,\bb)$ where ${\bf c}={\bf c}'(\ba,\bb)$. Then we also have a bijection 
\begin{equation}\label{eq:kappa_nw}
\xymatrixcolsep{3pc}\xymatrixrowsep{0pc}\xymatrix{
\kappa^{\nw} : \B^J  \ \ar@{->}[r] & \ {\mc T}^{\nw}  \\
\quad {\bf c} \ar@{|->}[r] & P^{\nw}({\bf c})}.
\end{equation}

\begin{ex}{\rm 
\ytableausetup {mathmode, boxsize=1.0em} 
Suppose that $n=5$.
Let ${\bf c} \in \B^J$ be given by
\begin{equation*} 
\xymatrixcolsep{-0.2pc}\xymatrixrowsep{0pc}\xymatrix{
& & & & 2 & & & & \\
& & & 1 & & 0 & & \\
& & 1 & & 2 & & 1 & \\  
& 2 & & 1 & & 0 & & 1}
\end{equation*}
(see Example \ref{ex:Delta_n}),
where ${\bf c} = {\bf c}({\bf a}, {\bf b})$ for $({\bf a}, {\bf b}) \in \Omega$ with
\begin{equation*}
\biggr(\begin{array}{c} \ba \\ \bb \end{array}\biggr)=
\left(\begin{array}{ccccccccccc}
\ov{5} & \ov{5} & \ov{5} & \ov{5} & \ov{5} & \ov{5} & \ov{4} & \ov{4} & \ov{4} & \ov{3} & \ov{2} \\
\ov{1} & \ov{1} & \ov{2} & \ov{3} & \ov{4} & \ov{4} & \ov{2} & \ov{2} & \ov{3} & \ov{1} & \ov{1}
\end{array} \right).
\end{equation*} 
The following is the sequence of tableaux $P_r, P_{r-1}, \ldots, P_1=:P^{\se}(\ba,\bb)$
given in the definition of $\kappa^{\se}$ \eqref{eq:kappa_se}:
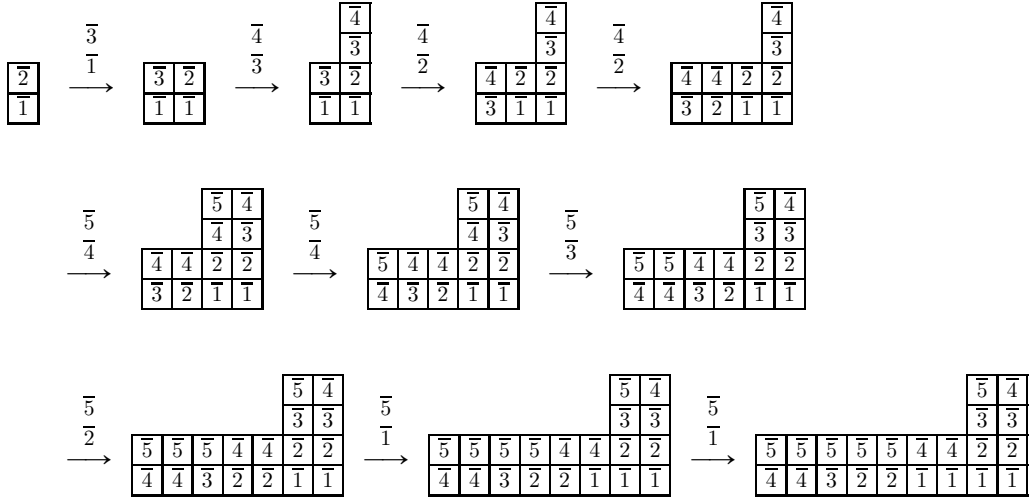
\begin{equation*}
\begin{split}
&
\begin{ytableau}
\none \\
\none \\
\ov{\tl{2}} \\
\ov{\tl{1}}
\end{ytableau} \quad
\raisebox{-5ex}{\text{$\overset{\begin{ytableau}\none[\ov{\tl{3}}] \\ \none[\ov{\tl{1}}]\end{ytableau}}{\longrightarrow}$}} \quad
\begin{ytableau}
\none \\
\none \\
\ov{\tl{3}} & \ov{\tl{2}} \\
\ov{\tl{1}} & \ov{\tl{1}}
\end{ytableau} \quad 
\raisebox{-5ex}{\text{$\overset{\begin{ytableau}\none[\ov{\tl{4}}] \\ \none[\ov{\tl{3}}]\end{ytableau}}{\longrightarrow}$}}\quad
\begin{ytableau}
\none & \ov{\tl{4}} \\
\none & \ov{\tl{3}} \\
\ov{\tl{3}} & \ov{\tl{2}} \\
\ov{\tl{1}} & \ov{\tl{1}}
\end{ytableau} \quad 
\raisebox{-5ex}{\text{$\overset{\begin{ytableau}\none[\ov{\tl{4}}] \\ \none[\ov{\tl{2}}]\end{ytableau}}{\longrightarrow}$}} \quad
\begin{ytableau}
\none & \none & \ov{\tl{4}} \\
\none & \none & \ov{\tl{3}} \\
\ov{\tl{4}} &\ov{\tl{2}} & \ov{\tl{2}} \\
\ov{\tl{3}} & \ov{\tl{1}} & \ov{\tl{1}}
\end{ytableau}  \quad
\raisebox{-5ex}{\text{$\overset{\begin{ytableau}\none[\ov{\tl{4}}] \\ \none[\ov{\tl{2}}]\end{ytableau}}{\longrightarrow}$}} \quad
\begin{ytableau}
\none & \none & \none & \ov{\tl{4}} \\
\none & \none & \none & \ov{\tl{3}} \\
\ov{\tl{4}} & \ov{\tl{4}} &\ov{\tl{2}} & \ov{\tl{2}} \\
\ov{\tl{3}} & \ov{\tl{2}} & \ov{\tl{1}} & \ov{\tl{1}}
\end{ytableau} \\ \mbox{} \\
&
\quad \quad 
\raisebox{-5ex}{\text{$\overset{\begin{ytableau}\none[\ov{\tl{5}}] \\ \none[\ov{\tl{4}}]\end{ytableau}}{\longrightarrow}$}}\quad
\begin{ytableau}
\none & \none & \ov{\tl{5}} & \ov{\tl{4}} \\
\none & \none & \ov{\tl{4}} & \ov{\tl{3}} \\
\ov{\tl{4}} & \ov{\tl{4}} &\ov{\tl{2}} & \ov{\tl{2}} \\
\ov{\tl{3}} & \ov{\tl{2}} & \ov{\tl{1}} & \ov{\tl{1}}
\end{ytableau}
\quad  \raisebox{-5ex}{\text{$\overset{\begin{ytableau}\none[\ov{\tl{5}}] \\ \none[\ov{\tl{4}}]\end{ytableau}}{\longrightarrow}$}}\quad
\begin{ytableau}
\none & \none & \none & \ov{\tl{5}} & \ov{\tl{4}} \\
\none & \none & \none & \ov{\tl{4}} & \ov{\tl{3}} \\
\ov{\tl{5}} & \ov{\tl{4}} & \ov{\tl{4}} &\ov{\tl{2}} & \ov{\tl{2}} \\
\ov{\tl{4}} & \ov{\tl{3}} & \ov{\tl{2}} & \ov{\tl{1}} & \ov{\tl{1}}
\end{ytableau} \quad 
\raisebox{-5ex}{\text{$\overset{\begin{ytableau}\none[\ov{\tl{5}}] \\ \none[\ov{\tl{3}}]\end{ytableau}}{\longrightarrow}$}}\quad
\begin{ytableau}
\none & \none & \none & \none & \ov{\tl{5}} & \ov{\tl{4}} \\
\none & \none & \none & \none & \ov{\tl{3}} & \ov{\tl{3}} \\
\ov{\tl{5}} & \ov{\tl{5}} & \ov{\tl{4}} & \ov{\tl{4}} &\ov{\tl{2}} & \ov{\tl{2}} \\
\ov{\tl{4}} & \ov{\tl{4}} & \ov{\tl{3}} & \ov{\tl{2}} & \ov{\tl{1}} & \ov{\tl{1}}
\end{ytableau}\\ \mbox{} \\
&\quad\quad
\raisebox{-5ex}{\text{$\overset{\begin{ytableau}\none[\ov{\tl{5}}] \\ \none[\ov{\tl{2}}]\end{ytableau}}{\longrightarrow}$}}\ \
\begin{ytableau}
\none & \none & \none & \none & \none & \ov{\tl{5}} & \ov{\tl{4}} \\
\none & \none & \none & \none & \none & \ov{\tl{3}} & \ov{\tl{3}} \\
\ov{\tl{5}} & \ov{\tl{5}} & \ov{\tl{5}} & \ov{\tl{4}} & \ov{\tl{4}} &\ov{\tl{2}} & \ov{\tl{2}} \\
\ov{\tl{4}} & \ov{\tl{4}} & \ov{\tl{3}} & \ov{\tl{2}} & \ov{\tl{2}} & \ov{\tl{1}} & \ov{\tl{1}}
\end{ytableau} \ \
\raisebox{-5ex}{\text{$\overset{\begin{ytableau}\none[\ov{\tl{5}}] \\ \none[\ov{\tl{1}}]\end{ytableau}}{\longrightarrow}$}}\ \
\begin{ytableau}
\none & \none & \none & \none & \none & \none & \ov{\tl{5}} & \ov{\tl{4}} \\
\none & \none & \none & \none & \none & \none & \ov{\tl{3}} & \ov{\tl{3}} \\
\ov{\tl{5}} & \ov{\tl{5}} & \ov{\tl{5}} & \ov{\tl{5}} & \ov{\tl{4}} & \ov{\tl{4}} &\ov{\tl{2}} & \ov{\tl{2}} \\
\ov{\tl{4}} & \ov{\tl{4}} & \ov{\tl{3}} & \ov{\tl{2}} & \ov{\tl{2}} & \ov{\tl{1}} & \ov{\tl{1}} & \ov{\tl{1}}
\end{ytableau}\ \
\raisebox{-5ex}{\text{$\overset{\begin{ytableau}\none[\ov{\tl{5}}] \\ \none[\ov{\tl{1}}]\end{ytableau}}{\longrightarrow}$}}\ \
\begin{ytableau}
\none & \none & \none & \none & \none & \none & \none & \ov{\tl{5}} & \ov{\tl{4}} \\
\none & \none & \none & \none & \none & \none & \none & \ov{\tl{3}} & \ov{\tl{3}} \\
\ov{\tl{5}} & \ov{\tl{5}} & \ov{\tl{5}} & \ov{\tl{5}} & \ov{\tl{5}} & \ov{\tl{4}} & \ov{\tl{4}} &\ov{\tl{2}} & \ov{\tl{2}} \\
\ov{\tl{4}} & \ov{\tl{4}} & \ov{\tl{3}} & \ov{\tl{2}} & \ov{\tl{2}} & \ov{\tl{1}} & \ov{\tl{1}} & \ov{\tl{1}} & \ov{\tl{1}}
\end{ytableau}
\end{split}
\end{equation*}
Here we use the notation
$T \overset{\begin{ytableau}\none[\ov{\tl{\em j}}] \\ \none[\ov{\tl{\em i}}]\end{ytableau}}{\rightarrow} T'
$ when $T=P_{k+1}$, $T'=P_k$ and $(a_k,b_k)=(\ov{j},\ov{i})$.
Hence, we have
\begin{equation*} \label{eq:ex_burge}
\raisebox{-2ex}{$\kappa^{\se}({\bf c}) =$}\ \underbrace{\begin{ytableau}
\none & \none & \none & \none & \none & \none & \none & \ov{\tl{5}} & \ov{\tl{4}} \\
\none & \none & \none & \none & \none & \none & \none & \ov{\tl{3}} & \ov{\tl{3}} \\
\ov{\tl{5}} & \ov{\tl{5}} & \ov{\tl{5}} & \ov{\tl{5}} & \ov{\tl{5}} & \ov{\tl{4}} & \ov{\tl{4}} &\ov{\tl{2}} & \ov{\tl{2}} \\
\ov{\tl{4}} & \ov{\tl{4}} & \ov{\tl{3}} & \ov{\tl{2}} & \ov{\tl{2}} & \ov{\tl{1}} & \ov{\tl{1}} & \ov{\tl{1}} & \ov{\tl{1}}
\end{ytableau}}_{\text{$\sharp$ of columns  $= 9$}}
\end{equation*}
On the other hand, we have by Theorem \ref{thm:epsilon^*_n} 
$
\varepsilon^*_4({\bf c}) = 9
$ 
since the following double path ${\bf p}$
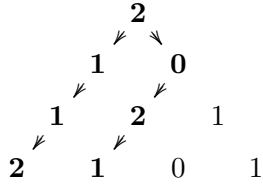
\begin{equation*} 
\xymatrixcolsep{0.1pc}\xymatrixrowsep{0.4pc}\xymatrix{
& & & & {\bf 2}\ar@{->}[ld]\ar@{->}[rd] & & & & \\
& & & {\bf 1}\ar@{->}[ld] & & {\bf 0}\ar@{->}[ld] & & \\
& & {\bf 1}\ar@{->}[ld] & & {\bf 2}\ar@{->}[ld] & & 1 & \\  
& {\bf 2} & & {\bf 1} & & 0 & & 1}
\end{equation*} 
takes the maximum value of $||{\bf c}||_{{\bf p}}$, 
which is equal to the number of columns of $\kappa^{\se}({\bf c})$.}
\end{ex}

\subsection{Isomorphism of affine crystals}

We regard
$[\ov{n}]=\{\,\ov{n}<\cdots< \ov{1}\,\}$ as the crystal of dual natural representation of $\mf l$ with ${\rm wt}(\ov{k})=-\epsilon_k$. 
Then $\W$ is a regular ${\mf l}$-crystal, where $w=w_1\ldots w_r$ is identified with $w_1\otimes \cdots \otimes w_r$.
For $\la\in \cP_n$, $SST(\la)$ is a regular $\mf l$-crystal with lowest weight $-\sum_{i=1}^n\la_i\ep_i$, where $T$ is identified with $w(T)$ \cite{KN}. In particular $\mc T^{\se}$ and ${\mc T}^{\nw}$ are regular ${\mf l}$-crystals.

Let us recall the $\hat{\mf g}_0$-crystal structure on $\mc T^{\se}$ \cite[Section 5.2]{K13}.
Let $T\in {\mc T}^{\se}$ be given. For $k\geq 1$, let $t_k$  be the entry
in the top of the $k$-th column of $T$ (enumerated from the right). Consider  $\sigma=(\sigma_1,\sigma_2,\ldots)$, where
\begin{equation*}
\sigma_k=
\begin{cases}
+ \ , & \text{if $t_k> \ov{n-1}$ or the $k$-th column is empty}, \\
- \ ,& \text{if  the $k$-th column has both $ \ov{n-1}$ and $ \ov{n}$ as its entries}, \\
\,\cdot\ \ , & \text{otherwise}.
\end{cases}
\end{equation*}
Then $\te_n T$ is
obtained from $T$ by removing  
\raisebox{-.6ex}{{\tiny ${\def\lr#1{\multicolumn{1}{|@{\hspace{.6ex}}c@{\hspace{.6ex}}|}{\raisebox{-.3ex}{$#1$}}}\raisebox{-.6ex}
{$\begin{array}[b]{c}
\cline{1-1}
\lr{ \ov{n}}\\
\cline{1-1}
\lr{ \!\ov{n\!-\!1}\!}\\
\cline{1-1}
\end{array}$}}$}}
in the column  corresponding to the right-most $-$ in ${\sigma}^{\rm red}$ (see Section \ref{subssec:tensor product rule on B} for $\sigma^{\rm red}$).  If there is no such $-$ sign, then we define $\te_n T={\bf 0}$, and 
$\tf_n T$ is obtained from $T$ by adding  
\raisebox{-.6ex}{{\tiny ${\def\lr#1{\multicolumn{1}{|@{\hspace{.6ex}}c@{\hspace{.6ex}}|}{\raisebox{-.3ex}{$#1$}}}\raisebox{-.6ex}
{$\begin{array}[b]{c}
\cline{1-1}
\lr{ \ov{n}}\\
\cline{1-1}
\lr{\!\ov{n\!-\!1}\!}\\
\cline{1-1}
\end{array}$}}$}}
column  corresponding to the left-most $+$ in ${\sigma}^{\rm red}$. 
Hence $\mc T^{\se}$ is a $\hat{\mf g}_0$-crystal with respect to ${\rm wt}$, 
$\varepsilon_i$, $\varphi_i$,
$\te_i$, $\tf_i$ $(i\in \hat{I} \setminus\{0\})$, where 
$\varepsilon_n(T)=\max\{\,k \,|\, \te_n^k T\neq {\bf 0}\,\}$ and
$\varphi_n(T)=\varepsilon_n(T)+ \langle {\rm wt}(T), h_n \rangle$.

Similarly, we have a $\hat{\mf g}_n$-crystal structure on $\mc T^{\nw}$ \cite[Section 5.2]{K13}. 
Let  $T\in {\mc T}^{\nw}$ be given.
For $k\geq 1$, let $t_k$  be the entry
in the bottom of the $k$-th column of $T$ (enumerated from the left). Consider  $\sigma=(\ldots,\sigma_2,\sigma_1)$, where
\begin{equation*}
\sigma_k=
\begin{cases}
- \ , & \text{if $t_k< \ov{2}$ or the $k$-th column is empty}, \\
+ \ ,& \text{if  the $k$-th column has both $\ov{1}$ and $\ov{2}$ as its entries,}\\
\, \cdot \ \ , & \text{otherwise}.
\end{cases}
\end{equation*}
Then $\te_0 T$ is given by adding
\raisebox{-.6ex}{{\tiny ${\def\lr#1{\multicolumn{1}{|@{\hspace{.6ex}}c@{\hspace{.6ex}}|}{\raisebox{-.3ex}{$#1$}}}\raisebox{-.6ex}
{$\begin{array}[b]{c}
\cline{1-1}
\lr{ \ov{2}}\\
\cline{1-1}
\lr{ \ov{1}}\\
\cline{1-1}
\end{array}$}}$}}
to the bottom of the column  corresponding to the right-most $-$ in ${\sigma}^{\rm red}$,
and $\tf_0 T$ is obtained from $T$ by removing
\raisebox{-.6ex}{{\tiny ${\def\lr#1{\multicolumn{1}{|@{\hspace{.6ex}}c@{\hspace{.6ex}}|}{\raisebox{-.3ex}{$#1$}}}\raisebox{-.6ex}
{$\begin{array}[b]{c}
\cline{1-1}
\lr{ \ov{2}}\\
\cline{1-1}
\lr{ \ov{1}}\\
\cline{1-1}
\end{array}$}}$}}
in the
column  corresponding to the left-most $+$ in ${\sigma}^{\rm red}$. 
If there is no such $+$ sign, then we define $\tf_0 T={\bf 0}$.
Hence ${\mc T}^{\nw}$ is a $\hat{\mf g}_n$-crystal with respect to ${\rm wt}$, $\varepsilon_i$, $\varphi_i$, $\te_i$, $\tf_i$ $(i\in \hat{I}\setminus\{n\})$, where 
$\varphi_0(T)=\max\{\,k\ \,|\, \tf_0^k T\neq {\bf 0}\,\}$ and
$\varepsilon_0(T)=\varphi_0(T)- \langle {\rm wt}(T), h_0 \rangle$.

\begin{thm}\label{thm:isomorphism theorem}
The bijection $\kappa^{\se}$ in \eqref{eq:kappa_se} is an isomorphism of $\hat{\mf g}_0$-crystals, and  
the bijection $\kappa^{\nw}$ in \eqref{eq:kappa_nw} is an isomorphism of $\hat{\mf g}_n$-crystals.
\end{thm}
\pf  The proof of (1) is given in Section \ref{pf:burge}. The proof of (2) is similar to (1).
\qed

For a semistandard tableau $T$ of skew shape, let $[T]$ denote the equivalence class of $T$ with respect to Knuth equivalence. If we define 
\begin{equation*}
\begin{split}
\td{x}_i[T]=
\begin{cases}
[\td{x}_0 T^{\nw}], & \text{if $i=0$}, \\
[\td{x}_n T^{\se}], & \text{if $i=n$}, \\
[\td{x}_iT], & \text{otherwise},
\end{cases}
\end{split}
\end{equation*}
for $i\in \hat{I}$ and $x=e,f$ (we assume that $[{\bf 0}]={\bf 0}$), then the set
\begin{equation}
{\mc T}=\{\,[T]\,|\,T\in {\mc T}^{\se}\,\}=\{\,[T]\,|\,T\in {\mc T}^{\nw}\,\}
\end{equation}
is a $\hat{\g}$-crystal with respect to $\te_i$, $\tf_i$ $(i\in I)$, 
where ${\rm wt}$, $\varepsilon_i$, and $\varphi_i$ are well-defined on $[T]$ \cite[Section 5.3]{K13}. 
Therefore, 
\begin{cor}
The map
\begin{equation*}
\xymatrixcolsep{3pc}\xymatrixrowsep{0pc}\xymatrix{
\kappa : \B^J  \ \ar@{->}[r] & \ {\mc T} \quad\quad\quad\quad\quad \\
\quad {\bf c} \ar@{|->}[r] & [P^{\nw}({\bf c})]=[P^{\se}({\bf c})]},
\end{equation*}
is an isomorphism of $\hat{\mf g}$-crystals.
\end{cor}

For $s\geq 1$, let 
${\mc T}^s=\{\,[T] \,|\,\ell({\rm sh}(T)')\leq s\,\}\subset {\mc T}$.
It is shown in \cite[Theorem 5.4]{K13} that ${\mc T}^s\otimes T_{s\varpi_n}\cong B^{n,s}$. Therefore,  

\begin{cor}\label{eq:isomorphism theorem at s}
The map $\kappa$ when restricted to $\B^{J,s}$ gives an isomorphism of $\hat{\mf g}$-crystals
\begin{equation*}
\xymatrixcolsep{3pc}\xymatrixrowsep{0pc}\xymatrix{
\kappa : \B^{J,s}  \ \ar@{->}[r] & \ {\mc T}^s}.
\end{equation*}
\end{cor}


\subsection{Shape formula} 

For ${\bf c}\in \B^J$, let 
\begin{equation}\label{eq:la(c)}
\la({\bf c})=(\la_1({\bf c})\geq \ldots \geq  \la_{\ell}({\bf c}))
\end{equation}
be the partition corresponding to the regular ${\mf l}$-subcrystal of $\B^J$ including ${\bf c}$, that is, $\la({\bf c})={\rm sh}(\kappa^{\nw}({\bf c}))$ by Theorem \ref{thm:isomorphism theorem}. 
Note that $\ell=2[\frac{n}{2}]$ and $\la_{2i-1}(\bc)=\la_{2i}(\bc)$ for $1\leq i\leq [\frac{n}{2}]$.
We have by Theorem \ref{thm:epsilon^*_n} and Corollary \ref{eq:isomorphism theorem at s}
\begin{equation}\label{eq:formula for la_1}
\la_1({\bf c}) = \max\{\,||{\bf c}||_{\bf p}\,|\,\text{$\bf p$ is a double path at $\theta$}\,\},
\end{equation} 

We can further characterize the whole partition $\la(\bc)$ in terms of double paths on $\Delta_n$ generalizing \eqref{eq:formula for la_1} as follows.

\begin{thm} \label{thm:shape}
For ${\bf c}\in \B^J$ and $1\leq l\leq [\frac{n}{2}]$, we have
\begin{equation*}
\la_1({\bf c})+\la_3({\bf c})+\cdots+\la_{2l-1}({\bf c}) = 
\max_{{\bf p}_1,\ldots , {\bf p}_{l}}
\{\,||{\bf c}||_{{\bf p}_1}+\cdots+||{\bf c}||_{{\bf p}_l}\,\},
\end{equation*}
where ${\bf p}_1,\ldots, {\bf p}_l$ are mutually non-intersecting double paths in $\Delta_n$ and each ${\bf p}_i$ starts at the $(2i-1)$-th row of $\Delta_n$ for $1\leq i\leq l$. 
\end{thm} 
\pf The proof is given in Section \ref{pf:shape}. \qed

\begin{ex}
{\rm
Let $n=6$ and let ${\bf c}\in \B^J$ be given by
\begin{equation*} 
\Tl{$\xymatrixcolsep{0pc}\xymatrixrowsep{0.3pc}\xymatrix{
& & & & 1 & &  & &    \\
& & & 2 & & 3 & & &  & \\
& & 2 & & 1 & & 1 & & \\
& 1 & & 3 & & 2 & & 1 & \\  
2 & & 3 & & 2 & & 0 & & 3}$}
\end{equation*}
Then we have
\begin{equation*} \label{eq:ex_burge2}
\Tl{$\ytableausetup{aligntableaux=center, boxsize=1.1em}
\kappa^{\se}({\bf c}) = \begin{ytableau}
\none & \none & \none & \none & \none & \none & \none & \none & \none & \none & \none & \none & \none & \none & \none & \none & \none & \ov{\tl{6}} & \ov{\tl{6}} \\
\none & \none & \none & \none & \none & \none & \none & \none & \none & \none & \none & \none & \none & \none & \none & \none & \none & \ov{\tl{5}} & \ov{\tl{5}} & \none \\
\none & \none & \none & \none & \none & \none & \none & \none & \none & \none & \none & \none & \none & \ov{\tl{6}} & \ov{\tl{6}}& \ov{\tl{5}} & \ov{\tl{4}} & \ov{\tl{4}} & \ov{\tl{4}} \\
\none & \none & \none & \none & \none & \none & \none & \none & \none & \none & \none & \none & \none & \ov{\tl{5}} & \ov{\tl{3}}& \ov{\tl{3}} & \ov{\tl{3}} & \ov{\tl{3}} & \ov{\tl{3}} & \none\\
\ov{\tl{6}} & \ov{\tl{6}} & \ov{\tl{6}} & \ov{\tl{6}} & \ov{\tl{5}} & \ov{\tl{5}} & \ov{\tl{5}} & \ov{\tl{5}} & \ov{\tl{5}} & \ov{\tl{4}} & \ov{\tl{3}} & \ov{\tl{3}} & \ov{\tl{2}} & \ov{\tl{2}} & \ov{\tl{2}} & \ov{\tl{2}} & \ov{\tl{2}} & \ov{\tl{2}} & \ov{\tl{2}} \\
\ov{\tl{5}} & \ov{\tl{5}} & \ov{\tl{5}} & \ov{\tl{4}} & \ov{\tl{4}} & \ov{\tl{4}} & \ov{\tl{4}} & \ov{\tl{4}} & \ov{\tl{3}} & \ov{\tl{2}} & \ov{\tl{1}} & \ov{\tl{1}} & \ov{\tl{1}} & \ov{\tl{1}} & \ov{\tl{1}} & \ov{\tl{1}} & \ov{\tl{1}} & \ov{\tl{1}} & \ov{\tl{1}} & \none[]
\end{ytableau}$}
\end{equation*}
where $\la(\bc)=(19,19,6,6,2,2)$.

On the other hand, the double path ${\bf p}$ at $\ep_1 + \ep_6$ given by
\begin{equation*} 
\Tl{$\xymatrixcolsep{0.3pc}\xymatrixrowsep{0.5pc}\xymatrix{
& & & & {\blue{\bf 1}}\ar@{->}[ld]\ar@{->}[rd] & &  & &    \\
& & &  {\blue{\bf 2}}\ar@{->}[ld]  & & {\blue{\bf 3}}\ar@{->}[rd] & & &  & \\
& & {\blue{\bf 2}}\ar@{->}[rd] & & {1} & &  {\blue{\bf 1}}\ar@{->}[ld] & & \\
& {1} & & {\blue{\bf 3}}\ar@{->}[ld] & & {\blue{\bf 2}}\ar@{->}[ld] &  & {1} & \\  
{2} &  & {\blue{\bf 3}} & & {\blue{\bf 2}} & & {0} & & {3}}$}
\end{equation*} has maximal value $||{\bf c}||_{{\bf p}} = 19$, and the pair of double paths ${\bf p}_{1}$ (in blue) and ${\bf p}_{2}$ (in red) at $\ep_1 + \ep_6$ and $\ep_3 + \ep_6$, respectively, given by
\begin{equation*} 
\Tl{$\xymatrixcolsep{0.3pc}\xymatrixrowsep{0.5pc}\xymatrix{
& & & & {\blue{\bf 1}}\ar@{->}[ld]\ar@{->}[rd] & &  & &    \\
& & & {\blue{\bf 2}}\ar@{->}[rd]  & & {\blue{\bf 3}}\ar@{->}[rd] & & &  & \\
& & {\red{\bf 2}}\ar@{->}[ld]\ar@{->}[rd] & & {\blue{\bf 1}}\ar@{->}[rd] & & {\blue{\bf 1}}\ar@{->}[rd] & & \\
& {\red{\bf 1}}\ar@{->}[rd] & & {\red{\bf 3}}\ar@{->}[rd] & & {\blue{\bf 2}}\ar@{->}[rd] &  & {\blue{\bf 1}}\ar@{->}[rd] & \\  
{2} &  & {\red{\bf 3}} & & {\red{\bf 2}} & & {\blue{\bf 0}} & & {\blue{\bf 3}}}$}
\end{equation*}
has maximal value $||{\bf c}||_{{\bf p}_{1}} + ||{\bf c}||_{{\bf p}_{2}} = 25$.
By Theorem \ref{thm:shape}, we have
\begin{equation*}
\lambda_1({\bf c}) = 19, \ \ \lambda_1({\bf c})+\lambda_3({\bf c})= 25, \ \ \lambda_1({\bf c})+\lambda_3({\bf c}) + \lambda_5({\bf c}) = 27,
\end{equation*} 
which implies $\lambda_3({\bf c}) = 6$, $\lambda_5({\bf c}) = 2$, 
and hence $\la(\bc)=(19,19,6,6,2,2)$.}
\end{ex}

\begin{rem}\label{rem:Greene}
{\rm Suppose that $\g$ is of type $A_n$ and $\mf l$ is of type $A_{r}\times A_{s}$ with $r+s=n-1$. The associated crystal $B(U_q(\mf u^-))$ can be realized as the set of $(r+1)\times (s+1)$ non-negative integral matrices (see \cite[Section 4.3]{K18}). For $M\in B(U_q(\mf u^-))$, let $\la=(\la_1,\la_2,\ldots)$ be the shape of the tableaux corresponding to $M$ under RSK. It is a well-known result due to Greene \cite{G} (cf. \cite{Ful}) that $\la_1+\cdots+\la_l$ is a maximal sum of entries in $M$ lying on mutually non-intersecting $l$ lattice paths on $(r+1)\times (s+1)$ array of points from northeast to southwest. A similar result when $\g$ is of type $B$, $C$ is obtained by folding crystals of type $A$ with $r=s$.
Hence, Theorem \ref{thm:shape} is a non-trivial generalization of \cite{G}  to the case of type $D$. We can also recover the result in \cite{G} by using the same argument as in Section \ref{pf:shape}.
}
\end{rem}

\begin{rem}
{\rm 
There is another reduced expression $\bi\in R(w_0)$ which gives a nice combinatorial description of $\B_{\bi}$ \cite{SST16}.
But we do not know whether we may obtain results similar to the ones in this paper using this reduced expression.
}
\end{rem}

\section{Proofs of Main Theorems}\label{sec:proof}

\subsection{Formula of Berenstein-Zelevinsky}  
Let us recall results on combinatorial formula for string parametrization of $B(\infty)$ \cite{BZ-1}, which play a crucial role in proving Theorems \ref{thm:epsilon^*_n} and  \ref{thm:shape}.


Let $\mathfrak{g}$ be a symmetrizable Kac-Moody algebra. We keep the notations in Section \ref{sec:crystal}.
For $i \in I$, let $B_i= \{\, (x)_i \,|\, x \in \mathbb{Z} \,\}$ 
be the abstract crystal given by 
${\rm wt}((x)_i) = x\alpha_i$, 
$\varepsilon_i((x)_i)=-x$, 
$\varphi_i((x)_i)=x$, 
$\varepsilon_j((x)_i)=-\infty$, $\varphi_j((x)_i)=-\infty$ for $j \neq i$ and 
$\tilde{e}_i(x)_i = (x+1)_i$, 
$\tilde{f}_i(x)_i = (x-1)_i$, 
$\tilde{e}_j(x)_i = \tilde{f}_j(x)_i = 0$ for $j \neq i$. 
It is well-known that for any $i \in I$, there is a unique embedding of crystals \cite{Kas93}
\begin{equation*}
	\Psi_i : B(\infty) \hookrightarrow B(\infty) \otimes B_i
\end{equation*} sending $b_{\infty} \mapsto b_{\infty} \otimes (0)_i$,
where $b_\infty$ is the highest weight element in $B(\infty)$. 
This embedding satisfies that for $b \in B(\infty)$, 
$\Psi_i(b) = b' \otimes (-a)_i$, where 
$a = \varepsilon_i(b^*)$ and 
$b' = \left( \tilde{e}_i^a(b^*) \right)^*$. 
Given $b\in B(\infty)$ and a sequence of indices $\bi = (i_1, \cdots, i_l)$ in $I$, consider the sequence $b_k \in B(\infty)$ and $a_k \in \mathbb{Z}_{+}$ for $1\leq k \leq l-1$ defined inductively by
\begin{equation*}
b_0 = b, \ \ \Psi_{i_k}(b_{k-1}) = b_k \otimes (-a_k)_{i_k}.
\end{equation*}
The sequence $t_{\bi}(b) = (a_l, \cdots, a_1)$ is called the {\em string of $b$ in direction $\bi$}. 
By construction, it can be reformulated by 
\begin{equation} \label{eq:epsilon}
a_k = \varepsilon_{i_k}(\tilde{e}_{i_{k-1}}^{a_{k-1}}\cdots \tilde{e}_{i_1}^{a_1}b^*),
\end{equation} 
for $2\leq k\leq l$, where $a_1 = \varepsilon_{i_1}(b^*)$.

Suppose that $\mathfrak{g}$ is of finite type. 
Let $V$ be a finite-dimensional $\mathfrak{g}$-module and $V_\la$ denote the weight space of $V$ for $\lambda \in {\rm wt}(V)$, where ${\rm wt}(V)$ is the set of weights of $V$. 
 
For $\lambda, \mu \in {\rm wt}(V)$, an {\em $\bi$-trail from $\lambda$ to $\mu$ in $V$} is a sequence of weights $\pi = (\lambda = \nu_0, \nu_1, \ldots, \nu_l = \mu)$ in ${\rm wt}(V)$ such that 

\begin{itemize}
	\item[(1)] for $1\leq k \leq l$, $\nu_{k-1} - \nu_{k} = d_k(\pi)\alpha_{i_k}$ for some  $d_k(\pi)\in\Z_+$,
	\item[(2)] $e_{i_1}^{d_1(\pi)} \cdots e_{i_l}^{d_l(\pi)}$ is a non-zero linear map from $V_{\mu}$ to $V_{\la}$.
\end{itemize} 
When $V$ is a module with a {\em minuscule} highest weight, then the condition (1) implies (2). Furthermore, if $B$ is a crystal of $V$, then we have 
$\te_{i_1}^{d_1(\pi)} \cdots \te_{i_l}^{d_l(\pi)}B_\mu = B_\la$ or 
$\tf_{i_l}^{d_l(\pi)} \cdots \tf_{i_1}^{d_1(\pi)}B_\la = B_\mu$. 


Let $\bi = (i_1, \cdots, i_N) \in R(w_0)$ given.
Let $\bi^* := (i_1^*, \cdots, i_N^*)$ and $\bi^{\rm op} := (i_N, \cdots, i_1)$, where $i \mapsto i^*$ is the involution on $I$ given by $w_0(\alpha_i)=-\alpha_{i^*}$. For $\bc\in \B_{\bi}$, 
we have by \cite[Proposition 3.3]{BZ-1} 
\begin{equation} \label{eq:star}
b_{\bi}(\textbf{c})^{*} = b_{{\bi^*}^{\rm op}}(\textbf{c}^{\rm op}),
\end{equation}
where $\bc^{\rm op} = (c^{\rm op}_k)$ is given by $c^{\rm op}_k=c_{N-k}$ for $\bc=(c_k)$.


\begin{thm} [\cite{BZ-1}, Theorem 3.7] \label{thm:BZ} 
For $\bi, \bi' \in R(w_0)$ and $\bc\in \B_{\bi}$,
let $\textbf{t} = t_{\bi}(b_{\bi'}(\textbf{c})^{\ast})$. 
Then $\textbf{t} = (t_k)$ and $\textbf{c} = (c_m)$ are related as follows : for any $k = 1, \cdots, N$ 
\begin{equation} \label{eq:BZ_formula1}
	t_k = \min_{\pi_1} \left\{ \sum_{m=1}^{N} d_m(\pi_1)c_m \right\} 
	- \min_{\pi_2} \left\{ \sum_{m=1}^{N} d_m(\pi_2)c_m \right\},
\end{equation}
where $\pi_1$ (resp. $\pi_2$) runs over $\bi'$-trails from $s_{i_1}\cdots s_{i_{k-1}}\varpi_{i_k}$ (resp. from $s_{i_1} \cdots s_{i_k}\varpi_{i_k}$) to $w_0 \varpi_{i_k}$ in the fundamental representation $V({\varpi_{i_k}})$.
\end{thm}

\begin{rem}
{\rm The string parametrization of $b\in B(\infty)$ given by \eqref{eq:epsilon} is the string parametrization of $b^\ast$ in \cite{BZ-2} (see also \cite[Remark in Section 2]{NZ}).}
\end{rem}



\subsection{Proof of Theorem \ref{thm:epsilon^*_n}} \label{pf:epsilon} 
From now on we assume that $\mathfrak{g}$ is of type $D_n$ $(n \ge 4)$ and let $\bi_0=(i_1,\ldots,i_N) \in R(w_0)$ given in \eqref{eq:rx} with $\bi^J=(i_1,\ldots,i_M)$ and $\bi_J=(i_{M+1},\ldots,i_N)$. 
We have $n^\ast = n-1$ (resp. $(n-1)^\ast =n$) when $n$ is odd, and $i^\ast = i$ otherwise. 
Put
\begin{equation} \label{eq:expression j_0}
\bj_0=(j_1,\ldots,j_N):=\bi_0^{*{\rm op}}=(i^*_N,\ldots,i^*_1).
\end{equation} 
Recall that the crystal $B(\varpi_n)$ of $V(\varpi_n)$ can be realized as
$$B(\varpi_n)=\{\,\tau=(\tau_1,\ldots,\tau_n)\,|\,\tau_k=\pm \ (1\le k\le n) \,\},$$
where ${\rm wt}(\tau)=\tfrac{1}{2}\sum_{k=1}^n \tau_k\epsilon_k$ and
\begin{equation}\label{eq:spin crystal}
\begin{split}
( \ldots \ldots, \underbrace{+,+}_{\tau_{n-1},\tau_n})\ &\stackrel{\tf_n}{\longrightarrow}\ (\ldots\ldots,-,-),\quad 
(\ldots,\underbrace{+,-}_{\tau_i,\tau_{i+1}},\ldots) \ \stackrel{\tf_i}{\longrightarrow}\ (\ldots,-,+,\ldots) ,
\end{split}
\end{equation}
$(1\leq i\leq n-1)$ with the highest weight element $(+,\ldots,+)$ \cite{KN}.
Since the spin representation $V({\varpi_n})$ is minuscule, 
any $\bi_0$-trail $\pi=(\nu_0,\ldots,\nu_N)$ in $V(\varpi_n)$ can be identified with a sequence $b_0,\ldots, b_N$ in $B(\varpi_n)$ such that ${\rm wt}(b_k)=\nu_k$ and $\tf_{i_k}^{d_k(\pi)}b_{k-1}=b_{k}$ with $d_k(\pi)=0,1$ for $1\leq k\leq N$.

\begin{lem}\label{lem:trail in levi}
There exists a unique $(i^*_N,\ldots,i^*_{M+1})=(j_1,\ldots,j_M)$-trail
from $\varpi_n-\alpha_n={\rm wt}(+,\ldots,+,-,-)$ to $\varpi_n+\alpha_0={\rm wt}(-,-,+\ldots,+)$. We denote this trail by $(\td{\nu}_0, \ldots , \td{\nu}_M)$.
\end{lem}
\pf Considering the crystal structure on $B(\varpi_n)$ \eqref{eq:spin crystal}, 
we see that  up to 2-term braid move 
$(j'_1,j'_2,\ldots, j'_{2n-4})=(n-2,n-1,n-3,n-2,\ldots,1,2)$ is the unique sequence of indices in $I$ such that $$\tf_{j'_{2n-4}}\cdots\tf_{j'_2}\tf_{j'_1}(+,\ldots,+,-,-)=(-,-,+\ldots,+).$$
On the other hand, there exists a subsequence 
$(j'_1,j'_2,\ldots, j'_{2n-4})$ of $(j_1,\ldots,j_N)$. 
Since no other subsequence gives $(n-2,n-1,n-3,n-2,\ldots,1,2)$ up to 2-braid move
by definition of $\bi_J$, $(j'_1,j'_2,\ldots, j'_{2n-4})$ determines a unique such trail. \qed

Let ${\mc T}$ be the set of $\bj_0$-trails $\pi$ from $s_n\varpi_n$ to $w_0\varpi_n$ in $V(\varpi_n)$. 
For ${\bf c}=(c_k)\in \B^J$ and $\pi\in\mc T$,
let 
\begin{equation}
||\bc||_{\pi} = (1-d_{N}(\pi))c_1 + \cdots + (1-d_{M+1}(\pi))c_M= \sum_{k=1}^{M} \left( 1 - d_{N-k+1}(\pi)\right)c_k.
\end{equation}
Recall that $c_k=c_{\beta_k}$ for $\beta_k\in \Phi^+(J)$ ($1\leq k\leq M$) with respect to the order \eqref{eq:convex} and hence \eqref{eq:convex order i_0}. Let us simply write ${\bf c}=(c_1,\ldots,c_M)$.
The following lemma plays a crucial role in the proof of Theorem \ref{thm:epsilon^*_n}.

\begin{lem}\label{lem:max}
For ${\bf c}\in \B^J$, we have
$\varepsilon_n^*(\bc)  
= \max \left\{\, ||\bc||_{\pi}\,|\, \pi\in{\mc T}\, \right\}$. 
\end{lem}
\pf Let $\bc=(c_k)\in \B^J$ given. 
Since $\varepsilon_n^*(\bc)=\varepsilon_n^*(b_{\bi_0}(\bc))$,
we have $\varepsilon_n^*(\bc)=t_1$, where 
$(t_k) = {\bf t}_{\bi_0}(b_{\bi_0}(\bc)) = {\bf t}_{\bi_0}(b_{\bi_0}(\bc)^{**})
={\bf t}_{{\bf i}_0}(b_{\bi_0^{*\rm op}}({\bf c}^{\rm op})^{*})$ by \eqref{eq:star}.

One can check that applying $\tf_{j_N}\cdots\tf_{j_{M+1}}$ to $(+,\ldots,+)$ gives a unique $\bj_0$-trail from $\varpi_n$ and $w_0\varpi_n$.
Hence by \eqref{eq:BZ_formula1}, we have 
\begin{equation*}
t_{1} = \sum_{1\leq k\leq M}c_{k} - 
\min_{\pi}
\left\{ \sum_{1\leq k\leq M}d_{N-k+1}(\pi)c_k\right\},
\end{equation*} 
where $\pi$ is a $\bj_0$-trail from $s_n\varpi_n$ to $w_0\varpi_n$. 
Hence $t_1=\max \left\{\, ||\bc||_{\pi}\,|\, \pi\in{\mc T}\, \right\}$.
\qed

For $\pi=(\nu_0,\ldots,\nu_N) \in {\mc T}$, let 
$\pi_J = (\nu_0,\ldots,\nu_M)$, $\pi^J = (\nu_{M+1},\ldots,\nu_N)$,
and
$${\mc T}'=\{\,\pi\,|\,\pi_J=(\td{\nu}_0, \ldots , \td{\nu}_M)\,\}\subset {\mc T},$$ where 
$(\td{\nu}_0, \ldots , \td{\nu}_M)$ as in Lemma \ref{lem:trail in levi}.

\begin{lem} \label{lem:T_trail}
For ${\bf c}\in \B^J$, we have
$\varepsilon_n^*(\bc)  
= \max \left\{\, ||\bc||_{\pi}\,|\, \pi\in{\mc T}'\, \right\}$. 
\end{lem}
\pf For simplicity, we assume that $n$ is even so that $w_0\varpi_n=-\varpi_n$. The proof for odd $n$ is almost identical.\
Let $\pi=(\nu_0,\ldots,\nu_N)\in {\mc T}$ given. 
It suffices to show that there exists $\pi'\in {\mc T}'$ such that
$||\bc||_{\pi} \le ||\bc||_{\pi'}$. 

If $\pi\in {\mc T}'$, then $\nu_{M}={\rm wt}(-,-,+\ldots,+)$ by Lemma \ref{lem:trail in levi}.   
Suppose that $\pi\not\in {\mc T}'$. Since $j_k\neq n$ for $1\leq k\leq M$, we have $\nu_{M}={\rm wt}(\tau)$ where $\tau=(\tau_1,\ldots,\tau_n)$ with $\tau_p=\tau_q=-$ for some $(p,q)\neq (1,2)$ and $\tau_i=+$ otherwise. 

Since $\pi\in \mc T$, there exists a subsequence $(j'_1,\ldots,j'_{N'})$ of $(j_{M+1},\ldots,j_N)$ such that 
\begin{equation*}
\tf_{j'_{N'}}\cdots\tf_{j'_2}\tf_{j'_1}\tau=
(-,\ldots,-),
\end{equation*}
the lowest weight element.
Ignoring $j'_k$ such that $-$'s in $\tau$ is moved to the left by $\tf_{j'_k}$ \eqref{eq:spin crystal}, we obtain a subsequence $(j''_1,\ldots,j''_{N''})$ of $(j'_1,\ldots,j'_{N'})$ such that
\begin{equation*}
\tf_{j''_{N''}}\cdots\tf_{j''_2}\tf_{j''_1}(+,\ldots,+)=
(+,+,-,\ldots,-).
\end{equation*}
This implies that there exists a unique $\pi'\in \mc T'$ such that for $M+1\leq k\leq N$
\begin{equation*}
d_k(\pi')=
\begin{cases}
1, & \text{if $k=j''_l$ for some $1\leq l\leq N''$},\\
0, & \text{otherwise}.
\end{cases}
\end{equation*}
Hence we have $||\bc||_{\pi} \le ||\bc||_{\pi'}$ by construction of $\pi'$.
\qed

Recall that ${\bf c}=(c_k)\in \B^J$ is by convention identified with the array, where $c_k$ is placed at the position of $\beta_k$ in $\Delta_n$ for $1\leq k\leq M$ (see Example \ref{ex:Delta_n}).

We note that if we consider the array $(j_k)$ for $M+1\leq k\leq N$, where $j_k$ is placed at the position of $\beta_{N-k+1}$ in $\Delta_n$, then the $r$-th row from the top is filled with $r$ for $1\leq r\leq n-2$ and the bottom row is filled with 
$\ldots, n-1,n,n-1,n$ from right to left.

Let $\mc D$ be the set of arrays, where either $0$ or $1$ is placed in each $r$-th row of $\Delta_n$ from the top ($1\leq r\leq n-1$) satisfying the following conditions;
\begin{itemize}
\item[(1)] the three entries in the first two rows are $0$,

\item[(2)] the number of $1$'s in each $r$-th row is $r-2$ for $3\leq r\leq n-1$,    

\item[(3)] if $r>3$ (resp. $r<n-1$) and there are two $1$'s in the $r$-th row 
such that the entries in the same row between them are zero, then there is exactly one $1$ in the $(r-1)$-th row (resp. $(r+1)$-th row) between them,

\item[(4)] the $j_k$'s corresponding to $1$'s in the $(n-1)$-th row are $n, n-1, n,\ldots$ from right to left.
\end{itemize}
We write ${\bf d}=(d_k)\in \mc D$, where $d_k$ denotes the entry at the position of $\beta_{N-k+1}$ in $\Delta_n$ for $M+1\leq k\leq N$.

\begin{ex}{\rm
When $n=6$,  we have
\begin{equation*}
\begin{split} 
\xymatrixcolsep{-0.4pc}\xymatrixrowsep{0.3pc}
\xymatrix{
& & & & & c_{5} & & & & \\
& & & & c_{4} & & c_{9} & & & \\
\bc=\ & & & c_{3} & & c_{8} & & c_{12} & & \\  
& & c_{2} & & c_{7} & & c_{11} & & c_{14}&\\
& c_{1} & & c_{6} & & c_{10} & & c_{13} & & \!\!c_{15}}\quad\quad
\xymatrixcolsep{-0.6pc}\xymatrixrowsep{0.1pc}
\xymatrix{
& & & & & d_{26} & & & & \\
& & & & d_{27} & & d_{22} & & & \\
{\bf d}=& & & d_{28} & & d_{23} & & d_{19} & & \\  
& & d_{29} & & d_{24} & & d_{20} & & d_{17}&\\
& d_{30} & & d_{25} & & d_{21} & & d_{18} & & d_{16}}
\end{split}
\end{equation*}
\begin{equation*}
\begin{split} 
\xymatrixcolsep{-0.6pc}\xymatrixrowsep{0.1pc}
\xymatrix{
& & & & j_{26} & & & & \\
& & & j_{27} & & j_{22} & & & \\
& & j_{28} & & j_{23} & & j_{19} & & \\  
& j_{29} & & j_{24} & & j_{20} & & j_{17}&\\
j_{30} & & j_{25} & & j_{21} & & j_{18} & & j_{16}} \quad
\xymatrixcolsep{-0.6pc}\xymatrixrowsep{0.1pc}
\xymatrix{
\\
\\
\\  
\\
=} \quad
\xymatrixcolsep{0.15pc}\xymatrixrowsep{0.35pc}
\xymatrix{
& & & & 1 & & & & \\
& & & 2 & & 2 & & \\
& & 3 & & 3 & & 3 & \\  
& 4 & & 4 & & 4 & & 4 \\
6 & & 5 & & 6 & & 5 & & 6}
\end{split}
\end{equation*}
}
\end{ex}

For $\pi\in \mc{T}'$, let ${\bf d}(\pi)$ denote the array $(d_k(\pi))$, where $d_k(\pi)$ is placed in the position of $\beta_{N-k+1}$ in $\Delta_n$ for $M+1\leq k\leq N$.

\begin{lem}\label{lem:bijection from T to D}
The map sending $\pi$ to ${\bf d}(\pi)$ is a bijection from $\mc T'$ to $\mc D$.
\end{lem}
\pf Let us assume that $n$ is even since the proof for odd $n$ is the same. 
We first show that the map is well-defined.
Let $\pi=(\nu_0,\ldots,\nu_N)\in \mc T'$ given, where $\nu_{M}={\rm wt}(-,-,+\ldots,+)$. Let $(j'_1,\ldots,j'_L)$ be the subsequence of $(j_{M+1},\ldots, j_{N})$ such that $d_{j'_k}(\pi)=1$. Then $$\tf_{j'_{L}}\cdots\tf_{j'_2}\tf_{j'_1}(-,-,+\ldots,+)=(-,-,-,\ldots,-),$$ equivalently,
\begin{equation}\label{eq:aux-5.1}
\tf_{j'_{L}}\cdots\tf_{j'_2}\tf_{j'_1}(+,+,+\ldots,+)=(+,+,-,\ldots,-).
\end{equation} 
From \eqref{eq:spin crystal}, \eqref{eq:aux-5.1} and the array $(j_k)$ on $\Delta_n$, one can check that 
(i) $L=(n-3)(n-2)/2$, 
(ii) $3\leq j'_k\leq n$,
(iii) the array ${\bf d}(\pi)$ satisfies the conditions (1) and (2) for $\mc D$.
To verify the condition (3), let us enumerate $-$'s appearing in \eqref{eq:aux-5.1} from left to right 
by $-_1, -_2,\ldots$. 

For $3\leq r\leq n-1$, let $1_{(r-2,r)},\ldots,1_{(2,r)}, 1_{(1,r)}$ denote the entries $1$ of ${\bf d}(\pi)$ in the $r$-th row, which are enumerated from the right.

For $1\leq k\leq L$, suppose that $j'_k$ corresponds to $1_{(s,r)}$ in ${\bf d}(\pi)$ for some $s$ with $r=j'_k$.
It is not difficult to see that $\tf_{j'_k}$ in \eqref{eq:aux-5.1} corresponds to
\begin{itemize}
\item[(1)] moving $-_s$ at the $(r+1)$-th coordinate of a vector in $B(\varpi_n)$ to the $r$-th one unless $r=n$ and $s$ is odd,

\item[(2)] placing $(-_{s},-_{s+1})$ at the last two coordinates if $r=n$ and $s$ is odd.
\end{itemize}
Then by looking at the arrangement of $d_k(\pi)$'s in $\Delta_n$, it follows that $1_{(s,r)}$ is located to the northeast of $1_{(s+1,r+1)}$ and to the northwest of $1_{(s,r+1)}$ for $r<n-1$, 
\begin{equation*}
\xymatrixcolsep{-1.5pc}\xymatrixrowsep{0.5pc}\xymatrix{
& & & \cdots & & \quad\quad\quad\quad 1_{(s,r)} & \cdots & &  & \\
1_{(s+1,r+1)}\quad &  & & & & & \quad\quad\quad 1_{(s,r+1)}  }
\end{equation*}
and the $j'_k$'s corresponding to $\ldots, 1_{(3,n-1)}, 1_{(2,n-1)}, 1_{(1,n-1)}$ are $\ldots, n, n-1,n$. 
Hence ${\bf d}(\pi)$ satisfies the condition (3) and (4) for $\mc D$, and the map $\pi\mapsto {\bf d}(\pi)$ is well-defined. 

Since the map is clearly injective, it remains to show that it is surjective. 
Let $\pi_0\in {\mc T}'$ be a unique trail such that $d_k(\pi_0)_k=1$ for $M+1\leq k\leq M+L$ and $0$ otherwise. 

We claim that for any ${\bf d}\in \mc D$ there exists a sequence ${\bf d}={\bf d}_0,{\bf d}_1,\ldots,{\bf d}_m={\bf d}(\pi_0)$ in $\mc D$ such that 
${\bf d}_{l+1}$ is obtained from ${\bf d}_l$ by moving an entry $1$ to the right.  If ${\bf d}\neq {\bf d}(\pi_0)$, then choose a minimal $k$ such that $d_k=0\neq d_k(\pi_0)$ for $M+1\leq k\leq M+L$. 
If $1_{(s,r)}$ denotes an entry corresponding to $d_k(\pi_0)$ in ${\bf d}(\pi)$, then there exists $1_{(s',r)}$ in ${\bf d}$ such that $s<s'$. Here we assume that $s'$ is minimal. 
Then by the condition (3) for $\mc D$ and the minimality of $s'$, we can move $1_{(s',r)}$ to the right by one position if $r<n-1$ and by two positions if $r=n-1$ to get another ${\bf d}'\in \mc D$ by definition of $\mc D$. Repeating this step, we obtain a required sequence. This proves our claim.

Now, let ${\bf d}=(d_k)\in \mc D$ be given and let $(j'_1,\ldots,j'_L)$ be the subsequence such that $d_{j'_k}=1$. By the above claim and definition of $\mc D$, we obtain the following two  reduced expressions;
\begin{equation*}
s_{j'_L}\cdots s_{j'_2}s_{j'_1} = s_{j_{M+L}}\cdots s_{j_{M+2}}s_{j_{M+1}}, 
\end{equation*}
where we obtain the right-hand side from the left only by applying 2-term braid move. 
Since 
$\tf_{j_{M+L}}\cdots\tf_{j_{M+2}}\tf_{j_{M+1}}(+,+,+\ldots,+)=(+,+,-,\ldots,-)$, 
we obtain \eqref{eq:aux-5.1}, 
which implies that  there exists $\pi\in \mc T'$ such that ${\bf d}(\pi)={\bf d}$. The proof completes.
\qed


Let $\mc P$ be the set of double paths at $\theta$. 
Consider two operations on ${\mc P}$ which change a part of ${\bf p}\in \mc P$ in the following way;
{\allowdisplaybreaks \begin{align}
&\quad\quad\quad \xymatrixcolsep{0.3pc}\xymatrixrowsep{0.4pc}
\xymatrix{
  &  {\bullet}\ar@{->}[ld]  &   \\
   {\bullet}\ar@{->}[rd] &  &  {\bullet}   \\  
  &  {\bullet} &   }
\quad
\xymatrix{
 \\
\longrightarrow \\  
}\quad
\xymatrix{
  &  {\bullet}\ar@{->}[rd]  &   \\
   {\bullet} &  &  {\bullet}\ar@{->}[ld]   \\  
  &  {\bullet} &   } \label{eq:aux2-5.1}\\  
&\xymatrixcolsep{0.3pc}\xymatrixrowsep{0.4pc}
\xymatrix{
\\
   &  {\bullet}\ar@{->}[ld] &  &  {\bullet}\ar@{->}[ld] & \\  
    {\bullet} & &  {\bullet} & &  {\bullet}}
\quad
\xymatrix{
 \\
\longrightarrow \\  
}\quad
\xymatrix{\\
   &  {\bullet}\ar@{->}[rd] &  &  {\bullet}\ar@{->}[rd] & \\  
    {\bullet} & &  {\bullet} & &  {\bullet}} \label{eq:aux3-5.1}
\end{align}}
where in \eqref{eq:aux3-5.1} the rows denote the two rows from the bottom in $\Delta_n$.

\begin{lem}\label{lem:bijection from P to D}
For $\bf p\in \mc P$, let ${\bf d}({\bf p})=(d_k)\in \mc D$ be given by $d_k=0$ if $\bf p$ passes the position of $d_k$, and $d_k=1$ otherwise.
Then the map sending ${\bf p}$ to ${\bf d}({\bf p})$ is a bijection from $\mc P$ to $\mc D$.
\end{lem}
\pf  Let ${\bf p}_0\in \mc P$ be a unique double path at $\theta$ such that ${\bf p}_0$ ends at the first two dots from the left in the bottom row of $\Delta_n$, that is, $\beta_1$ and $\beta_{n}$ (see the first double path in Example \ref{ex:double path}). It is clear that ${\bf d}({\bf p}_0)={\bf d}(\pi_0)\in \mc D$, where $\pi_0\in {\mc T}'$ is given in the proof of Lemma \ref{lem:bijection from T to D}. 

Let ${\bf p}\in \mc P$ given. 
Suppose that ${\bf p}'$ is obtained from ${\bf p}$ 
by applying either \eqref{eq:aux2-5.1} or \eqref{eq:aux3-5.1}.
If ${\bf d}({\bf p})\in \mc D$, then it is clear that ${\bf d}({\bf p}')\in \mc D$.
Since one can obtain ${\bf p}$ from ${\bf p}_0$ by applying \eqref{eq:aux2-5.1} and \eqref{eq:aux3-5.1} a finite number of times, we have ${\bf d}({\bf p})\in \mc D$.
Hence the map ${\bf p}\mapsto {\bf d}({\bf p})$ is well-defined and injective. The surjectivity follows from the fact that any ${\bf d}\in \mc D$ can be obtained from ${\bf d}(\pi_0)$ by moving an entry to the left by one or two depending on the row which it belongs to  (see the proof of Lemma \ref{lem:bijection from T to D}), which corresponds to \eqref{eq:aux2-5.1} or \eqref{eq:aux3-5.1}.
\qed

\begin{proof} [Proof of Theorem \ref{thm:epsilon^*_n}]
By Lemmas \ref{lem:bijection from T to D} and \ref{lem:bijection from P to D}, 
there exists a bijection from $\mc T'$ to $\mc P$. If $\pi\in \mc T'$ corresponds to $\bf p\in \mc P$, then we have  
$||{\bf c}||_{\pi}= ||{\bf c}||_{\bf p}$ for ${\bf c}\in \B^J$. 
Hence by Lemma \ref{lem:T_trail}, we have $\varepsilon_n^*(\bc)  
= \max \left\{\, ||\bc||_{\bf p}\,|\, \bf p\in{\mc P}\, \right\}$.
\end{proof}

\subsection{Proof of Theorem \ref{thm:shape}} \label{pf:shape}
We generalize the arguments in Theorem \ref{thm:epsilon^*_n}. We keep the notations in Section \ref{pf:epsilon}. Let $\bj_0$ be as in \eqref{eq:expression j_0}. For $1 \le l \le [\frac{n}{2}]$, let $k_l$ be the index such that ${j_{k_l}}$ belongs to the subword $({\bi_{2l-1}^J})^{\rm *op}$ of $\bj_0$ and $j_{k_l}=n$.
For $i\in I$ and an element $b$ of a crystal, let $\te^{max}_ib = \te_i^{\varepsilon_i(b)}b$.
The following is crucial when proving Theorem \ref{thm:shape}.

\begin{prop} \label{prop:epsilon_shape} 
For ${\bf c} = (c_k) \in \B^J$ and $1 \le l \le [\frac{n}{2}]$,  
\begin{equation} \label{eq:kth_length}
\lambda_{2l-1}({\bf c}) = 
\varepsilon_{j_{k_l}}
\left( \tilde{e}_{j_{{k_l}-1}}^{max}\cdots \tilde{e}_{j_2}^{max}\tilde{e}_{j_1}^{max}{\bf c} \right),
\end{equation} 
and it is equal to
\begin{equation}\label{eq:kth_realization}
\min_{\pi_1}\left\{ \sum_{k=1}^{M} d_k(\pi_1)c_k \right\} 
- 
\min_{\pi_2}\left\{ \sum_{k=1}^{M}d_k(\pi_2)c_k \right\},
\end{equation} where $\pi_1$ and $\pi_2$ are ${\bf i}_0$-trails from 
\begin{equation*}
\begin{split}
{\rm wt}(\underbrace{-, \cdots, -}_{2l-2}, \underbrace{+, \cdots, +}_{n-2l+2}) \quad \textrm{and} \quad  {\rm wt}(\underbrace{-, \cdots, -}_{2l}, \underbrace{+, \cdots, +}_{n-2l})
\end{split}
\end{equation*} 
to the lowest weight element in $B(\varpi_n)$, respectively.
\end{prop}
\pf Let $\bc\in \B^J$ given.
Since $(j_1,\ldots,j_M)$ is a reduced expression of the longest element for $\mf l$, $\tilde{e}_{j_{M}}^{max}\cdots \tilde{e}_{j_1}^{max}{\bf c}$ is an $\mf l$-highest weight element, which is of the form \eqref{eq:characterization of l h.w.}. 
Then it is straightforward to verify \eqref{eq:kth_length} using Proposition \ref{prop:signature for type D}.

On the other hand, the righthand side of \eqref{eq:kth_length} can be obtained by 
\eqref{eq:BZ_formula1} letting
\begin{equation} \label{eq:setting}
{\bf i} = {\bf j}_0, \ \ {\bf i}' = {\bf i}_0,
\end{equation}
where in this case
\begin{equation}\label{eq:head of trail}
\begin{split}
& s_{j_1} \cdots s_{j_{{k_l}-1}}(\varpi_{j_{k_l}}) = s_{j_1} \cdots s_{j_{{k_l}-1}}(\varpi_n)
= {\rm wt}(\underbrace{-, \cdots, -}_{2l-2}, \underbrace{+, \cdots, +}_{n-2l+2}),\\
& s_{j_1} \cdots s_{j_{k_l}}(\varpi_{j_{k_l}}) = {\rm wt}(\underbrace{-, \cdots, -}_{2l}, \underbrace{+, \cdots, +}_{n-2l}).
\end{split}
\end{equation}
Hence the formula \eqref{eq:BZ_formula1} gives \eqref{eq:kth_realization}.
\qed

Let $1 \le l \le [\frac{n}{2}]$ given.
Let ${\mc T}_{l}$ be the set of ${\bi_0}$-trails 
from $s_{j_1} \cdots s_{j_{k_l}}(\varpi_n)$ \eqref{eq:head of trail} to $w_0\varpi_n$.
Let ${\mc D}_l$ be the set of arrays where either $0$ or $1$ is placed in each $r$-th row of $\Delta_n$ from the top ($1 \le r \le n-1$) satisfying the following conditions;
\begin{itemize}
	\item[(1)] the entries in the first $2l$ rows are $0$,
	
	\item[(2)] the number of $1$'s in each $r$-th row is $r-2l$ for $2l+1 \le r \le n-1$,
	
	\item[(3)] if $r>2l+1$ (resp. $r<n-1$) and there are two $1$'s in the $r$-th row such that the entries in the same row between them are zero, then there is exactly one $1$ in the $(r-1)$-th row (resp. $(r+1)$-th row) between them,
	\item[(4)] the $j_k$'s corresponding to $1$'s in the $(n-1)$-th row are $n, n-1, n,\ldots$ from left to right.
\end{itemize}
Note that ${\mc D}_1={\mc D}$.

\begin{lem}\label{lem:bijection from T to D general} 
For $\pi\in {\mc T}_l$, the map sending $\pi$ to ${\bf d}(\pi)$ is a bijection from ${\mc T}_l$ to ${\mc D}_l$, where ${\bf d}(\pi)$ is defined in the same way as in $\mc T'$.
\end{lem}
\pf It can be shown by almost the same arguments as in Lemma \ref{lem:bijection from T to D} that the map is well-defined, and clearly injective.

It suffices to show that it is surjective.
Let $\Delta'_{n-2l}$ be the set $\Delta_{n-2l}$, which we regard as a subset of $\Delta_n$ sharing the same southwest corner with $\Delta_n$.
Let $\pi_0$ be a unique trail in ${\mc T}_l$ such that $d_k(\pi_0)=1$ if and only if $d_k(\pi_0)$ is located in $\Delta'_{n-2l}$. 
Then as in the proof of Lemma \ref{lem:bijection from T to D} we can check that for any ${\bf d}\in \mc D_l$ there exists a sequence ${\bf d}={\bf d}_0,{\bf d}_1,\ldots,{\bf d}_m={\bf d}(\pi_0)$ in $\mc D_l$ such that 
${\bf d}_{l+1}$ is obtained from ${\bf d}_l$ by moving an entry $1$ to the {\em left}, and hence that there exists $\pi\in \mc T_l$ such that ${\bf d}(\pi)={\bf d}$.
\qed

Let ${\mc P}_l$ be the set of $l$-tuple ${\bf p} = ({\bf p}_1, \cdots, {\bf p}_{l})$ of mutually non-intersecting double paths in $\Delta_n$ such that each ${\bf p}_{i}$ is a double path at some point in the $(2i-1)$-th row.

\begin{lem} \label{lem:bijection from P to D general} 
The map sending ${\bf p}$ to ${\bf d}({\bf p})$ is a surjective map from ${\mc P}_l$ to ${\mc D}_l$, where ${\bf d}({\bf p})$ is defined in the same way as in ${\mc P}$.
\end{lem}
\pf Suppose that ${\underline{\bf p}} = ({\bf p}_1, \cdots, {\bf p}_{l})\in {\mc P}_l$ is given. 
By definition of ${\mc P}_l$, one can check that all the points in the first $2l$ rows in $\Delta_n$ are occupied by $\underline{\bf p}$.

Let ${\underline{\bf p}^0=({\bf p}^0_1,\ldots,{\bf p}^0_l)}$ be given such that ${\bf p}^0_i$ starts at $\ep_{2i-1}+\ep_{n}$ and ends at $\ep_{2i-1}+\ep_{2i}$ and $\ep_{2i}+\ep_{2i+1}$ for $1\leq i\leq r$.
We have ${\bf d}(\underline{\bf p}^0)={\bf d}(\pi_0)$, where $\pi_0$ is given in the proof of Lemma \ref{lem:bijection from T to D general}. 
Applying the operations \eqref{eq:aux2-5.1} and \eqref{eq:aux3-5.1} on ${\mc P}_l$,
one can obtain a sequence in ${\mc P}_l$ {\em from $\underline{\bf p}$ to $\underline{\bf p}^0$}, whose image under ${\bf d}$ lies in $\mc D_l$. Then similar arguments as in Lemma \ref{lem:bijection from P to D} implies the surjectivity.
\qed

\begin{proof} [Proof of Theorem \ref{thm:shape}] 
Let ${\bf c} \in \B^J$ given. For $\pi\in \mc T_l$, there exists $\underline{\bf p}= ({\bf p}_1, \cdots, {\bf p}_{l})\in \mc P_l$ such that ${\bf d}(\underline{\bf p})={\bf d}(\pi)$ by Lemmas \ref{lem:bijection from T to D general}, and \ref{lem:bijection from P to D general}, and
\begin{equation}\label{eq:aux 5.3}
\sum_{1\leq k\leq M}d_k(\pi)c_k = 
\sum_{1\leq k\leq M}c_k - \left(||\bc||_{{\bf p}_1}+\cdots+||\bc||_{{\bf p}_l}\right).  
\end{equation}
Indeed, \eqref{eq:aux 5.3} holds for any $\underline{\bf p}\in \mc P_l$ such that ${\bf d}(\underline{\bf p})={\bf d}(\pi)$.
Therefore, we have by \eqref{eq:kth_realization} and \eqref{eq:aux 5.3}
{\allowdisplaybreaks
\begin{align*}
&\lambda_{2l-1}(\textbf{c}) \\
&= 
\min_{\pi_1\in \mc T_{l-1}}
\left\{\, \sum_{1\leq k\leq M}d_k(\pi_1)c_k\,\right\} - 
\min_{\pi_2\in \mc T_{l}}
\left\{\, \sum_{1\leq k\leq M}d_k(\pi_2)c_k\,\right\} \\
&= 
\left(\sum_{1\leq k\leq M}c_k - 
\max_{\underline{\bf p}\in {\mc P}_{l-1}}
\{\,||{\bf c}||_{{\bf p}_1}+\cdots+||{\bf c}||_{{\bf p}_{l-1}}\,\}\right)\\ 
& \hskip 4cm -\left(\sum_{1\leq k\leq M}c_k - 
\max_{\underline{\bf p}\in {\mc P}_{l}}
\{\,||{\bf c}||_{{\bf p}_1}+\cdots+||{\bf c}||_{{\bf p}_l}\,\}\right) \\
&= 
\max_{\underline{\bf p}\in {\mc P}_{l}}
\{\,||{\bf c}||_{{\bf p}_1}+\cdots+||{\bf c}||_{{\bf p}_l}\,\}
-
\max_{\underline{\bf p}\in {\mc P}_{l-1}}
\{\,||{\bf c}||_{{\bf p}_1}+\cdots+||{\bf c}||_{{\bf p}_{l-1}}\,\}.
\end{align*}} 
This gives the formula in Theorem \ref{thm:shape}.
\end{proof}

\subsection{Proof of Theorem \ref{thm:isomorphism theorem}} \label{pf:burge} 
We keep the notations in Section \ref{subsec:burge}.
We assume that $\Delta_i\subset \Delta_n$ for $1\leq i\leq n$, where both of $\Delta_i$ and $\Delta_n$ share the same southeast corner.
For $\textbf{c}=(c_k)\in \B^J$, 
let $\textbf{c}_{\Delta_i}\in \B^J$ (resp. $\textbf{c}_{\Delta_i^c}\in \B^J$) whose component in $\Delta_i$ (resp. $\Delta_i^c:=\Delta_n\setminus\Delta_i$) is $c_k$ and $0$ elsewhere. Let 
\begin{equation*}
\B^J_{\Delta_i} = \{\, \textbf{c} \in \B^J \,|\, \textbf{c}_{\Delta_i^c} = \textbf{0} \,\}.
\end{equation*}

Fix $i \in I\setminus\{n\}$.
Let $\textbf{c} \in \B^J_{\Delta_{i+1}}$ given with $\bc=\bc(\ba,\bb)$ for a unique $(\ba,\bb)=(a_1\ldots a_r,b_1\ldots b_r)\in \Omega$. 
We divide $(\ba,\bb)$ into two biwords $(\ba',\bb')=(a_1\ldots a_s,b_1\ldots b_s)$ and 
$(\ba'',\bb'')=(a_{s+1}\ldots a_r,b_{s+1}\ldots b_r)$ where
$a_s\leq \ov{i}$ and $a_{s+1}>\ov{i}$ so that
\begin{equation}\label{eq:biword(a,b)}
\biggr(\begin{array}{c} \ba' \\ \bb' \end{array}\biggr)
=
\biggr({\begin{array}{c} \ov{i+1} \\ \ov{1} \end{array}}^{c_{\ov{i+1}\ov{1}}} \cdots 
{\begin{array}{c} \ov{i+1} \\ \ov{i} \end{array}}^{c_{\ov{i+1}\ov{i}}}\ 
{\begin{array}{c} \ov{i} \\ \ov{1} \end{array}}^{c_{\ov{i}\ov{1}}} \cdots 
{\begin{array}{c} \ov{i} \\ \ov{i-1} \end{array}}^{c_{\ov{i}\ov{i-1}}} \biggr).
\end{equation}
Here the superscript means the multiplicity of each biletter.

Let ${\bf c}'={\bf c}(\ba',\bb')$ and 
${\bf c}''={\bf c}(\ba'',\bb'')\in \B^J_{\Delta_{i-1}}$ be the corresponding Lusztig data.
Put $T=\kappa^{\se}(\bc'')$. Then ${\rm sh}(T)=\mu^\pi$ for some $\mu\in \cP_{i-1}$ such that $\mu'$ is even. 
We define $({\tt P}(\bc),{\tt Q}(\bc))$ by
\begin{itemize}
\item[(1)] ${\tt P}(\bc) = ( (T\leftarrow b_{s}) \leftarrow \cdots \leftarrow b_1)$,

\item[(2)] ${\tt Q}(\bc)\in SST((\la/\mu)^\pi)$, where ${\rm sh}({\tt P}_k)/{\rm sh}({\tt P}_{k-1})$ is filled with $a_k$ for $1\leq k\leq s$. 
\end{itemize}
Here  
$\la={\rm sh}({\tt P}(\bc))^\pi$, ${\tt P}_k=( (T\leftarrow b_{k}) \leftarrow \cdots \leftarrow b_1)$, and ${\tt P}_0=T$.
The pair $({\tt P}(\bc),{\tt Q}(\bc))$ can be viewed as a skew-analogue of RSK correspondence applying insertion of $(\ba',\bb')$ into $T$ (cf. \cite[Proposition 1 in Section 5.1]{Ful} and \cite{SS90}).

For $\bc = ( c_{\ov{j}\,\ov{k}} ) \in {\bf B}_{\Delta_{i+1}}^J$, we define
${\bf c}^\circ = ( c_{\ov{j}\,\ov{k}}^\circ )$ by 
$$
	c_{\ov{j}\,\ov{k}}^\circ = 
	\begin{cases}
		c_{\ov{j}\,\ov{k}} & \text{if $(\ov{j}, \ov{k}) \neq (\ov{i+1}, \ov{i})$}, \\ 
		0 & \text{otherwise.}
	\end{cases}
$$
Then we identify ${\bf c}$ with $(c_{\ov{i+1}\,\ov{i}}) \otimes {\bf c}^\circ$ as a crystal element.
Set $(0) \mapsto ({\tt P}((0)), {\tt Q}((0))) = (\emptyset, \emptyset)$, where $\emptyset$ is understood as the empty tableau. Note that ${\tt Q}({\bf c}) = \emptyset$ if $c_{\ov{i}\,\ov{k}} = c_{\ov{i+1}\,\ov{l}} = 0$ for all $k, l$.

\begin{lem}\label{lem:Q commutes with f_k}
Let $\bc \in \B^J_{\Delta_{i+1}}$. Suppose that $c_{\ov{i+1}\,\ov{i}} = 0$. Then we have
\begin{equation*}
{\tt Q}(\tf_i\bc) =\tf_i{\tt Q}(\bc).
\end{equation*}
Furthermore, if $\tf_i \bc \neq {\bf 0}$, then ${\tt P}(\tf_i \bc) = {\tt P}(\bc)$.
\end{lem}
\pf Considering the action of $\tf_i$ on the subcrystal $\B^J_{\Delta_{i+1}}$ of $\B^J$ described in Proposition \ref{prop:signature for type D}, we may apply \cite[Proposition 4.6 and Remark 4.8(1)]{K09} to have ${\tt Q}(\tf_i\bc) =\tf_i{\tt Q}(\bc)$, and ${\tt P}(\tf_i \bc) = {\tt P}(\bc)$ if $\tf_i \bc^\circ \neq {\bf 0}$ (because we can naturally identify each element of $\B^J_{\Delta_{i+1}}$ with an element of the crystal $\mathcal{M}$ in \cite[Section 3]{K09}).
\qed 

Let $\ell(\la)=2m$ for some $m\geq 1$.
For $1\leq l\leq m$, let $V_l$ be the subtableau of ${\tt Q}(\bc)$ lying in the $(2l-1)$-th and $2l$-th rows from the bottom, and let $U_l$ be the subtableau of ${\tt P}(\bc)$ corresponding to $V_l$. 
Note that $U_l$ and $V_l$ are of anti-normal shapes, and $V_l^{\nw}$ is the tableau of normal shape obtained from $V_l$ by jeu de taquin to the northwest corner. We also let ${\tt P}(\bc)_{l}$ be the subtableau of ${\tt P}(\bc)$ lying above the $(2l-2)$-th row from the bottom, where ${\tt P}(\bc)_{1}={\tt P}(\bc)$.

Now we {\em glue} each $V_l^{\nw}$ to ${\tt P}(\bc)$ to define a tableau ${\tt T}(\bc)$ by the following inductive algorithm; 
\begin{itemize}
	\item[(\textbf{g}-1)] Let ${\tt T}(\bc)_m$ be the tableau obtained from ${\tt P}(\bc)_m$ by gluing $V_m^{\nw}$ to $U_m$ (so that $V_m^{\nw}$ and $U_m$ form a two-rowed rectangle $W_m$).
	\item[(\textbf{g}-2)] Consider a tableau obtained from ${\tt P}(\bc)_{m-1}$ by gluing $V_{m-1}^{\nw}$ to $U_{m-1}$ and replacing ${\tt P}(\bc)_m$ by ${\tt T}(\bc)_m$. If the number of columns in $W_m$ is greater than $\mu_{2m-3}-\mu_{2m-1}$ for $m \ge 2$, then we move dominos
\raisebox{-.6ex}{{\tiny ${\def\lr#1{\multicolumn{1}{|@{\hspace{.6ex}}c@{\hspace{.6ex}}|}{\raisebox{-.3ex}{$#1$}}}\raisebox{-.6ex}
{$\begin{array}[b]{c}
\cline{1-1}
\lr{ \!\ov{i\!+\!1}\!}\\
\cline{1-1}
\lr{ \ov{i}}\\
\cline{1-1}
\end{array}$}}$}}
down to the next two rows as many as the difference, and denote the resulting tableau by ${\tt T}(\bc)_{m-1}$. Also denote by $W_{m-1}$ the new two-rowed rectangle in the first two rows of $\ov{\tt T}(\bc)_{m-1}$ from the bottom.

	\item[(\textbf{g}-3)] Repeat (\textbf{g}-2) to have ${\tt T}(\bc)_{m-2},\ldots, {\tt T}(\bc)_1$, and let ${\tt T}(\bc)={\tt T}(\bc)_1$. 
\end{itemize}

Let us also introduce an inductive algorithm to define a tableau $\ov{\tt T}(\bc)$, which is slightly different from the above algorithm;
\begin{itemize}
	\item[($\ov{\text{\textbf{g}}}$-1)] Let $\ov{\tt T}(\bc)_m$ be the tableau obtained from ${\tt P}(\bc^\circ)_m$ by gluing $V_m^{\nw}$ to $U_m$, and $c_{\ov{i+1}\,\ov{i}}$ copies of\,
\raisebox{-.6ex}{{\tiny ${\def\lr#1{\multicolumn{1}{|@{\hspace{.6ex}}c@{\hspace{.6ex}}|}{\raisebox{-.3ex}{$#1$}}}\raisebox{-.6ex}
{$\begin{array}[b]{c}
\cline{1-1}
\lr{ \!\ov{i\!+\!1}\!}\\
\cline{1-1}
\lr{ \ov{i}}\\
\cline{1-1}
\end{array}$}}$}}\,
to the left of the resulting tableau so that it forms a two-rowed rectangle $W_m$.
	\smallskip
	
	\item[($\ov{\text{\textbf{g}}}$-2)] Consider a tableau obtained from ${\tt P}(\bc^\circ)_{m-1}$ by gluing $V_{m-1}^{\nw}$ to $U_{m-1}$ and replacing ${\tt P}(\bc^\circ)_m$ by $\ov{\tt T}(\bc)_m$. If the number of columns in $W_m$ is greater than $\mu_{2m-3}-\mu_{2m-1}$ for $m \ge 2$, then we move dominos
\raisebox{-.6ex}{{\tiny ${\def\lr#1{\multicolumn{1}{|@{\hspace{.6ex}}c@{\hspace{.6ex}}|}{\raisebox{-.3ex}{$#1$}}}\raisebox{-.6ex}
{$\begin{array}[b]{c}
\cline{1-1}
\lr{ \!\ov{i\!+\!1}\!}\\
\cline{1-1}
\lr{ \ov{i}}\\
\cline{1-1}
\end{array}$}}$}}
down to the next two rows as many as the difference, and denote the resulting tableau by $\ov{\tt T}(\bc)_{m-1}$. Also denote by $W_{m-1}$ the new two-rowed rectangle in the first two rows of $\ov{\tt T}(\bc)_{m-1}$ from the bottom.
	\smallskip
	
	\item[($\ov{\text{\textbf{g}}}$-3)] Repeat ($\ov{\text{\textbf{g}}}$-2) to have $\ov{\tt T}(\bc)_{m-2},\ldots, \ov{\tt T}(\bc)_1$, and let $\ov{\tt T}(\bc)=\ov{\tt T}(\bc)_1$. 
\end{itemize}
Here $\lambda = {\rm sh}\left( {\tt P}(\bc^\circ) \right)^\pi$ and let $m = \lfloor \frac{\ell(\lambda)}{2}  \rfloor + 1$. Note that the algorithm above use $({\tt P}(\bc^\circ), {\tt Q}(\bc^\circ))$, rather than $({\tt P}(\bc), {\tt Q}(\bc))$.

\begin{ex} \label{ex:algorithm}
{\rm  Suppose that $n=6$ and $i=4$. 
Let ${\bf c} \in \B^J_{\Delta_5}$ be given by
\begin{equation*}
\Tl{$\xymatrixcolsep{0pc}\xymatrixrowsep{0.3pc}\xymatrix{
& & & & 0 & &  & &    \\
& & & 0 & & {\bf 3} & & &  & \\
& & 0 & & {\bf 1} & & {\bf 1} & & \\
& 0 & & {\bf 3} & & {\bf 2} & & 1 & \\  
0 & & {\bf 3} & & {\bf 2} & & 0 & & 3}$}
\end{equation*}
where ${\bf c}'$ is given by the entries in bold letters.
Then
\begin{equation*}
\Tl{$\ytableausetup{aligntableaux=center, boxsize=1.1em}
T = \kappa^{\se}({\bf c}'') = 
\begin{ytableau}
\ov{\tl 3} & \ov{\tl 2} & \ov{\tl 2} & \ov{\tl 2} \\
\ov{\tl 1} & \ov{\tl 1} & \ov{\tl 1} & \ov{\tl 1}
\end{ytableau}\ , \quad 
\biggr(\begin{array}{c} \ba' \\ \bb' \end{array}\biggr)
=
\biggr(\ {\begin{array}{c} \ov{5} \\ \ov{1} \end{array}}^{3}
{\begin{array}{c} \ov{5} \\ \ov{2} \end{array}}^{1}
{\begin{array}{c} \ov{5} \\ \ov{3} \end{array}}^{3}
{\begin{array}{c} \ov{5} \\ \ov{4} \end{array}}^{3}
{\begin{array}{c} \ov{4} \\ \ov{1} \end{array}}^{1} 
{\begin{array}{c} \ov{4} \\ \ov{2} \end{array}}^{2} 
{\begin{array}{c} \ov{4} \\ \ov{3} \end{array}}^{2} \ \biggr).$}
\end{equation*}
By definition, the pair $({\tt P}(\bc),{\tt Q}(\bc))$ is given by
\begin{equation*}
\ytableausetup{aligntableaux=center, boxsize=1.1em}
\Tl{${\tt P}(\bc) = 
\begin{ytableau}
\none &\none &\none &\none &\none &\none &\none &\none & \red{\ov{\tl{4}}} & \red{\ov{\tl{4}}} \\
\none & \none & \none & \none & \none & \none & \red{\ov{\tl{3}}} & \red{\ov{\tl{3}}} & \red{\ov{\tl{3}}}&\red{\ov{\tl{3}}} \\
\none & \none & \none & \blue{\ov{\tl{4}}} & \blue{\ov{\tl{3}}} & \blue{\ov{\tl{2}}} & \ov{\tl{2}} & \ov{\tl{2}} & \ov{\tl{2}} & \ov{\tl{2}} \\
\blue{\ov{\tl{3}}} & \blue{\ov{\tl{2}}} & \blue{\ov{\tl{1}}} & \blue{\ov{\tl{1}}}& \blue{\ov{\tl{1}}}& \blue{\ov{\tl{1}}}& \ov{\tl{1}}& \ov{\tl{1}}& \ov{\tl{1}}& \ov{\tl{1}}
\end{ytableau}$}\ , \quad
{\tt Q}(\bc) =
\begin{ytableau}
 \none &  \none &  \none &  \none &  \none &  \none &  \none &  \none & \red{\ov{\tl{5}}} &\red{\ov{\tl{5}}} \\
 \none & \none & \none & \none & \none & \none &  \red{\ov{\tl{5}}} & \red{\ov{\tl{5}}} & 
\red{\ov{\tl{4}}} & \red{\ov{\tl{4}}} \\
 \none & \none & \none & \blue{\ov{\tl{5}}} & \blue{\ov{\tl{5}}} & \blue{\ov{\tl{5}}} &  \none[\cdot] & \none[\cdot] & \none[\cdot] & \none[\cdot] \\
\blue{\ov{\tl{5}}} & \blue{\ov{\tl{5}}} & \blue{\ov{\tl{5}}} & \blue{\ov{\tl{4}}} &  \blue{\ov{\tl{4}}} & \blue{\ov{\tl{4}}} & \none[\cdot] & \none[\cdot] & \none[\cdot] & \none[\cdot]
\end{ytableau}
\end{equation*} \vskip 2mm
\noindent where $U_1$ and $V_1$ (resp. $U_2$ and $V_2$) are given in blue (resp. in red).
Since
\begin{equation*}
V_1^{\nw}=
\ytableausetup{aligntableaux=center, boxsize=1.15em}
\begin{ytableau}
\blue{\ov{\tl{5}}} & \blue{\ov{\tl{5}}} & \blue{\ov{\tl{5}}} & \blue{\ov{\tl{5}}} & \blue{\ov{\tl{5}}} & \blue{\ov{\tl{5}}} \\
\blue{\ov{\tl{4}}} & \blue{\ov{\tl{4}}} & \blue{\ov{\tl{4}}} & \none[\cdot] & \none[\cdot] & \none[\cdot]
\end{ytableau} \ , \quad
V_2^{\nw}=
\begin{ytableau}
\red{\ov{\tl{5}}} & \red{\ov{\tl{5}}} &\red{\ov{\tl{5}}} &\red{\ov{\tl{5}}} \\
\red{\ov{\tl{4}}} & \red{\ov{\tl{4}}} & \none[\cdot] & \none[\cdot]
\end{ytableau}
\end{equation*}
we have by algorithm (\textbf{g}-1)--(\textbf{g}-3), \vskip 2mm
\begin{equation*}
\ytableausetup{aligntableaux=center, boxsize=1.15em}
\Tl{${\tt T}(\bc) = \begin{ytableau}
\none &\none &\none &\none &\none &\none &\none &\none &\none &\none &\none &\red{\ov{\tl{5}}} &\red{\ov{\tl{5}}} & \red{\ov{\tl{4}}} & \red{\ov{\tl{4}}} \\
\none & \none &\none &\none &\none &\none & \none & \none & \none & \none & \none & \red{\ov{\tl{3}}} & \red{\ov{\tl{3}}} & \red{\ov{\tl{3}}}&\red{\ov{\tl{3}}} \\
\red{\ov{\tl{5}}} &\red{\ov{\tl{5}}} &\blue{\ov{\tl{5}}} &\blue{\ov{\tl{5}}} &\blue{\ov{\tl{5}}}  &\blue{\ov{\tl{5}}}  & \blue{\ov{\tl{5}}}  & \blue{\ov{\tl{5}}}  & \blue{\ov{\tl{4}}} & \blue{\ov{\tl{3}}} & \blue{\ov{\tl{2}}} & \ov{\tl{2}} & \ov{\tl{2}} & \ov{\tl{2}} & \ov{\tl{2}} \\
\red{\ov{\tl{4}}} &\red{\ov{\tl{4}}} &\blue{\ov{\tl{4}}} &\blue{\ov{\tl{4}}} &\blue{\ov{\tl{4}}} & \blue{\ov{\tl{3}}} & \blue{\ov{\tl{2}}} & \blue{\ov{\tl{1}}} & \blue{\ov{\tl{1}}}& \blue{\ov{\tl{1}}}& \blue{\ov{\tl{1}}}& \ov{\tl{1}}& \ov{\tl{1}}& \ov{\tl{1}}& \ov{\tl{1}}
\end{ytableau}$}
\end{equation*}}
\end{ex}

\begin{ex} \label{ex:algorithm2}
{\em 
Let us compare $\ov{\tt T}(\bc)$ with ${\tt T}(\bc)$.
Throughout this example, we write the bumping (or recording) letters from $c_{\ov{i+1}\,\ov{i}}$ (resp.~$c_{\ov{i}\,\ov{j}}$ and $c_{\ov{i+1}\,\ov{k}}$) in red (resp.~blue) as shown below. 
Suppose that $n = 9$ and $i = 7$. Let us consider 
\begin{equation*}
{\bf c} = 
\vcenter{
\xymatrixcolsep{0pc}\xymatrixrowsep{0pc}\xymatrix{
& & & & & & & 0 & & & & & & &    \\
& & & & & & 0 & & \blue{0} & & & & & &    \\
& & & & & 0 & & \blue{1} & & \blue{1} & & & & &    \\
& & & & 0 & & \blue{0} & & \blue{1} & & 0 & & & &    \\
& & & 0 & & \blue{0} & & \blue{0} & & 0 & & 0 & & & \\
& & 0 & & \blue{1} & & \blue{0} & & 0 & & 1 & & 0 & & \\
& 0 & & \blue{1} & & \blue{1} & & 0 & & 0 & & 1 & & 1 & \\  
0 & & \red{4} & & \blue{2} & & 2 & & 1 & & 2 & & 1 & & 2}}
\end{equation*}
Then we have
\ytableausetup{aligntableaux=center, boxsize=1.15em}
\begin{equation} \label{eq: PQ circ vs PQ}
\begin{split}
	{\tt P}(\bc^\circ) = 
	\begin{ytableau}
		\none & \none & \none & \none & \none & \none & \ov{\tl{6}} & \ov{\tl{6}} \\
		\none & \none & \none & \none & \none & \blue{\ov{\tl{6}}} & \ov{\tl{5}} & \ov{\tl{5}} \\
		\none & \none & \none & \none & \blue{\ov{\tl{6}}} & \ov{\tl{5}} & \ov{\tl{4}} & \ov{\tl{4}} \\
		\none & \none & \blue{\ov{\tl{6}}} & \blue{\ov{\tl{5}}} & \blue{\ov{\tl{5}}} & \ov{\tl{4}} & \ov{\tl{3}} & \ov{\tl{3}} \\
		\none & \blue{\ov{\tl{5}}} & \ov{\tl{4}} & \ov{\tl{3}} & \ov{\tl{2}} & \ov{\tl{2}} & \ov{\tl{2}} & \ov{\tl{2}} \\
		\blue{\ov{\tl{3}}} & \blue{\ov{\tl{2}}} & \ov{\tl{2}} & \ov{\tl{2}} & \ov{\tl{1}} & \ov{\tl{1}} & \ov{\tl{1}} & \ov{\tl{1}} \\
		\none & \none & \none & \none & \none & \none & \none & \none \\
	\end{ytableau}
	\,\qquad\,
	&
	{\tt Q}(\bc^\circ) = 
	\begin{ytableau}
		\none & \none & \none & \none & \none & \none & \none[\cdot] & \none[\cdot] \\
		\none & \none & \none & \none & \none & \blue{\ov{\tl{7}}} & \none[\cdot] & \none[\cdot]\\
		\none & \none & \none & \none & \blue{\ov{\tl{8}}} & \none[\cdot] & \none[\cdot] & \none[\cdot] \\
		\none & \none & \blue{\ov{\tl{8}}} & \blue{\ov{\tl{7}}} & \blue{\ov{\tl{7}}} & \none[\cdot] & \none[\cdot] & \none[\cdot] \\
		\none & \blue{\ov{\tl{8}}} & \none[\cdot] & \none[\cdot] & \none[\cdot] & \none[\cdot] & \none[\cdot] & \none[\cdot] \\
		\blue{\ov{\tl{7}}} & \blue{\ov{\tl{7}}} & \none[\cdot] & \none[\cdot] & \none[\cdot] & \none[\cdot] & \none[\cdot] & \none[\cdot] \\
		\none & \none & \none & \none & \none & \none & \none & \none \\
	\end{ytableau}
	\\
		{\tt P}(\bc) = 
	\begin{ytableau}
		\none & \none & \none & \none & \none & \none & \red{\ov{\tl{7}}} & \red{\ov{\tl{7}}} \\
		\none & \none & \none & \none & \none & \red{\ov{\tl{7}}} & \ov{\tl{6}} & \ov{\tl{6}} \\
		\none & \none & \none & \none & \none & \blue{\ov{\tl{6}}} & \ov{\tl{5}} & \ov{\tl{5}} \\
		\none & \none & \none & \red{\ov{\tl{7}}} & \blue{\ov{\tl{6}}} & \ov{\tl{5}} & \ov{\tl{4}} & \ov{\tl{4}} \\
		\none & \none & \blue{\ov{\tl{6}}} & \blue{\ov{\tl{5}}} & \blue{\ov{\tl{5}}} & \ov{\tl{4}} & \ov{\tl{3}} & \ov{\tl{3}} \\
		\none & \blue{\ov{\tl{5}}} & \ov{\tl{4}} & \ov{\tl{3}} & \ov{\tl{2}} & \ov{\tl{2}} & \ov{\tl{2}} & \ov{\tl{2}} \\
		\blue{\ov{\tl{3}}} & \blue{\ov{\tl{2}}} & \ov{\tl{2}} & \ov{\tl{2}} & \ov{\tl{1}} & \ov{\tl{1}} & \ov{\tl{1}} & \ov{\tl{1}} \\
	\end{ytableau}
	\,\qquad\,
	&
	{\tt Q}(\bc) = 
	\begin{ytableau}
		\none & \none & \none & \none & \none & \none & \red{\ov{\tl{8}}} & \red{\ov{\tl{8}}} \\
		\none & \none & \none & \none & \none & \red{\ov{\tl{8}}} & \none[\cdot] & \none[\cdot] \\
		\none & \none & \none & \none & \none & \blue{\ov{\tl{7}}} & \none[\cdot] & \none[\cdot]\\
		\none & \none & \none & \red{\ov{\tl{8}}} & \blue{\ov{\tl{8}}} & \none[\cdot] & \none[\cdot] & \none[\cdot] \\
		\none & \none & \blue{\ov{\tl{8}}} & \blue{\ov{\tl{7}}} & \blue{\ov{\tl{7}}} & \none[\cdot] & \none[\cdot] & \none[\cdot] \\
		\none & \blue{\ov{\tl{8}}} & \none[\cdot] & \none[\cdot] & \none[\cdot] & \none[\cdot] & \none[\cdot] & \none[\cdot] \\
		\blue{\ov{\tl{7}}} & \blue{\ov{\tl{7}}} & \none[\cdot] & \none[\cdot] & \none[\cdot] & \none[\cdot] & \none[\cdot] & \none[\cdot] \\
	\end{ytableau}
\end{split}
\end{equation}
where $\lambda = (8, 7, 6, 4, 3, 2)$ with $m = 4$, and $\mu = (6, 6, 3, 3, 2, 2)$. 
By definition, $U_l$ and $V_l^{\nw}$ $(1 \le l \le 4)$ associated with $({\tt P}(\bc^\circ), {\tt Q}(\bc^\circ))$ are given by
\begin{equation*}
\begin{split}
	U_4 = \emptyset \,\qquad\, & \,\,\,	U_3 = 
	\begin{ytableau}
		\blue{\ov{\tl{6}}}
	\end{ytableau}
	\,\qquad\,\,\,\,
	U_2 = 
	\begin{ytableau}
		\none & \none & \blue{\ov{\tl{6}}} \\
		\blue{\ov{\tl{6}}} & \blue{\ov{\tl{5}}} & \blue{\ov{\tl{5}}} \\
	\end{ytableau}
	\qquad\quad\!
	U_1 = 
	\begin{ytableau}
		\none & \blue{\ov{\tl{5}}} \\
		\blue{\ov{\tl{2}}} & \blue{\ov{\tl{3}}} \\
	\end{ytableau}
	\\
	V_4^{\nw} = \emptyset \,\qquad\, 
	& V_3^{\nw} = 
	\begin{ytableau}
		\blue{\ov{\tl{7}}} \\
		\none[\cdot]
	\end{ytableau}
	\,\qquad\,\,
	V_2^{\nw} = 
	\begin{ytableau}
		\blue{\ov{\tl{8}}} & \blue{\ov{\tl{8}}} & \blue{\ov{\tl{7}}} \\
		\blue{\ov{\tl{7}}} & \none[\cdot] & \none[\cdot] \\
	\end{ytableau}
	\,\qquad\,\,
	V_1^{\nw} = 
	\begin{ytableau}
		\blue{\ov{\tl{8}}} & \blue{\ov{\tl{7}}} \\
		\blue{\ov{\tl{7}}} & \none[\cdot] \\
	\end{ytableau}
\end{split}
\end{equation*}
Now we apply the algorithm ($\ov{\text{\textbf{g}}}$-1)--($\ov{\text{\textbf{g}}}$-3) to obtain $\ov{\tt T}(\bc)$ as follows.
\begin{equation} \label{eq: ov tt T}
\begin{split}
	& \ov{\tt T}(\bc)_4 = 
	\begin{ytableau}
		\red{\ov{\tl{8}}} & \red{\ov{\tl{8}}} & \red{\ov{\tl{8}}} & \red{\ov{\tl{8}}} \\
		\red{\ov{\tl{7}}} & \red{\ov{\tl{7}}} & \red{\ov{\tl{7}}} & \red{\ov{\tl{7}}} \\
	\end{ytableau}
	\,\qquad\,\qquad\,\qquad\qquad\,\,
	\ov{\tt T}(\bc)_3 = 
	\begin{ytableau}
		\none & \none & \none & \red{\ov{\tl{8}}} & \red{\ov{\tl{8}}} \\
		\none & \none & \none & \red{\ov{\tl{7}}} & \red{\ov{\tl{7}}} \\
		\red{\ov{\tl{8}}} & \red{\ov{\tl{8}}} & \blue{\ov{\tl{7}}} & \ov{\tl{6}} & \ov{\tl{6}} \\
		\red{\ov{\tl{7}}} & \red{\ov{\tl{7}}} & \blue{\ov{\tl{6}}} & \ov{\tl{5}} & \ov{\tl{5}} \\
	\end{ytableau}
	\\
	& \ov{\tt T}(\bc)_2 = 
		\begin{ytableau}
						 \none & \none & \none & \none & \none & \none & \none & \none & \none & \none & \none & \none \\
				 \none & \none & \none & \none & \none & \none & \none & \none & \none & \none & \none & \none \\
		\none & \none & \none & \none & \none & \none & \none & \red{\ov{\tl{8}}} & \red{\ov{\tl{8}}} \\
		\none & \none & \none & \none & \none & \none & \none & \red{\ov{\tl{7}}} & \red{\ov{\tl{7}}} \\
		\none & \none & \none & \none & \none & \none & \blue{\ov{\tl{7}}} & \ov{\tl{6}} & \ov{\tl{6}} \\
		\none & \none & \none & \none & \none & \none & \blue{\ov{\tl{6}}} & \ov{\tl{5}} & \ov{\tl{5}} \\
		\red{\ov{\tl{8}}} & \red{\ov{\tl{8}}} & \blue{\ov{\tl{8}}} & \blue{\ov{\tl{8}}} & \blue{\ov{\tl{7}}} & \blue{\ov{\tl{6}}} & {\ov{\tl{5}}} & \ov{\tl{4}} & \ov{\tl{4}} \\
		\red{\ov{\tl{7}}} & \red{\ov{\tl{7}}} & \blue{\ov{\tl{7}}} & \blue{\ov{\tl{6}}} & \blue{\ov{\tl{5}}} & \blue{\ov{\tl{5}}} & {\ov{\tl{4}}} & \ov{\tl{3}} & \ov{\tl{3}} \\
	\end{ytableau}
	\!\!\!\!
	\ov{\tt T}(\bc)_1 = 
		\begin{ytableau}
		 \none & \none & \none & \none & \none & \none & \none & \none & \none & \none & \red{\ov{\tl{8}}} & \red{\ov{\tl{8}}} \\
		 \none & \none & \none & \none & \none & \none & \none & \none & \none & \none & \red{\ov{\tl{7}}} & \red{\ov{\tl{7}}} \\
		 \none & \none & \none & \none & \none & \none & \none & \none & \none & \blue{\ov{\tl{7}}} & \ov{\tl{6}} & \ov{\tl{6}} \\
		 \none & \none & \none & \none & \none & \none & \none & \none & \none & \blue{\ov{\tl{6}}} & \ov{\tl{5}} & \ov{\tl{5}} \\
		 \none & \none & \none & \none & \none & \none & \blue{\ov{\tl{8}}} & \blue{\ov{\tl{7}}} & \blue{\ov{\tl{6}}} & {\ov{\tl{5}}} & \ov{\tl{4}} & \ov{\tl{4}} \\
		 \none & \none & \none & \none & \none & \none & \blue{\ov{\tl{6}}} & \blue{\ov{\tl{5}}} & \blue{\ov{\tl{5}}} & {\ov{\tl{4}}} & \ov{\tl{3}} & \ov{\tl{3}} \\
		 \red{\ov{\tl{8}}} & \red{\ov{\tl{8}}} & \blue{\ov{\tl{8}}} & \blue{\ov{\tl{8}}}& \blue{\ov{\tl{7}}} & \blue{\ov{\tl{5}}} & \ov{\tl{4}} & \ov{\tl{3}} & \ov{\tl{2}} & \ov{\tl{2}}  & \ov{\tl{2}} & \ov{\tl{2}} \\
		\red{\ov{\tl{7}}} & \red{\ov{\tl{7}}}& \blue{\ov{\tl{7}}} & \blue{\ov{\tl{7}}} & \blue{\ov{\tl{3}}} & \blue{\ov{\tl{2}}} & \ov{\tl{2}} & \ov{\tl{2}} & \ov{\tl{1}} & \ov{\tl{1}} & \ov{\tl{1}}& \ov{\tl{1}} \\
	\end{ytableau}
\end{split}
\end{equation}
On the other hand, we have by the algorithm (${\text{\textbf{g}}}$-1)--(${\text{\textbf{g}}}$-3)
\begin{equation} \label{eq: tt T}
\begin{split}
	& {\tt T}(\bc)_4 = 
	\begin{ytableau}
		 \red{\ov{\tl{8}}} & \red{\ov{\tl{8}}} \\
		 \red{\ov{\tl{7}}} & \red{\ov{\tl{7}}} \\
	\end{ytableau}
	\,\qquad\,\qquad\,\qquad\qquad\qquad\,
	{\tt T}(\bc)_3 = 
	\begin{ytableau}
		 \none & \none & \red{\ov{\tl{8}}} & \red{\ov{\tl{8}}} \\
		 \none & \none & \red{\ov{\tl{7}}} & \red{\ov{\tl{7}}} \\
		 \red{\ov{\tl{8}}} & \red{\ov{\tl{7}}} & \ov{\tl{6}} & \ov{\tl{6}} \\
		 \blue{\ov{\tl{7}}} & \blue{\ov{\tl{6}}} & \ov{\tl{5}} & \ov{\tl{5}} \\
	\end{ytableau}
	\\
	& {\tt T}(\bc)_2 = 
		\begin{ytableau}
		 \none & \none & \none & \none & \none & \none & \none & \none & \none & \none \\
		 \none & \none & \none & \none & \none & \none & \none & \none & \none & \none \\
		\none & \none & \none & \none & \none & \none & \none & \red{\ov{\tl{8}}} & \red{\ov{\tl{8}}} \\
		\none & \none & \none & \none & \none & \none & \none & \red{\ov{\tl{7}}} & \red{\ov{\tl{7}}} \\
		\none & \none & \none & \none & \none & \none & \red{\ov{\tl{7}}} & \ov{\tl{6}} & \ov{\tl{6}} \\
		\none & \none & \none & \none & \none & \none & \blue{\ov{\tl{6}}} & \ov{\tl{5}} & \ov{\tl{5}} \\
		\red{\ov{\tl{8}}} & \blue{\ov{\tl{8}}} & \red{\ov{\tl{8}}} & \blue{\ov{\tl{8}}} & \red{\ov{\tl{7}}} & \blue{\ov{\tl{6}}} & {\ov{\tl{5}}} & \ov{\tl{4}} & \ov{\tl{4}} \\
		\blue{\ov{\tl{7}}} & \blue{\ov{\tl{7}}} & \blue{\ov{\tl{7}}} & \blue{\ov{\tl{6}}} & \blue{\ov{\tl{5}}} & \blue{\ov{\tl{5}}} & {\ov{\tl{4}}} & \ov{\tl{3}} & \ov{\tl{3}} \\
	\end{ytableau}
	\,\quad\,
	{\tt T}(\bc)_1 = 
		\begin{ytableau}
		 \none & \none & \none & \none & \none & \none & \none & \none & \none & \none & \red{\ov{\tl{8}}} & \red{\ov{\tl{8}}} \\
		 \none & \none & \none & \none & \none & \none & \none & \none & \none & \none & \red{\ov{\tl{7}}} & \red{\ov{\tl{7}}} \\
		 \none & \none & \none & \none & \none & \none & \none & \none & \none & \red{\ov{\tl{7}}} & \ov{\tl{6}} & \ov{\tl{6}} \\
		 \none & \none & \none & \none & \none & \none & \none & \none & \none & \blue{\ov{\tl{6}}} & \ov{\tl{5}} & \ov{\tl{5}} \\
		 \none & \none & \none & \none & \none & \none & \blue{\ov{\tl{8}}} & \red{\ov{\tl{7}}} & \blue{\ov{\tl{6}}} & {\ov{\tl{5}}} & \ov{\tl{4}} & \ov{\tl{4}} \\
		 \none & \none & \none & \none & \none & \none & \blue{\ov{\tl{6}}} & \blue{\ov{\tl{5}}} & \blue{\ov{\tl{5}}} & {\ov{\tl{4}}} & \ov{\tl{3}} & \ov{\tl{3}} \\
		 \red{\ov{\tl{8}}} & \blue{\ov{\tl{8}}} & \red{\ov{\tl{8}}} & \blue{\ov{\tl{8}}}& \blue{\ov{\tl{7}}} & \blue{\ov{\tl{5}}} & \ov{\tl{4}} & \ov{\tl{3}} & \ov{\tl{2}} & \ov{\tl{2}}  & \ov{\tl{2}} & \ov{\tl{2}} \\
		\blue{\ov{\tl{7}}} & \blue{\ov{\tl{7}}}& \blue{\ov{\tl{7}}} & \blue{\ov{\tl{7}}} & \blue{\ov{\tl{3}}} & \blue{\ov{\tl{2}}} & \ov{\tl{2}} & \ov{\tl{2}} & \ov{\tl{1}} & \ov{\tl{1}} & \ov{\tl{1}}& \ov{\tl{1}} \\
	\end{ytableau}
\end{split}
\end{equation}
Thus we have $\ov{\tt T}(\bc) = {\tt T}(\bc) = \kappa^{\searrow}({\bf c})$. Note that $\kappa^{\searrow}({\bf c}^\circ)$ is obtained from $\ov{\tt T}(\bc)$ by removing the red dominos
\raisebox{-.6ex}{{\tiny ${\def\lr#1{\multicolumn{1}{|@{\hspace{.6ex}}c@{\hspace{.6ex}}|}{\raisebox{-.3ex}{$#1$}}}\raisebox{-.6ex}
{$\begin{array}[b]{c}
\cline{1-1}
\lr{ \!\ov{8}\!}\\
\cline{1-1}
\lr{ \ov{7}}\\
\cline{1-1}
\end{array}$}}$}}.
}
\end{ex}

\begin{lem} \label{lem:Delta_k}
Under the above hypothesis, we have
\begin{itemize}
\item[(1)] ${\tt T}(\bc) = \kappa^{\se}({\bf c})$,

\item[(2)] $\ov{\tt T}(\bc) = {\tt T}(\bc)$,

\item[(3)] $\kappa^{\se}(\tf_i{\bc}) = \tf_i\kappa^{\se}({\bf c})$.
\end{itemize}
\end{lem}

\pf (1) 
Let $C_{i+1}=\sum_{1\leq j\leq i}c_{\ov{i+1}\ov{j}}$. 
We use induction on $C_{i+1}$.  Note that when $C_{i+1} = 0$, we clearly have ${\tt T}(\bc) = \kappa^{\se}({\bf c})$ by definition of ${\tt T}(\bc)$.

First, assume that $C_{i+1} = 1$. Then $c_{\ov{i+1}\ov{j}} =1$ for some $j$. 
Suppose that the box in ${\tt P}(\bc)$, which appears after insertion of the corresponding $\ov{j}$, belongs to $U_{l}$ for some $1\leq l \leq m$.
Recall that $\mu^\pi={\rm sh}(T)$. 
Let $d=\mu_{2l-3}-\mu_{2l-1}$ and 
let $u$ be the length of the bottom row of $U_l$.  

{\em Case 1}. 
Suppose that $\ell({\rm sh}(U_l))=2$ and $d > u$. Then we have \vskip 2mm
\begin{center}
$U_l=$
\ytableausetup {mathmode, boxsize=1.25em} 
			\begin{ytableau}
				\none & \none & \none & \none & *(gray) \empty \\
				\empty & \empty & \scriptstyle \cdots & \empty & \empty
			\end{ytableau}\ , \quad 
$W_l=$
\ytableausetup {mathmode, boxsize=1.25em} 
\begin{ytableau}
\scriptstyle \ov{i\!+\! 1} & \scriptstyle\ov{i} & \scriptstyle\ov{i} & \scriptstyle \cdots & \scriptstyle\ov{i} & *(gray) \empty \\
				\scriptstyle\ov{i} & \empty & \empty & \scriptstyle \cdots & \empty & \empty
			\end{ytableau}\ .

\end{center}\vskip 2mm
where the gray box denotes the one created after the insertion of $\ov{j}$.
In this case, the domino in the leftmost column of $W_l$ does not move to lower rows. Hence it is clear that ${\tt T}(\bc)$ coincides with $\kappa^{\se}({\bf c})$.

{\em Case 2}.
Suppose that $\ell({\rm sh}(U_l))=2$ and $d = u$. Then we have\vskip 2mm
\begin{center}
$U_l=$
\ytableausetup {mathmode, boxsize=1.25em} 
			\begin{ytableau}
				\none & \none & \none & \none & *(gray) \empty \\
				\empty & \empty & \scriptstyle \cdots & \empty & \empty
			\end{ytableau} \ ,\quad 
$W_l=$
\ytableausetup {mathmode, boxsize=1.25em} 
\begin{ytableau}
\scriptstyle \ov{i\!+\!1} & \scriptstyle\ov{i} & \scriptstyle\ov{i} & \scriptstyle \cdots & \scriptstyle\ov{i} & *(gray) \empty \\
				\scriptstyle\ov{i} & \empty & \empty & \scriptstyle \cdots & \empty & \empty
			\end{ytableau}\ .

\end{center}\vskip 2mm
In this case, the leftmost domino in $W_l$ moves down to a lower row by ({\bf g}-2), and it is easy to see that ${\tt T}(\bc)=\kappa^{\se}({\bf c})$.

{\em Case 3}. Finally suppose that $\ell({\rm sh}(U_l))=1$. Then \vskip 2mm
\begin{center}
$U_l =$
\ytableausetup {mathmode, boxsize=1.25em} 
			\begin{ytableau}
				\none & \none & \none & \none & \none & \none \\
				*(gray)  \empty & \empty & \empty & \scriptstyle \cdots & \empty & \empty
			\end{ytableau} \ ,\quad
$W_l =$
\ytableausetup {mathmode, boxsize=1.25em} 
			\begin{ytableau}
				\scriptstyle\ov{i\!+\!1} & \scriptstyle\ov{i} & \scriptstyle\ov{i} & \scriptstyle\cdots & \scriptstyle\ov{i} & \scriptstyle\ov{i} \\
				*(gray) \scriptstyle\ov{i} \empty & \empty & \empty & \scriptstyle \cdots & \empty & \empty
			\end{ytableau}\ .
\end{center}\vskip 2mm
As in Case 1, the domino in the leftmost column of $W_l$ does not move to lower rows, and hence ${\tt T}(\bc)=\kappa^{\se}({\bf c})$.

Next, we assume that $C_{i+1} > 1$. 
Let $(\underline{\ba}',\underline{\bb}')$ be the biword removing $(a_1,b_1)$ in $(\ba',\bb')$ in \eqref{eq:biword(a,b)}, and let $\underline{\bc}=\bc(\underline{\ba}',\underline{\bb}')$. 
Note that $(a_1,b_1)=(\ov{i+1},\ov{j})$ for some $j$.

By induction hypothesis, we have ${\tt T}(\underline{\bc})=\kappa^{\se}(\underline{\bc})$. 
On the other hand, when we apply the insertion of $(\ov{i+1},\ov{j})$ into ${\tt T}(\underline{\bc})$, the possible cases are given similarly as above. Then it is straightforward to check that the tableau obtained by insertion of $(\ov{i+1},\ov{j})$ into 
${\tt T}(\underline{\bc})$ (see the step (2) in the definition of $P^{\se}(\bc)=P^{\se}(\ba,\bb)$ \eqref{eq:kappa_se}) is equal to ${\tt T}({\bc})$. 
Therefore, we have ${\tt T}(\bc)=\kappa^{\se}(\bc)$. This completes the induction.
\smallskip

(2) Since the letters $\ov{i}$'s from $c_{\ov{i+1}\,\ov{i}}$ are the smallest insertion letters in the algorithm for ${\tt T}$ (not $\ov{\tt T}$), the 	
\raisebox{0.1ex}{{\scriptsize ${\def\lr#1{\multicolumn{1}{|@{\hspace{.6ex}}c@{\hspace{.6ex}}|}{\raisebox{-.3ex}{$#1$}}}\raisebox{-.6ex}
{$\begin{array}[b]{c}
\cline{1-1}
\lr{ \ov{i}}\\
\cline{1-1}
\end{array}$}}$}}'s
induced from $c_{\ov{i+1}\,\ov{i}}$ are arranged from the northeastern most and then moved to the left or down to the next rows if the insertion letters $\ov{k}$'s from $c_{\ov{i+1}\,\ov{k}}$ bump out 
\raisebox{0.1ex}{{\scriptsize ${\def\lr#1{\multicolumn{1}{|@{\hspace{.6ex}}c@{\hspace{.6ex}}|}{\raisebox{-.3ex}{$#1$}}}\raisebox{-.6ex}
{$\begin{array}[b]{c}
\cline{1-1}
\lr{ \ov{i}}\\
\cline{1-1}
\end{array}$}}$}}'s, where $k < i$. Then the recording letters 
\raisebox{0.1ex}{{\scriptsize ${\def\lr#1{\multicolumn{1}{|@{\hspace{.6ex}}c@{\hspace{.6ex}}|}{\raisebox{-.3ex}{$#1$}}}\raisebox{-.6ex}
{$\begin{array}[b]{c}
\cline{1-1}
\lr{ \!\ov{i\!+\!1}\!}\\
\cline{1-1}
\end{array}$}}$}}'s corresponding to the %
\raisebox{0.1ex}{{\scriptsize ${\def\lr#1{\multicolumn{1}{|@{\hspace{.6ex}}c@{\hspace{.6ex}}|}{\raisebox{-.3ex}{$#1$}}}\raisebox{-.6ex}
{$\begin{array}[b]{c}
\cline{1-1}
\lr{ \ov{i}}\\
\cline{1-1}
\end{array}$}}$}}'s
should yield the dominos 
\raisebox{-.6ex}{{\tiny ${\def\lr#1{\multicolumn{1}{|@{\hspace{.6ex}}c@{\hspace{.6ex}}|}{\raisebox{-.3ex}{$#1$}}}\raisebox{-.6ex}
{$\begin{array}[b]{c}
\cline{1-1}
\lr{ \!\ov{i\!+\!1}\!}\\
\cline{1-1}
\lr{ \ov{i}}\\
\cline{1-1}
\end{array}$}}$}}
in each ${\tt T}(\bc)_k$ (cf.~\eqref{eq: PQ circ vs PQ}).
Now we may identify the dominos \raisebox{-.6ex}{{\tiny ${\def\lr#1{\multicolumn{1}{|@{\hspace{.6ex}}c@{\hspace{.6ex}}|}{\raisebox{-.3ex}{$#1$}}}\raisebox{-.6ex}
{$\begin{array}[b]{c}
\cline{1-1}
\lr{ \!\ov{i\!+\!1}\!}\\
\cline{1-1}
\lr{ \ov{i}}\\
\cline{1-1}
\end{array}$}}$}}
in ${\tt T}(\bc)_k$ arising from $c_{\ov{i+1}\,\ov{i}}$ with ones in $\ov{\tt T}(\bc)_k$ at the same positions. In fact, there exists $l \in \Z_{\ge 1}$ such that ${\tt T}(\bc)_k$ is a subtableau of $\ov{\tt T}(\bc)_k$ for $k > l$, and ${\tt T}(\bc)_k = \ov{\tt T}(\bc)_k$ for $k \le l$ (cf.~\eqref{eq: ov tt T} and \eqref{eq: tt T}).
Thus we have $\ov{\tt T}(\bc) = {\tt T}(\bc)$.
\smallskip
 
(3) 
By definition of $\ov{\tt T}(\bc^\circ)$ and Lemma \ref{lem:Q commutes with f_k}, we have $\ov{\tt T}(\tf_i \bc^\circ) = \tf_i \ov{\tt T}(\bc^\circ)$. Thus $\kappa^{\searrow}(\tf_i \bc^\circ) = \tf_i \kappa^{\searrow}(\bc^\circ)$ by (1) and (2). 
By definition of $\ov{\tt T}(\bc)$, (1) and (2), we have
\begin{equation} \label{eq:2}
	\kappa^{\searrow}({\bf c}) = \mathsf{D}  \,\sqcup\, \kappa^{\searrow}({\bf c}^\circ),
\end{equation}
where $\mathsf{D}$ is a skew-subtableau of $\kappa^{\searrow}({\bf c})$ consisting of $c_{\ov{i+1}\,\ov{i}}$ dominos\,
\raisebox{-.6ex}{{\tiny ${\def\lr#1{\multicolumn{1}{|@{\hspace{.6ex}}c@{\hspace{.6ex}}|}{\raisebox{-.3ex}{$#1$}}}\raisebox{-.6ex}
{$\begin{array}[b]{c}
\cline{1-1}
\lr{ \!\ov{i\!+\!1}\!}\\
\cline{1-1}
\lr{ \ov{i}}\\
\cline{1-1}
\end{array}$}}$}} (cf.~\eqref{eq: ov tt T}). 
Recall that since $\tf_i$ acts on $(c_{\ov{i+1}\,\ov{i}})$ trivially, we have 
\begin{equation} \label{eq: action of tfi}
	\tf_i {\bf c} = (c_{\ov{i+1}\,\ov{i}}) \otimes ( \tf_i {\bf c}^\circ ).
\end{equation}
Suppose $\tf_i \bc \neq {\bf 0}$. It follows from Lemma \ref{lem:Q commutes with f_k} and \eqref{eq: action of tfi} that $\mathsf{D}$ is invariant under $\tf_i$, since it depends only on $c_{\ov{i+1}\,\ov{i}}$ and the number of columns of $\ov{\tt T}(\bc^\circ)_k$'s.
Therefore we have
\begin{equation*}
\kappa^{\searrow}(\tf_i \bc) = \mathsf{D}  \,\sqcup\, \kappa^{\searrow}(\tf_i {\bf c}^\circ) = \tf_i \left( \mathsf{D}  \,\sqcup\, \kappa^{\searrow}( {\bf c}^\circ) \right)  = \tf_i \kappa^{\searrow}( \bc ).
\end{equation*}
If $\tf_i \bc = {\bf 0}$, then $\tf_i \bc^\circ = {\bf 0}$ by \eqref{eq: action of tfi}, and $\tf_i \kappa^{\searrow}( {\bf c}^\circ) = {\bf 0}$. Thus $\tf_i \kappa^{\searrow}( \bc ) = {\bf 0}$ by \eqref{eq:2}. This completes the proof of (3).
\qed

\begin{proof}[Proof of Theorem \ref{thm:isomorphism theorem}]
It suffices prove that for $i\in I$ and $\bc\in \B^J$
\begin{equation*}\label{eq:kappa commutes with f_k}
\kappa^{\diamond}(\tf_i{\bc}) = 
\tf_i\kappa^{\diamond}({\bf c}) \quad (\diamond = {\scriptstyle \searrow},\, \scriptstyle{\nwarrow})
\end{equation*} 
We prove only the case when $\diamond={\scriptstyle \searrow}$ since the proof for the other case is identical.

Suppose first that $i\in I\setminus\{n\}$.
By Proposition \ref{prop:signature for type D} ($\sigma_{k,3}(\bc)$ is trivial in this case), 
we have
\begin{equation*}
{\bf c} =  \textbf{c}_{\Delta_{i+1}^c} \otimes \textbf{c}_{\Delta_{i+1}},
\end{equation*} 
as elements of $\gl_2$-crystals with respect to $\te_i, \tf_i$.

Let us denote by $\overset{B}{\longrightarrow}$ the insertion of a biword into a tableau following the algorithm given in \eqref{eq:kappa_se}. 
If we ignore the entries smaller than $\ov{i+1}$, then 
$\bc_{\Delta_{i+1}^c} \overset{B}{\longrightarrow} \kappa^{\se}(\bc_{\Delta_{i+1}})$ is equal to a usual Schensted's column insertion. Hence 
\begin{equation}\label{eq:usual Knuth}
\left(\bc_{\Delta_{i+1}^c} \overset{B}{\longrightarrow} \kappa^{\se}(\bc_{\Delta_{i+1}})\right)
= \bc_{\Delta_{i+1}^c} \otimes \kappa^{\se}(\bc_{\Delta_{i+1}}),
\end{equation}
as elements of $\gl_{i+1}$-crystals with respect to $\te_j,\tf_j$ for $1\leq j\leq i$. 
Moreover, the subtableau of $\kappa^{\se}(\bc)$ consisting of entries $\ov{n},\ldots,\ov{i+2}$ is invariant under the action of $\tf_j$ on $\kappa^{\se}(\bc)$ for $1\leq j\leq i$ since it depends only on the Knuth equivalence class of the subtableau with entries $\ov{i+1},\ldots,\ov{1}$ by definition of $\kappa^{\se}$.
 
{\em Case 1}. Suppose that 
$\tilde{f}_i {\bc} = \bc_{\Delta_{i+1}^c} \otimes \tf_i\bc_{\Delta_{i+1}}$.
Then we have
\begin{equation}\label{eq:kappa_se commutes with f_k}
\begin{split}
\kappa^{\se}(\tilde{f}_i{\bf c}) 
&= \left( \bc_{\Delta_{i+1}^c} \overset{B}{\longrightarrow} \kappa^{\se}(\tilde{f}_i{\bf c}_{\Delta_{i+1}}) \right) \\
&= \left( \bc_{\Delta_{i+1}^c} \overset{B}{\longrightarrow} \tilde{f}_i\kappa^{\se}(\bc_{\Delta_{i+1}}) \right) \quad \text{by Lemma \ref{lem:Delta_k}(3)} \\
&=  \tilde{f}_i\left( \bc_{\Delta_{i+1}^c} \overset{B}{\longrightarrow} \kappa^{\se}(\bc_{\Delta_{i+1}}) \right) \quad \text{by \eqref{eq:usual Knuth}} \\
&= \tilde{f}_i \kappa^{\se}({\bf c}).
\end{split}
\end{equation}  

{\em Case 2}. Suppose that 
$\tilde{f}_i {\bc} = \tf_i\bc_{\Delta_{i+1}^c} \otimes \bc_{\Delta_{i+1}}$. 
By the same argument as in \eqref{eq:kappa_se commutes with f_k}, we have 
$\kappa^{\se}(\tf_i{\bc})=\tf_i\kappa^{\se}({\bc})$. 


Next, suppose that $i=n$. 
We may identify $\bc$ with the pair $({\bc}_{\Delta^c_{n-1}}, {\bc}_{\Delta_{n-1}})$.
Regarding $\bc$ as an element of $\B^J_{\Delta_{n+1}}$, the crystal of type $D_{n+1}$,  
we have by Lemma \ref{lem:Delta_k}(1) the following commuting diagram;
\begin{center}
\begin{tikzpicture}
  \matrix (m) [matrix of math nodes,row sep=2em,column sep=3em,minimum width=1em]
  {
     {\bf c} & ({\bc}_{\Delta^c_{n-1}}, {\bc}_{\Delta_{n-1}}) & ({\bc}_{\Delta^c_{n-1}}, \kappa^{\se}({\bc}_{\Delta_{n-1}}) ) \\
    \kappa^{\se}(\bc)={\tt T}({\bc}) & \nonumber & ({\tt P}({\bc}), {\tt Q}({\bc})) \\};
  \path[-stealth]
    (m-1-1) edge [|->] (m-1-2)
    (m-1-1) edge [|->](m-2-1)
    (m-1-2) edge [|->](m-1-3)
    (m-1-3) edge [|->] node [right] {(i)} (m-2-3)
    (m-2-3) edge [|->] node [below] {(ii)} (m-2-1);
\end{tikzpicture}
\end{center}
Now, we can apply the same argument for the proof of \cite[Theorem 3.6]{K09} to see that the composition of (i) and (ii) commutes with $\tilde{f}_n$. Therefore, we have $\kappa^{\se}(\tf_n{\bc})=\tf_n\kappa^{\se}({\bc})$.
\end{proof}







\end{document}